\documentclass[a4paper,reqno]{amsart} 
\usepackage{cmap} 
\usepackage[T1]{fontenc}
\usepackage{amssymb,amsthm,bm}
\usepackage{mathtools}
\usepackage{nicefrac}
  \mathtoolsset{mathic}
  \allowdisplaybreaks
\usepackage{textcomp}
\usepackage{fouridx}
\usepackage{courier,mathptmx}
\usepackage[spacing]{microtype}
\usepackage{longtable}
\usepackage[numbers,sort&compress]{natbib} 
\usepackage{hyperref,url}
  \hypersetup{  
    bookmarksnumbered
  }
  \hypersetup{  
    breaklinks=false,
    pdfborderstyle={/S/U/W 0},  
    citebordercolor=.235 .702 .443,
    urlbordercolor=.255 .412 .882,
    linkbordercolor=.804 .149 .19,
  }
  \hypersetup{ 
    pdfauthor={Gert de Cooman, Jasper De Bock and Stavros Lopatatzidis},
    pdftitle={A study of imprecise Markov chains: joint lower expectations and point-wise ergodic theorems},
    pdfkeywords={Imprecise probabilities, lower expectation, point-wise ergodic theorem, imprecise Markov chain, Markov property, stationarity, law of iterated expectation, game-theoretic probability}
  }
\usepackage[all]{hypcap} 
\usepackage{xcolor}
\usepackage{tikz}
  \usetikzlibrary{calc,matrix,plotmarks}
  \tikzstyle{root}=[rectangle,draw=blue!90]
  \tikzstyle{nonterminal}=[rectangle,rounded corners,fill=blue!15,draw=blue!15]
  \tikzstyle{terminal}=[rectangle]
  \tikzstyle{cut}=[thick,dotted,draw=green!50!black]
  \tikzstyle{local}=[color=green!50!black,text=green!25!black]

\usepackage{enumitem}
\usepackage{natbib}

\DeclareMathOperator{\trans}{T}
\DeclareMathOperator{\atrans}{S}

\DeclarePairedDelimiter{\group}{(}{)}
\DeclarePairedDelimiter{\sqgroup}{[}{]}
\DeclarePairedDelimiter{\lsqgroup}{[}{)}
\DeclarePairedDelimiter{\usqgroup}{(}{]}
\DeclarePairedDelimiter{\set}{\{}{\}}
\DeclarePairedDelimiter{\avg}{\langle}{\rangle}
\DeclarePairedDelimiter{\norm}{\Vert}{\Vert}
\DeclarePairedDelimiter{\abs}{\vert}{\vert}
\DeclarePairedDelimiter{\floor}{\lfloor}{\rfloor}

\newcommand{\var}{Y}
\newcommand{\varvalues}{\mathcal{\var}}
\newcommand{\vargambles}{\gamblessymbol(\varvalues)}

\newcommand{\statesymbol}{X}
\newcommand{\statessymbol}{\mathcal{X}}
\newcommand{\gamblessymbol}{\mathcal{G}}

\newcommand{\state}[1]{\statesymbol_{#1}}
\newcommand{\xval}[1]{x_{#1}}
\newcommand{\zval}[1]{z_{#1}}
\newcommand{\states}[1]{\statessymbol_{#1}}
\newcommand{\stateset}{\statessymbol}
\newcommand{\gambles}{\gamblessymbol(\statessymbol)}
\newcommand{\normedgambles}{\gamblessymbol_1(\statessymbol)}
\newcommand{\stateto}[1]{\state{1:#1}}
\newcommand{\zstateto}[2]{\zvalto{#1}\statefromto{#1+1}{#2}}

\newcommand{\statefromto}[2]{\state{#1:#2}}
\newcommand{\xvalto}[1]{\xval{1:#1}}
\newcommand{\zvalto}[1]{\zval{1:#1}}

\newcommand{\xvalfromto}[2]{\xval{#1:#2}}
\newcommand{\statesto}[1]{\states{1:#1}}

\newcommand{\statesfromto}[2]{\states{#1:#2}}
\newcommand{\gambleson}[1]{\gamblessymbol(\states{#1})}
\newcommand{\gamblesto}[1]{\gamblessymbol(\states{1:#1})}

\newcommand{\nats}{\mathbb{N}}
\newcommand{\natswithzero}{\mathbb{N}_0}
\newcommand{\reals}{\mathbb{R}}
\newcommand{\xreals}{\mathbb{R}^*}
\newcommand{\rationals}{\mathbb{Q}}
\newcommand{\upplus}{\mathbin{+^*}}

\newcommand{\init}{\square}
\newcommand{\pth}{\omega}
\newcommand{\sits}{\pths^\lozenge}
\newcommand{\pths}{\Omega}
\newcommand{\precedes}{\sqsubseteq}

\newcommand{\sprecedes}{\sqsubset}
\newcommand{\sfollows}{\sqsupset}

\newcommand{\ucut}{U}
\newcommand{\vcut}{V}
\newcommand{\highcut}[2][a,b]{\ucut^{#1}_{#2}}
\newcommand{\lowcut}[2][a,b]{\vcut^{#1}_{#2}}

\newcommand{\ex}{E}
\newcommand{\lex}{\underline{\ex}}
\newcommand{\uex}{\overline{\ex}}
\newcommand{\aex}{F}
\newcommand{\alex}{\underline{\aex}}

\newcommand{\aaex}{Q}
\newcommand{\aalex}{\underline{\aaex}}
\newcommand{\aauex}{\overline{\aaex}}

\newcommand{\pr}{P}
\newcommand{\lpr}{\underline{\pr}}
\newcommand{\upr}{\overline{\pr}}

\newcommand{\joint}[1][]{\ex_{#1}}
\newcommand{\ljoint}[1][]{\lex_{#1}}

\newcommand{\ujoint}[1][]{\uex_{#1}}
\newcommand{\statljoint}[1][]{\lex_{\mathrm{st}}}

\newcommand{\xinllocal}[3][]{\aalex\group[#1]{#2#1\vert\xval{#3}}}

\newcommand{\xtoinlocal}[3][]{\aaex\group[#1]{#2#1\vert\xvalto{#3}}}
\newcommand{\xtoinllocal}[3][]{\aalex\group[#1]{#2#1\vert\xvalto{#3}}}
\newcommand{\xtoinulocal}[3][]{\aauex\group[#1]{#2#1\vert\xvalto{#3}}}
\newcommand{\sitinlocal}[3][]{\aaex\group[#1]{#2#1\vert#3}}
\newcommand{\sitinllocal}[3][]{\aalex\group[#1]{#2#1\vert#3}}
\newcommand{\sitinulocal}[3][]{\aauex\group[#1]{#2#1\vert#3}}

\newcommand{\lmargin}[2][]{\aalex\group[#1]{#2#1\vert\init}}

\newcommand{\ltransition}[3][]{\aalex\group[#1]{#2#1\vert#3}}

\newcommand{\xtoinlglobal}[3][]{\lex\group[#1]{#2#1\vert\xvalto{#3}}}
\newcommand{\ztoinlglobal}[3][]{\lex\group[#1]{#2#1\vert\zvalto{#3}}}
\newcommand{\zxtoinlglobal}[4][]{\lex\group[#1]{#2#1\vert\zvalto{#3}\xvalfromto{#3+1}{#4}}}
\newcommand{\xtoinuglobal}[3][]{\uex\group[#1]{#2#1\vert\xvalto{#3}}}
\newcommand{\stateinlglobal}[3][]{\lex\group[#1]{#2#1\vert\state{#3}}}
\newcommand{\statetoinlglobal}[3][]{\lex\group[#1]{#2#1\vert\stateto{#3}}}
\newcommand{\zstatetoinlglobal}[4][]{\lex\group[#1]{#2#1\vert\zvalto{#3}\statefromto{#3+1}{#4}}}

\newcommand{\sitinlglobal}[3][]{\lex\group[#1]{#2#1\vert#3}}
\newcommand{\sitinuglobal}[3][]{\uex\group[#1]{#2#1\vert#3}}
\newcommand{\newlglobal}[3]{\lex_{\vert#1}(#2\vert#3)}
\newcommand{\newuglobal}[3]{\uex_{\vert#1}(#2\vert#3)}
\newcommand{\altsitinlglobal}[3][]{\underline{F}\group[#1]{#2#1\vert#3}}

\newcommand{\exact}[1]{\Gamma(#1)}
\newcommand{\indexact}[1]{\ind{\exact{#1}}}
\newcommand{\indsing}[1]{\ind{\set{#1}}}

\newcommand{\ltrans}{\underline{\trans}\,}
\newcommand{\utrans}{\overline{\trans}\,}
\newcommand{\latrans}{\underline{\atrans}}

\newcommand{\process}{\mathcal{F}}
\newcommand{\predictable}{\mathcal{S}}
\newcommand{\submartin}{\mathcal{M}}
\newcommand{\supermartin}{\mathcal{M}}
\newcommand{\esupermartin}{\supermartin^*}
\newcommand{\test}[1][]{\mathcal{T}^{#1}}
\newcommand{\gambleprocess}{\mathcal{D}}
\newcommand{\summed}[1]{\mathcal{I}^{#1}}

\newcommand{\gain}[1]{\mathcal{W}[#1]}
\newcommand{\avgain}[1]{\avg{\mathcal{W}}[#1]}
\newcommand{\ergodic}[1]{\mathcal{A}[#1]}
\newcommand{\pavg}[2][\predictable]{\avg{#2}_{#1}}

\newcommand{\martins}{\mathbb{M}}
\newcommand{\submartins}{\underline{\martins}}
\newcommand{\supermartins}{\overline{\martins}}
\newcommand{\basubmartins}{\underline{\martins}_{\mathrm{b}}}
\newcommand{\bbsupermartins}{\overline{\martins}_{\mathrm{b}}}

\newcommand{\timefromto}[3][]{\tau^{#1}_{{#2}\to{#3}}}
\newcommand{\ltimefromto}[3][]{\underline{\tau}^{#1}_{{#2}\to{#3}}}
\newcommand{\utimefromto}[3][]{\overline{\tau}^{#1}_{{#2}\to{#3}}}

\newcommand{\coe}{\rho}
\newcommand{\varnorm}[2][]{\norm[#1]{#2}_{\mathrm{v}}}
\newcommand{\cset}[2]{\set{#1\colon#2}}
\newcommand{\ind}[1]{\mathbb{I}_{#1}}
\newcommand{\then}{\Rightarrow}

\newcommand{\ifandonlyif}{\Leftrightarrow}
\newcommand{\solp}{\mathfrak{M}}
\newcommand{\andstate}{\,\cdot}
\newcommand{\andpath}{\,\bullet}
\newcommand{\ltheta}{\underline{\theta}}
\newcommand{\utheta}{\overline{\theta}}

\newcommand{\pflike}{Perron--Frobenius-like}
\newcommand{\pftext}{Perron--Frobenius}
\newcommand{\pf}{\mathrm{PF}}

\newtheorem{theorem}{Theorem}
\newtheorem{proposition}[theorem]{Proposition}
\newtheorem{lemma}[theorem]{Lemma}
\newtheorem{corollary}[theorem]{Corollary}

\theoremstyle{definition}
\newtheorem{example}{Example}

\begin{document}
\title[A study of imprecise Markov chains]{A study of imprecise Markov chains:\\joint lower expectations and point-wise ergodic theorems}
\author{Gert de Cooman}
\address{Ghent University, SYSTeMS Research Group, Technologiepark--Zwijnaarde 914, 9052 Zwijnaarde, Belgium.}
\email{Gert.deCooman@UGent.be}
\author{Jasper De Bock} 
\address{Ghent University, SYSTeMS Research Group, Technologiepark--Zwijnaarde 914, 9052 Zwijnaarde, Belgium.}
\email{Jasper.DeBock@UGent.be}
\author{Stavros Lopatatzidis}
\address{Ghent University, SYSTeMS Research Group, Technologiepark--Zwijnaarde 914, 9052 Zwijnaarde, Belgium.}
\email{Stavros.Lopatatzidis@UGent.be}

\begin{abstract}
We justify and discuss expressions for joint lower and upper expectations in imprecise probability trees, in terms of the sub- and supermartingales that can be associated with such trees.
These imprecise probability trees can be seen as discrete-time stochastic processes with finite state sets and transition probabilities that are imprecise, in the sense that they are only known to belong to some convex closed set of probability measures. 
We derive various properties for their joint lower and upper expectations, and in particular a law of iterated expectations.
We then focus on the special case of imprecise Markov chains, investigate their Markov and stationarity properties, and use these, by way of an example, to derive a system of non-linear equations for lower and upper expected transition and return times.
Most importantly, we prove a game-theoretic version of the strong law of large numbers for submartingale differences in imprecise probability trees, and use this to derive point-wise ergodic theorems for imprecise Markov chains.
\end{abstract}


\maketitle

\section{Introduction}\label{sec:introduction}
In Ref.~\cite{cooman2007d}, \Citeauthor{cooman2007d} made a first attempt at laying the foundations for a theory of discrete-event (and discrete-time) stochastic processes that are governed by sets of, rather than single, probability measures.
They showed how this can be done by connecting \citeauthor{walley1991}'s \citeyearpar{walley1991} theory of coherent lower previsions with ideas and results from \citeauthor{shafer2001}'s \citeyearpar{shafer2001} game-theoretic approach to probability theory.
In later papers, \Citet{cooman2008} applied these ideas to finite-state discrete-time Markov chains, inspired by the work of \citet{hartfiel1998}. 
They showed how to perform efficient inferences in, and proved a {\pflike} theorem for, so-called imprecise Markov chains, which are finite-state discrete-time Markov chains whose transition probabilities are imprecise, in the sense that they are only known to belong to a convex closed set of probability measures---typically due to partial assessments involving probabilistic inequalities. 
This work was later refined and extended by \citet{hermans2012} and \citet{skulj2013}.
\par
The {\pflike} theorems in these papers give equivalent necessary and sufficient conditions for the uncertainty model---a set of probabilities---about the state $\state{n}$ to converge, for $n\to+\infty$, to an uncertainty model that is independent of the uncertainty model for the initial state $\state{1}$. 
\par
In Markov chains with `precise' transition probabilities, this convergence behaviour is sufficient for a point-wise ergodic theorem to hold, namely that:
\begin{equation*}
\lim_{n\to+\infty}\frac{1}{n}\sum_{k=1}^{n}f(\state{k})
=\ex_{\infty}(f)
\text{ almost surely}
\end{equation*}
for all real functions $f$ on the finite state set $\statessymbol$, where $\ex_{\infty}$ is the limit expectation operator that the expectation operators $\ex_{n}$ for the state $\state{n}$ at time $n$ converge to point-wise, independently of the initial model $\ex_{1}$ for $\state{1}$, according to the classical {\pftext} Theorem.\footnote{Actually, much more general results can be proved, for functions $f$ that do not depend on a single state only, but on the entire sequence of states; see for instance Ref.~\cite[Chapter~20]{kallenberg2002}. In this paper, we will focus on the simpler version, but we will show that it can be extended to functions on a finite number of states.}
\par
One of the aims of the present paper is to extend this result to a version for imprecise Markov chains; see Theorem~\ref{thm:point-wise:ergodic:theorem} further on.
In contradistinction with the so-called \emph{Markov set-chains} more commonly encountered in the literature~\cite{hartfiel1991,hartfiel1994,hartfiel1998}, our imprecise Markov chains are not merely collections of (precise) Markov chains---incidentally, for such Markov set-chains, proving an ergodic theorem would be a fairly trivial affair, as it would amount to applying the classical point-wise ergodic theorem to each of the Markov chains in the collection.
Rather, as we will explain in Section~\ref{sec:markov}, our imprecise Markov chains correspond to a collection of stochastic processes that need not satisfy the Markov property.
They are only `superficially Markov', in the sense that their \emph{sets} of transition probabilities satisfy a Markov condition, whereas the individual members of those sets need not.
In other words, imprecise Markov chains are \emph{not} simply collections of precise Markov chains, but rather correspond to collections of general stochastic processes whose transition models belong to sets that satisfy a Markov condition.
\par
How do we mean to go about proving our ergodicity result?
In Section~\ref{sec:ip}, we explain what we mean by imprecise probability models: we extend the notion of an expectation operator to so-called lower (and upper) expectation operators, and explain how these can be associated with (convex and closed) sets of expectation operators. 
\par
In Section~\ref{sec:processes}, we explain how these generalised uncertainty models can be combined with event trees to form so-called imprecise probability trees, to produce a simple theory of discrete-time stochastic processes.
We show in particular how to combine local uncertainty models associated with the nodes in the tree into global uncertainty models (global conditional lower expectations) about the paths in the tree, and how this procedure is related to sub- and supermartingales.
We also indicate how it extends and subsumes the (precise-)probabilistic approach. 
\par
In Section~\ref{sec:large:numbers} we prove a very general strong law of large numbers for submartingale differences in our imprecise probability trees.
Our point-wise ergodic theorem will turn out to be a consequence of this in the particular context of imprecise Markov chains.
Section~\ref{sec:properties:of:global:models} is more technical, and is devoted to extending the joint lower and upper expectations to extended real variables, and to proving a number of important properties for them, such as generalisations of well-known coherence properties, and a version of the law of iterated (lower) expectations.
\par
We explain what imprecise Markov chains are in Section~\ref{sec:markov}: how they are special cases of imprecise probability trees, how to do efficient inference for them, and how to define {\pflike} behaviour.
We generalise existing results~\cite{cooman2008} about global lower expectations in such imprecise Markov trees from a finite to an infinite time horizon, and from bounded real argument functions to extended real-valued ones.
We also explore the influence of time shifts on the global (conditional) lower expectations, investigate their Markov properties, prove various corollaries of the law of iterated lower expectations, and discuss stationarity and its relation with {\pflike} behaviour.
As an illustration of the power of our approach, we derive in Section~\ref{sec:transition:and:return:times} a system of non-linear equations for lower and upper expected transition and return times, and solve it in special case. 
\par
In Section~\ref{sec:interesting:equality} we show that there is an interesting identity between the time averages that appear in our strong law of large numbers, and the ones that appear in the point-wise ergodic theorem.
The discussion in Section~\ref{sec:perron:frobenius} first focusses on a number of terms in this identity, and investigates their convergence for {\pflike} imprecise Markov chains.
This allows us to use the identity to prove two versions of the point-wise ergodic theorem: one for functions of a single state (Theorem~\ref{thm:point-wise:ergodic:theorem}) and its extension (Corollary~\ref{cor:point-wise:ergodic:theorem}) to functions of a finite number of states.
We briefly discuss their significance in Section~\ref{sec:conclusion}. 
\par
Some of the results in this paper have already been discussed---without proofs---in an earlier conference version \cite{cooman2015:isipta:markov}. 
This paper significantly extends the earlier version.

\section{Basic notions from imprecise probabilities}\label{sec:ip}
Let us begin with a brief sketch of a few basic definitions and results about imprecise probabilities. 
For more details, we refer to \citeauthor{walley1991}'s~\cite{walley1991} seminal book, as well as more recent textbooks~\cite{augustin2013:itip,troffaes2013:lp}.

Suppose a subject is uncertain about the value that a variable $\var$ assumes in a non-empty set of possible values $\varvalues$.
He is therefore also uncertain about the value $f(\var)$ a so-called \emph{gamble}---a bounded real-valued function---$f\colon\varvalues\to\reals$ on the set $\varvalues$ assumes in $\reals$.
We will also call such an $f$ a gamble \emph{on $\var$} when we want to make explicit what variable $\var$ the gamble $f$ is intended to depend on.
The subject's uncertainty is modelled by a \emph{lower expectation}\footnote{In the literature~\cite{walley1991,augustin2013:itip,troffaes2013:lp}, other names, such as coherent lower expectation, or coherent lower prevision, have also been given to this concept.} $\lex$, which is a real functional defined on the set $\vargambles$ of all gambles on the set $\varvalues$, satisfying the following basic so-called \emph{coherence axioms}:
\begin{enumerate}[label=\upshape LE\arabic*.,ref=\upshape LE\arabic*,leftmargin=*]
\item\label{it:lex:bounds} $\lex(f)\geq\inf f$ for all $f\in\vargambles$;\hfill[bounds]
\item\label{it:lex:super} $\lex(f+g)\geq\lex(f)+\lex(g)$ for all $f,g\in\vargambles$;\hfill[superadditivity]
\item\label{it:lex:homo} $\lex(\lambda f)=\lambda\lex(f)$ for all $f\in\vargambles$ and real $\lambda\geq0$.\hfill[non-negative homogeneity]
\end{enumerate}
One---but by no means the only\footnote{See Refs.~\cite{walley1991,cooman2014:itip:previsions,troffaes2013:lp} for other interpretations.}---way to interpret $\lex(f)$ is as a lower bound on the expectation $\ex(f)$ of the gamble $f(\var)$.
The corresponding upper bounds are given by the \emph{conjugate upper expectation} $\uex$, defined by $\uex(f)\coloneqq-\lex(-f)$ for all $f\in\vargambles$. 
It follows from the coherence axioms~\ref{it:lex:bounds}--\ref{it:lex:homo} that
\begin{enumerate}[label=\upshape LE\arabic*.,ref=\upshape LE\arabic*,leftmargin=*,resume]
\item\label{it:lex:monotone} $\lex(f)\leq\lex(g)$ and $\uex(f)\leq\uex(g)$ for all $f,g\in\vargambles$ with $f\leq g$;
\item\label{it:lex:more:bounds} $\inf f\leq\lex(f)\leq\uex(f)\leq\sup f$ for all $f\in\vargambles$;
\item\label{it:lex:sub} $\uex(f+g)\leq\uex(f)+\uex(g)$ for all $f,g\in\vargambles$;\hfill[subadditivity]
\item\label{it:lex:homo:upper} $\uex(\lambda f)=\lambda\uex(f)$ for all $f\in\vargambles$ and real $\lambda\geq0$.\hfill[non-negative homogeneity]
\item\label{it:lex:constant:additivity} $\lex(f+\mu)=\lex(f)+\mu$ and $\uex(f+\mu)=\uex(f)+\mu$ for all $f\in\vargambles$ and real $\mu$.
\end{enumerate}
Lower and upper expectations will be the basic uncertainty models we consider in this paper.

The \emph{indicator} $\ind{A}$ of an \emph{event} $A$---a subset of $\varvalues$---is the gamble on $\var$ that assumes the value $1$ on $A$ and $0$ outside $A$.
It allows us to introduce the \emph{lower} and \emph{upper probabilities} of the event~$A$ as $\lpr(A)\coloneqq\lex(\ind{A})$ and $\upr(A)\coloneqq\uex(\ind{A})$, respectively.
They can be seen as lower and upper bounds on the probability $\pr(A)$ of $A$, and satisfy the conjugacy relation $\upr(A)=1-\lpr(\varvalues\setminus A)$.

When the lower bound $\lex$ coincides with the upper bound $\uex$, the resulting functional $\ex\coloneqq\lex=\uex$ satisfies the defining axioms of an \emph{expectation}:
\begin{enumerate}[label=\upshape E\arabic*.,ref=\upshape E\arabic*,leftmargin=*]
\item $\ex(f)\geq\inf f$ for all $f\in\vargambles$;\hfill[bounds]
\item $\ex(f+g)=\ex(f)+\ex(g)$ for all $f,g\in\vargambles$;\hfill[additivity]
\item $\ex(\lambda f)=\lambda\ex(f)$ for all $f\in\vargambles$ and real $\lambda$.\hfill[homogeneity]
\end{enumerate}
When $\varvalues$ is finite, $\ex$ is trivially the expectation associated with a (probability) mass function $p$ defined by $p(y)\coloneqq\lpr(\set{y})=\upr(\set{y})$ for all $y\in\varvalues$, because it follows from the expectation axioms that then $\ex(f)=\sum_{y\in\varvalues}f(y)p(y)$; see for instance also the detailed discussion in Ref.~\cite{troffaes2013:lp}.

With any lower expectation $\lex$, we can always associate the following convex and closed\footnote{The `closedness' is associated with the weak* topology of point-wise convergence~\cite[Section~3.6]{walley1991}.} set of \emph{compatible} expectations:
\begin{equation}\label{eq:linear:dominates}
\solp(\lex)\coloneqq\cset{\ex\text{ expectation}}{(\forall f\in\vargambles)\lex(f)\leq\ex(f)\leq\uex(f)},   
\end{equation}   
and the properties \ref{it:lex:bounds}--\ref{it:lex:homo} then guarantee that
\begin{equation}\label{eq:lower:envelope}
\lex(f)=\min\cset{\ex(f)}{\ex\in\solp(\lex)}
\text{ and }
\uex(f)=\max\cset{\ex(f)}{\ex\in\solp(\lex)}
\text{ for all $f\in\vargambles$}.
\end{equation}
In this sense, an imprecise probability model $\lex$ can always be identified with a closed convex set $\solp(\lex)$ of compatible `precise' probability models $\ex$.

\section{Discrete-time finite-state imprecise stochastic processes}\label{sec:processes}
We consider a discrete-time process as a sequence of variables, henceforth called \emph{states}, $\state{1}$, $\state{2}$, \dots, $\state{n}$, \dots, where the state $\state{k}$ at time $k$ is assumed to take values in a non-empty \emph{finite} set~$\states{k}$. 

\subsection{Event trees, situations, paths and cuts}
We will use, for any natural $k\leq\ell$, the notation $\statefromto{k}{\ell}$ for the tuple $(\state{k},\dots,\state{\ell})$, which can be seen as a variable assumed to take values in the Cartesian product set $\statesfromto{k}{\ell}\coloneqq\times_{r=k}^\ell\states{r}$.
We denote the set of all natural numbers (without $0$) by $\nats$, and let $\natswithzero\coloneqq\nats\cup\set{0}$.

We call any $\xvalto{n}\in\statesto{n}$ for $n\in\natswithzero$ a \emph{situation} and we denote the set of all situations by $\sits$. 
So any situation is a finite string of possible values for the consecutive states, and if we denote the empty string by $\init$, then in particular, $\statesto{0}=\set{\init}$.
$\init$ is called the \emph{initial situation}. 
We also use the generic notations $s$, $t$ or $u$ for situations.

An infinite sequence of state values is called a \emph{path}, and we denote the set of all paths---also called the \emph{sample space}---by $\pths$.
Hence
\begin{equation*}
\sits\coloneqq\bigcup_{n\in\natswithzero}\statesto{n}
\text{ and }
\pths\coloneqq\times_{r=1}^\infty\states{k}.
\end{equation*}
We will denote generic paths by $\pth$. 
For any path $\pth\in\pths$, the initial sequence that consists of its first $n$ elements is a situation in $\statesto{n}$ that is denoted by $\pth^{n}$. 
Its $n$-th element belongs to $\states{n}$ and is denoted by $\pth_n$.
As a convention, we let its $0$-th element be the initial situation $\pth^0=\pth_0=\init$. 
The possible realisations $\pth$ of a process can be represented graphically as paths in a so-called \emph{event tree}, where each node is a situation; see Figure~\ref{fig:eventtree}. 

\begin{figure}[ht]
\centering\footnotesize
\begin{tikzpicture}
\tikzstyle{level 1}=[sibling distance=20em]
\tikzstyle{level 2}=[sibling distance=10em]
\tikzstyle{level 3}=[sibling distance=5em]
\tikzstyle{level 4}=[level distance=2em]
\node[root] (root) {} [grow=down,level distance=8ex]
child {node[nonterminal] (a) {$a\vphantom{)}$}
  child {node[nonterminal] (aa) {$(a,a)$}
    child {node[nonterminal] (aaa) {$(a,a,a)$}
      child[black,dotted]} 
    child {node[nonterminal] (aab) {$(a,a,b)$}
      child[black,dotted]}
  }
  child {node[nonterminal] (ab) {$(a,b)$}
    child {node[nonterminal] (aba) {$(a,b,a)$}
      child[black,dotted]}
    child {node[nonterminal] (abb) {$(a,b,b)$}
      child[black,dotted]}
  }
}
child {node[nonterminal] (b) {$b\vphantom{)}$}
  child {node[nonterminal] (ba) {$(b,a)$}
    child {node[nonterminal] (baa) {$(b,a,a)$}
      child[black,dotted]}
    child {node[nonterminal] (bab) {$(b,a,b)$}
      child[black,dotted]}
  }
  child {node[nonterminal] (bb) {$(b,b)$}
    child {node[nonterminal] (bba) {$(b,b,a)$}
      child[black,dotted]}
    child {node[nonterminal] (bbb) {$(b,b,b)$}
      child[black,dotted]}
  }
};
\draw[cut] (b) -- +(1,0);
\draw[cut] (b) -- (a) -- +(-2,0) node[left,local] {$\statesto{1}$};
\draw[cut] (bb) -- +(1,0);
\draw[cut] (bb) -- (ba) -- (ab) -- (aa) -- +(-1,0) node[left,local] {$\statesto{2}$};
\draw[cut] (bbb) -- +(1,0);
\draw[cut] (bbb) -- (bba) -- (bab) -- (baa) -- (abb) -- (aba) -- (aab) -- (aaa) -- +(-1,0) node[left,local] {$\statesto{3}$};
\end{tikzpicture}
\caption{
    The (initial part of the) event tree for a process whose states can assume two values, $a$ and~$b$, and can change at time instants $n=1,2,3,\dots$\ 
    Each node in the tree corresponds to a situation.
    Also depicted are the respective sets of situations (cuts)~$\statesto{1}$, $\statesto{2}$ and $\statesto{3}$ where the states at times~$1$, $2$ and~$3$ are revealed.}
\label{fig:eventtree}
\end{figure}
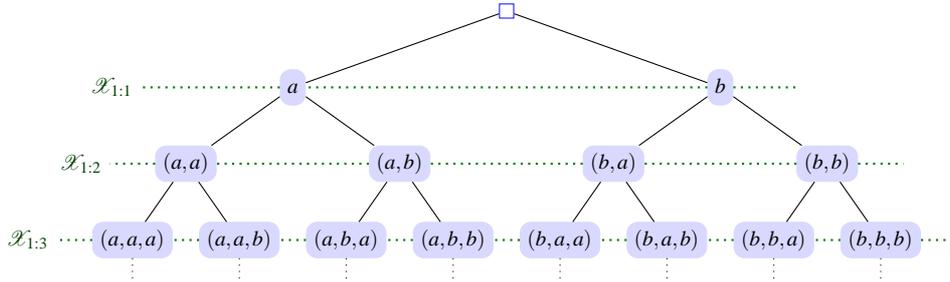

We write that $s\precedes t$, and say that \emph{$s$ precedes $t$} or that \emph{$t$ follows $s$}, when every path that goes through $t$ also goes through $s$.
The binary relation $\precedes$ is a partial order, and we write $s\sprecedes t$ whenever $s\precedes t$ but not $s=t$.
We say that $s$ and $t$ are \emph{incomparable} when neither $s\precedes t$ nor $t\precedes s$.

A (partial) \emph{cut} $\ucut$ is a collection of mutually incomparable situations, and represents a stopping time.
For any two cuts $\ucut$ and $\vcut$, we define the following sets of situations:
\begin{align*}
\sqgroup{\ucut,\vcut}
\coloneqq&\cset{s\in\sits}{(\exists u\in\ucut)(\exists v\in\vcut)u\precedes s\precedes v}\\
\lsqgroup{\ucut,\vcut}
\coloneqq&\cset{s\in\sits}{(\exists u\in\ucut)(\exists v\in\vcut)u\precedes s\sprecedes v}\\
\usqgroup{\ucut,\vcut}
\coloneqq&\cset{s\in\sits}{(\exists u\in\ucut)(\exists v\in\vcut)u\sprecedes s\precedes v}\\
\group{\ucut,\vcut}
\coloneqq&\cset{s\in\sits}{(\exists u\in\ucut)(\exists v\in\vcut)u\sprecedes s\sprecedes v}.
\end{align*}
When a cut $\ucut$ consists of a single element $u$, then we will identify $\ucut=\set{u}$ and $u$.
This slight abuse of notation will for instance allow us to write $\sqgroup{u,v}=\cset{s\in\sits}{u\precedes s\precedes v}$ and also $\group{\ucut,v}=\cset{s\in\sits}{(\exists u\in\ucut)u\sprecedes s\sprecedes v}$.
We also write $\ucut\sprecedes\vcut$ if $(\forall v\in\vcut)(\exists u\in\ucut)u\sprecedes v$. 
Observe that in that case $\ucut\cap\vcut=\emptyset$.
In particular, $s\sfollows\ucut$ when there is some $u\in\ucut$ such that $s\sfollows u$, or in other words if $\lsqgroup{\ucut,s}\neq\emptyset$.

\subsection{Processes}
A \emph{process} $\process$ is a map defined on $\sits$. 
A \emph{real process} is a real-valued process: it associates a real number $\process(\xvalto{n})\in\reals$ with any situation $\xvalto{n}$.
It is called \emph{bounded below} if there is some real $B$ such that $\process(s)\geq B$ for all situations $s\in\sits$, and \emph{bounded above} if $-\process$ is bounded below.

A \emph{gamble process} $\gambleprocess$ is a process that associates with any situation $\xvalto{n}$ a gamble $\gambleprocess(\xvalto{n})\in\gambleson{n+1}$ on $\state{n+1}$.
It is called \emph{uniformly bounded} if there is some real $B$ such that $\abs{\gambleprocess(s)}\leq B$ for all situations $s\in\sits$.
With any real process $\process$, we can always associate a gamble process $\Delta\process$, called the \emph{process difference}. For every situation $\xvalto{n}$, the gamble $\Delta\process(\xvalto{n})\in\gambleson{n+1}$ is defined by\footnote{Our assumption that $\states{n+1}$ is finite is crucial here because it guarantees that $\Delta\process(\xvalto{n})$ is bounded, which in turn implies that it is indeed a gamble.}
\begin{equation*}
\Delta\process(\xvalto{n})(\xval{n+1})
\coloneqq\process(\xvalto{n+1}
)
-\process(\xvalto{n})
\text{ for all $\xval{n+1}\in\states{n+1}$}.
\end{equation*}
We will denote this more succinctly by $\Delta\process(\xvalto{n})=\process(\xvalto{n}\andstate)-\process(\xvalto{n})$, where the `$\cdot$' represents the generic value of the next state $\state{n+1}$.

Conversely, with a gamble process $\gambleprocess$, we can associate a real process $\summed{\gambleprocess}$, defined by  
\begin{equation*}
\summed{\gambleprocess}(\xvalto{n})
\coloneqq\sum_{k=0}^{n-1}\gambleprocess(\xvalto{k})(\xval{k+1})
\quad\text{ for all $n\in\natswithzero$ and $\xvalto{n}\in\statesto{n}$}.
\end{equation*}
Clearly, $\Delta\summed{\gambleprocess}=\gambleprocess$ and $\process=\process(\init)+\summed{\Delta\process}$.

Also, with any real process $\process$ we can associate the \emph{path-averaged process} $\avg{\process}$, which is the real process defined by: 
\begin{equation*}
\avg{\process}(\xvalto{n})
\coloneqq
\begin{cases}
0&\text{if $n=0$}\\
\dfrac{\sum_{k=0}^{n-1}\Delta\process(\xvalto{k})(\xval{k+1})}{n}
&\text{if $n>0$}
\end{cases}
\quad\text{ for all $n\in\natswithzero$ and $\xvalto{n}\in\statesto{n}$}.
\end{equation*}
We can generalise this notion of path-averaging even further as follows.
Consider any real process $\predictable$ that only assumes values in $\{0,1\}$.
Then the \emph{$\predictable$-averaged process} $\pavg{\process}$ is the real process defined by: 
\begin{multline*}
\pavg{\process}(\xvalto{n})
\coloneqq
\begin{cases}
0&\text{if $\sum_{k=0}^{n-1}\predictable(\xvalto{k})=0$}\\
\dfrac{\sum_{k=0}^{n-1}\predictable(\xvalto{k})\Delta\process(\xvalto{k})(\xval{k+1})}{\sum_{k=0}^{n-1}\predictable(\xvalto{k})}
&\text{if $\sum_{k=0}^{n-1}\predictable(\xvalto{k})>0$}
\end{cases}\\
\quad\text{ for all $n\in\natswithzero$ and $\xvalto{n}\in\statesto{n}$}.
\end{multline*}
Of course, if $\predictable$ is identically equal to $1$ in all situations, then $\pavg{\process}=\avg{\process}$.
In order to unburden our formulas somewhat, we will permit ourselves the slight abuse of notation $\summed{\predictable}(\xvalto{n})\coloneqq\sum_{k=0}^{n-1}\predictable(\xvalto{k})$ for all $n\in\natswithzero$ and $\xvalto{n}\in\statesto{n}$.

\subsection{Imprecise probability trees, submartingales and supermartingales}
\label{sec:basic:iptrees}
The standard way to turn an event tree into a \emph{probability tree} is to attach to each of its nodes, or situations $\xvalto{n}$, a \emph{local} probability model $\xtoinlocal{\cdot}{n}$ for what will happen immediately afterwards, i.e.~for the value that the next state $\state{n+1}$ will assume in $\states{n+1}$.
This local model $\xtoinlocal{\cdot}{n}$ is then an expectation operator on the set $\gambleson{n+1}$ of all gambles $g(\state{n+1})$ on the next state $\state{n+1}$, conditional on observing $\stateto{n}=\xvalto{n}$.

In a completely similar way, we can turn an event tree into an \emph{imprecise probability tree} by attaching to each of its situations $\xvalto{n}$ a local \emph{imprecise} probability model $\xtoinllocal{\cdot}{n}$ for what will happen immediately afterwards, i.e.~for the value that the next state $\state{n+1}$ will assume in $\states{n+1}$.
This local model $\xtoinllocal{\cdot}{n}$ is then a \emph{lower} expectation operator on the set $\gambleson{n+1}$ of all gambles $g(\state{n+1})$ on the next state $\state{n+1}$, conditional on observing $\stateto{n}=\xvalto{n}$.
This is represented graphically in Figure~\ref{fig:iptree}.

\begin{figure}[ht]
\centering\footnotesize
\begin{tikzpicture}
\tikzstyle{level 1}=[sibling distance=20em]
\tikzstyle{level 2}=[sibling distance=10em]
\tikzstyle{level 3}=[sibling distance=5em]
\tikzstyle{level 4}=[level distance=2em]
\node[root] (root) {} [grow=down,level distance=8ex]
child {node[nonterminal] (a) {$a\vphantom{)}$}
  child {node[nonterminal] (aa) {$(a,a)$}
    child {node[nonterminal] (aaa) {$(a,a,a)$}
      child[black,dotted]}
    child {node[nonterminal] (aab) {$(a,a,b)$}
      child[black,dotted]}
  }
  child {node[nonterminal] (ab) {$(a,b)$}
    child {node[nonterminal] (aba) {$(a,b,a)$}
      child[black,dotted]}
    child {node[nonterminal] (abb) {$(a,b,b)$}
      child[black,dotted]}
  }
}
child {node[nonterminal] (b) {$b\vphantom{)}$}
  child {node[nonterminal] (ba) {$(b,a)$}
    child {node[nonterminal] (baa) {$(b,a,a)$}
      child[black,dotted]}
    child {node[nonterminal] (bab) {$(b,a,b)$}
      child[black,dotted]}
  }
  child {node[nonterminal] (bb) {$(b,b)$}
    child {node[nonterminal] (bba) {$(b,b,a)$}
      child[black,dotted]}
    child {node[nonterminal] (bbb) {$(b,b,b)$}
      child[black,dotted]}
  }
};
\draw[local,thick] (root) +(190:1.5em) arc (190:350:1.5em);
\draw[local,thick] (b) +(210:2em) arc (210:330:2em);
\draw[local,thick] (a) +(210:2em) arc (210:330:2em);
\draw[local,thick] (bb) +(230:2.25em) arc (230:310:2.25em);
\draw[local,thick] (ba) +(230:2.25em) arc (230:310:2.25em);
\draw[local,thick] (ab) +(230:2.25em) arc (230:310:2.25em);
\draw[local,thick] (aa) +(230:2.25em) arc (230:310:2.25em);
\path (root) +(275:2.35em) node[local] {$\lmargin{\cdot}$};
\path (a) +(270:2.95em) node[local] {$\ltransition{\cdot}{a}$};
\path (b) +(270:2.95em) node[local] {$\ltransition{\cdot}{b}$};
\path (aa) +(300:2.8em) node[local,above right] {$\ltransition{\cdot}{a,a}$};
\path (bb) +(300:2.8em) node[local,above right] {$\ltransition{\cdot}{b,b}$};
\path (ba) +(300:2.8em) node[local,above right] {$\ltransition{\cdot}{b,a}$};
\path (ab) +(300:2.8em) node[local,above right] {$\ltransition{\cdot}{a,b}$};
\end{tikzpicture}
\caption{The (initial part of the) imprecise probability tree for a process whose states can assume two values, $a$ and~$b$, and can change at time instants $n=1,2,3,\dots$}
\label{fig:iptree}
\end{figure}
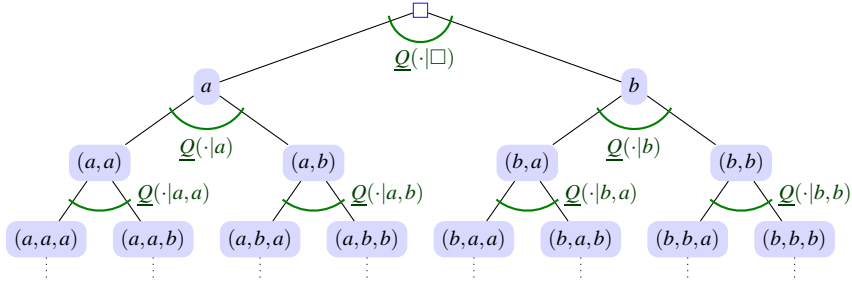

In a \emph{given} imprecise probability tree, a \emph{submartingale} $\submartin$ is a real process such that $\xtoinllocal{\Delta\submartin(\xvalto{n})}{n}\geq0$ for all $n\in\natswithzero$ and $\xvalto{n}\in\statesto{n}$: all submartingale differences have non-negative lower expectation.
A real process $\supermartin$ is a \emph{supermartingale} if $-\supermartin$ is a submartingale, meaning that $\xtoinulocal{\Delta\submartin(\xvalto{n})}{n}\leq0$ for all $n\in\natswithzero$ and $\xvalto{n}\in\statesto{n}$: all supermartingale differences have non-positive upper expectation.
We denote the set of all submartingales for a given imprecise probability tree by $\submartins$---whether a real process is a submartingale depends of course on the local uncertainty models.
The set of all submartingales that are bounded above is denoted by $\basubmartins$.
Similarly, the set $\supermartins\coloneqq-\submartins$ is the set of all supermartingales, and $\bbsupermartins\coloneqq-\basubmartins$ the set of all supermartingales that are bounded below.

In the present context of probability trees, we will also call \emph{variable} any function defined on the so-called \emph{sample space}---the set $\pths$ of all paths.
When this variable is real-valued and bounded, we will also call it a \emph{gamble} on $\pths$.
When it is extended real-valued, meaning that it assumes values in the set $\xreals\coloneqq\reals\cup\set{-\infty,+\infty}$, we call in an \emph{extended real variable}.
An \emph{event} $A$ in this context is a subset of $\pths$, and its indicator $\ind{A}$ is a gamble on $\pths$ assuming the value $1$ on $A$ and $0$ elsewhere.
With any situation $\xvalto{n}$, we can associate the so-called \emph{exact event} $\exact{\xvalto{n}}$ that $\stateto{n}=\xvalto{n}$, which is the set of all paths $\pth\in\pths$ that go through $\xvalto{n}$:
\begin{equation*}
\exact{\xvalto{n}}\coloneqq\cset{\pth\in\pths}{\pth^{n}=\xvalto{n}}.
\end{equation*}
We will also use the generic notation `$\xvalto{n}\andpath$' for all such paths in $\exact{\xvalto{n}}$.
For a given $n\in\natswithzero$, we call a variable $\xi$ \emph{$n$-measurable} if it is constant on the exact events $\exact{\xvalto{n}}$ for all $\xvalto{n}\in\statesto{n}$, or in other words, if it only depends on the values of the first $n$ states $\stateto{n}$.
We then use the obvious notation $\xi(\xvalto{n})$ for its constant value $\xi(\pth)$ on all paths $\pth$ in $\exact{\xvalto{n}}$.
In particular, this means that we can---and will---identify an $n$-measurable gamble $g$ on $\pths$ with a gamble on $\statesto{n}$, and write that $g\in\gamblesto{n}$.

We will also use a convenient notational device often encountered in texts on stochastic processes: when we want to indicate which states a process or variable depends on, we indicate them explicitly in the notation.
Thus, we use for instance the notation $\process(\stateto{n})$ to indicate the `uncertain' value of the process $\process$ after the first $n$ time steps, write $f(\state{n})$ for a variable that only depends on the value of the $n$-th state, and similarly $g(\stateto{n})$ for a variable that only depends on the values of the first $n$ states.

With a real process $\process$, we can associate in particular the following extended real variables $\liminf\process$ and $\limsup\process$, defined for all $\pth\in\pths$ by:
\begin{equation*}
\liminf\process(\pth)
\coloneqq\liminf_{n\to\infty}\process(\pth^{n})
\text{ and }
\limsup\process(\pth)
\coloneqq\limsup_{n\to\infty}\process(\pth^{n}).
\end{equation*}
If $\liminf\process(\pth)=\limsup\process(\pth)$ on some path $\pth$, then we also denote the common value there by $\lim\process(\pth)=\lim_{n\to\infty}\process(\pth^{n})$. 

The following useful result is a variation on a result proved in Ref.~\cite[Lemma~1]{shafer2012:zero-one}, and is similar in spirit to a result proved earlier in Ref.~\cite[Lemma~2]{cooman2007d}.

\begin{lemma}\label{lem:very:basic:inequality}
Consider any submartingale $\submartin$ and any situation $s\in\sits$, then:
\begin{equation*}
\submartin(s)
\leq\sup_{\pth\in\exact{s}}\liminf\submartin(\pth)
\leq\sup_{\pth\in\exact{s}}\limsup\submartin(\pth).  
\end{equation*}  
\end{lemma}

\begin{proof}
Consider any real $\alpha$, and assume that $\submartin(s)>\alpha$.
Assume that $s=\xvalto{n}$ with $n\in\natswithzero$.
Since $\submartin$ is a submartingale, we know that $\xtoinllocal{\submartin(\xvalto{n}\andstate)-\submartin(\xvalto{n})}{n}\geq0$, and therefore, by coherence [\ref{it:lex:more:bounds} and~\ref{it:lex:constant:additivity}] and the assumption, that
\begin{equation*}
\max\submartin(\xvalto{n}\andstate)\geq\xtoinllocal{\submartin(\xvalto{n}\andstate)}{n}\geq\submartin(\xvalto{n})>\alpha,
\end{equation*}
implying that there is some $\xval{n+1}\in\states{n+1}$ such that $\submartin(\xvalto{n+1})>\alpha$.
Repeating the same argument over and over again, this leads to the conclusion that there is some $\pth\in\exact{\xvalto{n}}$ such that $\submartin(\pth^{n+k})>\alpha$ for all $k\in\natswithzero$, whence $\liminf\submartin(\pth)\geq\alpha$, and therefore also $\sup_{\pth\in\exact{\xvalto{n}}}\liminf\submartin(\pth)\geq\alpha$.
The rest of the proof is now immediate.
\end{proof}

\subsection{Going from local to global belief models}\label{sec:global:models}
So far, we have associated local uncertainty models with an imprecise probability tree.
These represent, in any situation $\xvalto{n}$, beliefs about what will happen immediately afterwards, or in other words about the step from $\xvalto{n}$ to $\xvalto{n}\,\state{n+1}$.

We now want to turn these local models into global ones: uncertainty models about which entire path $\pth$ is taken in the event tree, rather than which local steps are taken from one situation to the next.
We will use the following expression for the global lower expectation conditional on the situation $s$:
\begin{equation}\label{eq:global:lower:gambles}
\sitinlglobal{g}{s}
\coloneqq\sup\cset{\submartin(s)}
{\submartin\in\submartins\text{ and }\limsup\submartin(\pth)\leq g(\pth)\text{ for all $\pth\in\exact{s}$}},
\end{equation}
and for the conjugate global upper expectation conditional on the situation $s$:
\begin{align}
\sitinuglobal{g}{s}
\coloneqq&\inf\cset{\supermartin(s)}
{\supermartin\in\supermartins\text{ and }\liminf\supermartin(\pth)\geq g(\pth)\text{ for all $\pth\in\exact{s}$}}
\label{eq:global:upper:gambles}\\
=&-\sitinlglobal{-g}{s},\notag
\end{align}
where $g$ is any gamble on $\pths$, and $s\in\sits$ any situation.
We use the simplified notations $\ljoint=\sitinlglobal{\cdot}{\init}$ and $\ujoint=\sitinuglobal{\cdot}{\init}$ for the (unconditional) global models, associated with the initial situation $\init$.

Our reasons for using these so-called \emph{Shafer--Vovk--Ville formulae}\footnote{We give this name to these formulae because Glenn Shafer and Vladimir Vovk first suggested them, based on the ideas of Jean Ville; see the discussion of Ville's Theorem in Ref.~\cite[Appendix~8.5]{shafer2001}.} are fourfold.

First of all, they are formally very closely related to the expressions for lower and upper prices in Shafer and Vovk's game-theoretic approach to probabilities, see for instance Refs.~\cite[Chapter~8.3]{shafer2001} and~\cite[Section 6.3]{vovk2014:itip}.
This allows us to import and adapt, with the necessary care, quite a number of powerful convergence results from that theory, as we will see in Section~\ref{sec:large:numbers}.
Moreover, Shafer and Vovk (see for instance Refs.~\cite[Proposition~8.8]{shafer2001} and~\cite[Section~6.3]{vovk2014:itip}) have shown that they satisfy our defining properties for lower and upper expectations in Section~\ref{sec:ip}, which is why we are calling them lower and upper expectations; see also Proposition~\ref{prop:properties:of:global:expectations} further on.

Secondly, as we gather from the following proposition and corollary, the expressions~\eqref{eq:global:lower:gambles} and~\eqref{eq:global:upper:gambles} coincide for $n$-measurable gambles on $\pths$ with the formulae derived in Ref.~\cite{cooman2007d} as the most conservative\footnote{By more conservative, we mean associated with a larger set of precise models, so point-wise smaller for lower expectations, and point-wise larger for upper expectations.} global lower and upper expectations that extend the local models.\footnote{We have also shown in recent, still unpublished work that in a more general context---where $\state{k}$ takes values in a possibly \emph{infinite} set $\states{k}$---for arbitrary gambles on $\pths$ they are the most conservative global models that extend the local ones and satisfy additional conglomerability and continuity properties.}

\begin{proposition}\label{prop:finite:horizon:gambles}
For any situation $\xvalto{m}\in\sits$ and any $n$-measurable gamble~$g$ on $\pths$, with $n,m\in\natswithzero$ such that $n\geq m$:
\begin{align*}
\xtoinlglobal{g}{m}
&=\sup\cset{\submartin(\xvalto{m})}
{\submartin\in\submartins
\text{ and }
(\forall\xvalfromto{m+1}{n}\in\statesfromto{m+1}{n})\submartin(\xvalto{n})\leq g(\xvalto{n})}\\
\xtoinuglobal{g}{m}
&=\inf\cset{\supermartin(\xvalto{m})}
{\supermartin\in\supermartins
\text{ and }
(\forall\xvalfromto{m+1}{n}\in\statesfromto{m+1}{n})\supermartin(\xvalto{n})\geq g(\xvalto{n})}.
\end{align*}
\end{proposition}

\begin{proof}
This is an immediate consequence of Proposition~\ref{prop:finite:horizon} in Section~\ref{sec:properties:of:global:models}.
\end{proof}

\begin{corollary}\label{cor:global:and:local:gambles}
Consider any $n\in\natswithzero$, any situation $\xvalto{n}\in\sits$, and any $(n+1)$-measurable gamble $g$ on $\pths$. Then
\begin{equation*}
\xtoinlglobal{g}{n}=\xtoinllocal{g(\xvalto{n}\andstate)}{n}
\text{ and }
\xtoinuglobal{g}{n}=\xtoinulocal{g(\xvalto{n}\andstate)}{n}.
\end{equation*}
\end{corollary}

\begin{proof}
We give the proof for the lower expectation; the proof for the upper expectation is completely similar.

First, consider any $\submartin\in\submartins$ such that $\submartin(\xvalto{n}\andstate)\leq g(\xvalto{n}\andstate)$, then it follows from coherence [\ref{it:lex:monotone} and~\ref{it:lex:constant:additivity}] and the submartingale character of $\submartin$ that
\begin{equation*}
\xtoinllocal{g(\xvalto{n}\andstate)}{n}
\geq\xtoinllocal{\submartin(\xvalto{n}\andstate)}{n}
\geq\submartin(\xvalto{n}),
\end{equation*}
so Proposition~\ref{prop:finite:horizon:gambles} guarantees that $\xtoinlglobal{g}{n}\leq\xtoinllocal{g(\xvalto{n}\andstate)}{n}$.

To show that the inequality is actually an equality, consider any submartingale $\submartin$ such that $\submartin(\xvalto{n})=\xtoinllocal{g(\xvalto{n}\andstate)}{n}$ and $\submartin(\xvalto{n}\andstate)=g(\xvalto{n}\andstate)$. 
\end{proof}

Thirdly, it is (essentially) the expressions in Proposition~\ref{prop:finite:horizon:gambles} that we have used in Ref.~\cite{cooman2008,hermans2012,skulj2013} for our studies of imprecise Markov chains, which we report in Section~\ref{sec:markov}.
The main result of the present paper, Theorem~\ref{thm:point-wise:ergodic:theorem} in Section~\ref{sec:perron:frobenius}, will build on the {\pflike} results proved in those papers. 

Fourthly, it was also shown in Ref.~\cite{cooman2007d} that the expressions in Proposition~\ref{prop:finite:horizon:gambles} have an interesting interpretation in terms of (precise) probability trees.
Indeed, we can associate with an imprecise probability tree a (usually infinite) collection of (so-called \emph{compatible}) precise probability trees with the same event tree, by associating with each situation $s$ in the event tree some arbitrarily chosen precise local expectation $\sitinlocal{\cdot}{s}$ that belongs to the convex closed set $\solp(\sitinllocal{\cdot}{s})$ of expectations that are compatible with the local lower expectation $\sitinllocal{\cdot}{s}$.
For any $n$-measurable gamble $f$ on $\pths$, the global precise expectations in the compatible precise probability trees will then range over a closed interval whose lower and upper bounds are given by the expressions in Proposition~\ref{prop:finite:horizon:gambles}.

And finally, Shafer and Vovk have shown~\cite[Chapter~8]{shafer2001} that when the local models are precise probability models, these formulae~\eqref{eq:global:lower:gambles} and~\eqref{eq:global:upper:gambles} lead to global models that coincide with the ones found in measure-theoretic probability theory. 
\emph{This implies that the results we will prove below, subsume, as special cases, the classical results of measure-theoretic probability theory.}

\section{A strong law of large numbers for submartingale differences}\label{sec:large:numbers}
We now discuss and prove two powerful convergence results for the processes we have defined in the previous section.

We call an event $A$ \emph{null} if $\upr(A)=\ujoint(\ind{A})=0$, and \emph{strictly null} if there is some test supermartingale $\test$ that converges to $+\infty$ on $A$, meaning that:
\begin{equation*}
\lim\test(\pth)=+\infty\quad\text{for all $\pth\in A$}.
\end{equation*}
Here, a \emph{test supermartingale} is a supermartingale with $\test(\init)=1$ that is moreover non-negative in the sense that $\test(s)\geq0$ for all situations $s\in\sits$.
Any strictly null event is null, but null events need not be strictly null~\cite{vovk2014:itip}.

\begin{proposition}\label{prop:strictly:then:null}
Any strictly null event is null, but not {\itshape vice versa}.\footnote{We infer from the proof that for the null and strictly null events to be the same, it is necessary to consider supermartingales that may assume extended real values, as is done in Refs.~\cite{vovk2014:itip,shafer2012:zero-one}. We see no need for doing so in the context of the present paper.}
\end{proposition}

\begin{proof}
Consider any event $A$. 
Recall the following expression for $\upr(A)$:
\begin{equation}\label{eq:strictly:then:null:uprob}
\upr(A)=\uex(\ind{A})
=\inf\cset{\supermartin(\init)}
{\supermartin\in\supermartins\text{ and }\liminf\supermartin\geq\ind{A}}.
\end{equation}  
Also, consider any supermartingale $\supermartin$ such that $\liminf\supermartin\geq\ind{A}$, then it follows from Lemma~\ref{lem:very:basic:inequality} and the fact that $-\supermartin$ is a submartingale that
\begin{equation}\label{eq:strictly:then:null:nonnegative}
\supermartin(\init)
\geq\inf_{\pth\in\pths}\liminf\supermartin(\pth)
\geq\inf_{\pth\in\pths}\ind{A}(\pth)
\geq0.
\end{equation}
Combined with Equation~\eqref{eq:strictly:then:null:uprob}, this implies that $\upr(A)\geq0$.\footnote{This will also follow from \ref{it:elex:more:bounds} further on.}

We are now ready for the proof.
Assume that $A$ is strictly null, so there is some test supermartingale $\test$ that converges to $+\infty$ on $A$.
Then for any $\alpha>0$, $\alpha\test$ is a supermartingale such that $\liminf(\alpha\test)\geq\ind{A}$, and therefore we infer that $0\leq\upr(A)\leq\alpha\test(\init)=\alpha$, where the second inequality follows from Equation~\eqref{eq:strictly:then:null:uprob}. 
Since this holds for all $\alpha>0$, we find that $\upr(A)=0$.

To show that not every null event is strictly null, we show that while an exact event may be null, it can never be strictly null.

First, we show that exact events may be null.
Consider any  situation $\xvalto{n+1}$, with $n\in\natswithzero$, such that $\xtoinulocal{\ind{\set{\xval{n+1}}}}{n}=0$, then we show that $\upr(\exact{\xvalto{n+1}})=0$.
Indeed, consider the real process $\supermartin$ that assumes the value $1$ in all situations that follow (or coincide with) $\xvalto{n+1}$, and $0$ elsewhere. 
Then clearly $\supermartin(\init)=0$, $\liminf\supermartin=\indexact{\xvalto{n+1}}$ and $\supermartin$ is a supermartingale because $\xtoinulocal{\supermartin(\xvalto{n}\andstate)}{n}=\xtoinulocal{\ind{\set{\xval{n+1}}}}{n}=0=\supermartin(\xvalto{n})$.
Equation~\eqref{eq:strictly:then:null:uprob} now implies that $\upr(\exact{\xvalto{n+1}})\leq0$ and therefore---since we already know that $\upr(\exact{\xvalto{n+1}})\geq0$---that $\upr(\exact{\xvalto{n+1}})=0$.

Next, if $\exact{s}$ were strictly null, there would be a test supermartingale $\supermartin$ that converges to $+\infty$ on $\exact{s}$, and therefore Lemma~\ref{lem:very:basic:inequality} and the fact that $-\supermartin$ is a submartingale would imply that $\supermartin(s)\geq\inf_{\pth\in\exact{s}}\liminf\supermartin(\pth)=+\infty$, which is impossible for the real process~$\supermartin$.
\end{proof}

In this paper, we will use the `strict' approach, and prove that events are strictly null---and therefore also null---by actually showing that there is a test supermartingale that converges to $+\infty$ there.

As usual, an inequality or equality between two variables is said to hold \emph{(strictly) almost surely} when the event that it does not hold is (strictly) null.
Shafer and Vovk~\cite{shafer2001,vovk2014:itip} have proved the following interesting result, which we will have occasion to use a few times further on.
It can be seen as a generalisation of Doob's supermartingale convergence theorem~\cite[Sections~11.5--7]{williams1991} to imprecise probability trees.
We provide its proof, adapted from Ref.~\cite{shafer2012:zero-one} to our specific definitions and assumptions, with corrections for a few tiny glitches, for the sake of completeness.

\begin{theorem}[\protect{\cite[Section~6.5]{vovk2014:itip}} Supermartingale convergence theorem]\label{thm:supermartingale:convergence}
Let\/ $\supermartin$ be a supermartingale that is bounded below.
Then $\supermartin$ converges strictly almost surely to a real variable.
\end{theorem}

\begin{proof}
Because $\supermartin$ is bounded below, we may assume without loss of generality that $\supermartin$ is non-negative and that $\supermartin(\init)=1$, as adding real constants to $\supermartin$, or multiplying it with positive real constants, does not affect its convergence properties nor---by coherence [\ref{it:lex:homo} and~\ref{it:lex:constant:additivity}] of the local models---the fact that it is a supermartingale. Hence, $\supermartin$ is a test supermartingale.
Also, because $\supermartin$ is bounded below, it cannot converge to $-\infty$ on any path.
Let $A$ be the event where $\supermartin$ converges to $+\infty$, and let $B$ be the event where it diverges.
We have to show that there is a test supermartingale that converges to $+\infty$ on $A\cup B$.
\par
Associate with any couple of rational numbers $0<a<b$ the following recursively defined sequences of cuts $\highcut{k}$ and $\lowcut{k}$.
Let $\lowcut{0}\coloneqq\set{\init}$, and for $k\in\nats$:
\begin{align}
\highcut{k}
\coloneqq&\cset{s\sfollows\lowcut{k-1}}
{\supermartin(s)>b\text{ and }(\forall t\in\group{\lowcut{k-1},s})\supermartin(t)\leq b}
\label{eq:highcuts}\\
\lowcut{k}
\coloneqq&\cset{s\sfollows\highcut{k}}
{\supermartin(s)<a\text{ and }(\forall t\in\group{\highcut{k},s})\supermartin(t)\geq a}.
\label{eq:lowcuts}
\end{align}
Consider the real process $\test[a,b]$ with the following recursive definition: 
\begin{equation}\label{eq:test:supermartin:first}
\test[a,b](\init)\coloneqq1
\text{ and }
\test[a,b](s\andstate)
\coloneqq
\begin{cases}
\test[a,b](s)+\Delta\supermartin(s)
&\text{ if $s\in\bigcup_{k\in\nats}\lsqgroup{\lowcut{k-1},\highcut{k}}$}\\
\test[a,b](s)
&\text{ otherwise}.
\end{cases}
\end{equation}
We now show that $\test[a,b]$ is a test supermartingale that converges to $+\infty$ on any path $\pth$ for which $\liminf\supermartin(\pth)<a<b<\limsup\supermartin(\pth)$.

In what follows, for any situation $s$ and for any $k\in\nats$, when $s\sfollows\highcut{k}$, we denote by $u^s_k$ the (necessarily unique) situation in $\highcut{k}$ such that $u^s_k\sprecedes s$.
Similarly, for any $k\in\natswithzero$, when $s\sfollows\lowcut{k}$, we denote by $v^s_k$ the (necessarily unique) situation in $\lowcut{k}$ such that $v^s_k\sprecedes s$; observe that $v^s_0=\init$.
Recall from Equations~\eqref{eq:highcuts} and~\eqref{eq:lowcuts} that, for all $k\in\nats$, $\supermartin(u^s_k)>b$ and $\supermartin(v^s_k)<a$.

Since it follows from Equation~\eqref{eq:test:supermartin:first} that $\Delta\test[a,b](s)$ is zero or equal to $\Delta\supermartin(s)$, it follows from coherence [\ref{it:lex:more:bounds}] and $\sitinulocal{\Delta\supermartin(s)}{s}\leq0$ that $\sitinulocal{\Delta\test[a,b](s)}{s}\leq0$ for all situations $s$, so $\test[a,b]$ is indeed a supermartingale.

To prove that $\test[a,b]$ is non-negative, we recall from Equation~\eqref{eq:test:supermartin:first} that $\test[a,b]$ can only change in situations $s\in\lsqgroup{\lowcut{k-1},\highcut{k}}$, with $k\in\nats$.
Since $\test[a,b](\init)=1$, taking into account Lemma~\ref{lem:low:high}, this means that we only have to prove that $\test[a,b](c)\geq0$ for the children $c$ of the situations $s\in\lsqgroup{\lowcut{k-1},\highcut{k}}$, with $k\in\nats$.
There are two possible cases to consider:
The first case (a) is that $s\in\lsqgroup{\init,\highcut{1}}$.
Since $\test[a,b](\init)=\supermartin(\init)=1$, we gather from Equation~\eqref{eq:test:supermartin:first} for $k=1$ that then $\test[a,b](c)=\supermartin(c)\geq0$ for all children $c$ of $s$.
The second case (b) is that $s\in\lsqgroup{\lowcut{k},\highcut{k+1}}$ for some $k\in\nats$.
We then gather from Equation~\eqref{eq:test:supermartin:first} and Lemma~\ref{lem:low:high} that for all children $c$ of $s$
\begin{align*}
\test[a,b](c)
&=\test[a,b](\init)+[\supermartin(u^s_1)-\supermartin(\init)]
+\sum_{\ell=2}^{k}[\supermartin(u^s_{\ell})-\supermartin(v^s_{\ell-1})]
+[\supermartin(c)-\supermartin(v^s_k)]\\
&\geq b+(k-1)(b-a)+\supermartin(c)-\supermartin(v^s_k)
\geq k(b-a)+\supermartin(c)
\geq k(b-a)
\geq0.
\end{align*}
We conclude that $\test[a,b]$ is indeed non-negative.

It remains to prove that $\test[a,b]$ converges to $+\infty$ on all paths $\pth$ where $\liminf\supermartin(\pth)<a<b<\limsup\supermartin(\pth)$.
By Lemma~\ref{lem:low:high}, any such path $\pth$ goes through the entire chain of cuts $\lowcut{0}\sprecedes\highcut{1}\sprecedes\lowcut{1}\sprecedes\cdots\sprecedes\highcut{n}\sprecedes\lowcut{n}\sprecedes\cdots$, meaning that for any situation $s$ on this path $\pth$, one of the following cases obtains.
The first case is that $s\in\sqgroup{\init,\highcut{1}}$.
We gather from the discussion of case (a) above that then $\test[a,b](s)=\supermartin(s)$.
The second case is that $s\in\usqgroup{\highcut{k},\lowcut{k}}$ for some $k\in\nats$.
Then we gather from Equation~\eqref{eq:test:supermartin:first} and Lemma~\ref{lem:low:high} that
\begin{equation*}
\test[a,b](s)
=\test[a,b](\init)+[\supermartin(u^s_1)-\supermartin(\init)]
+\sum_{\ell=2}^{k}[\supermartin(u^s_\ell)-\supermartin(v^s_{\ell-1})]
\geq b+(k-1)(b-a).
\end{equation*}
And the third possible case is that $s\in\usqgroup{\lowcut{k},\highcut{k+1}}$ for some $k\in\nats$.
Then we gather from the discussion of case (b) above that $\test[a,b](s)\geq k(b-a)$.
Since $b>a$, we conclude that indeed $\lim\supermartin(\pth)=+\infty$.

To finish, use the countable set of rational couples $K\coloneqq\cset{(a,b)\in\rationals^2}{0<a<b}$ to define the process $\test$ by letting $\test(\init)\coloneqq1$ and, for all $s\in\sits$, $\Delta\test(s)\coloneqq\sum_{(a,b)\in K}w^{a,b}\Delta\test[a,b](s)$, a countable convex combination of the real numbers $\Delta\test[a,b](s)$, with coefficients $w^{a,b}>0$ that sum to $1$.
Observe that
\begin{equation*}
\Delta\test(s)
=\sum_{(a,b)\in K}w^{a,b}\Delta\test[a,b](s)
=\gamma(s)\Delta\supermartin(s)\in\reals,
\end{equation*}
where $\gamma(s)\in[0,1]$, because it follows from Equation~\eqref{eq:test:supermartin:first} that for any $(a,b)\in K$, $\Delta\test[a,b](s)$ is equal to $\Delta\supermartin(s)$ or zero.
As an immediate consequence, $\test$ is a real process and $\test=\sum_{(a,b)\in K}w^{a,b}\test[a,b]$.
Clearly, $\test$ has $\test(\init)=1$, is non-negative and converges to $+\infty$ on $B$.
Moreover, since $\Delta\test(s)=\gamma(s)\Delta\supermartin(s)$, it follows from coherence [\ref{it:lex:homo:upper}] that $\sitinulocal{\Delta\test(s)}{s}=\gamma(s)\sitinulocal{\Delta\supermartin(s)}{s}\leq0$ for all $s\in\sits$, so $\test$ is  a test supermartingale.
\par
Since coherence [\ref{it:lex:sub} and~\ref{it:lex:homo:upper}] implies that a convex combination of two test supermartingales is again a test supermartingale, we conclude from all these considerations that the process $\frac{1}{2}(\supermartin+\test)$ is a test supermartingale that converges to $+\infty$ on $A\cup B$.
\end{proof}

\begin{lemma}\label{lem:low:high}
$\lowcut{k-1}\sprecedes\highcut{k}\sprecedes\lowcut{k}$ for all $k\in\nats$;
\end{lemma}

\begin{proof}
The statement follows immediately from Equations~\eqref{eq:highcuts} and~\eqref{eq:lowcuts}. 
The case $\lowcut{k-1}=\emptyset$ presents no problem, because Equation~\eqref{eq:highcuts} tells us that then $\highcut{k}=\emptyset$ as well.
Neither does the case $\highcut{k}=\emptyset$, because Equation~\eqref{eq:lowcuts} tells us that then $\lowcut{k}=\emptyset$ as well.
\end{proof}

We now turn to a very general version of the strong law of large numbers. 
Weak (as well as less general) versions of this law were proved by one of us in Refs.~\cite{cooman2004a,cooman2007d}.
It is this law that will, in Section~\ref{sec:perron:frobenius}, be used to derive our version of the point-wise ergodic theorem.
Its proof is based on a tried-and-tested method for constructing test supermartingales that goes back to an idea in Ref.~\cite[Lemma~3.3]{shafer2001}.

\begin{theorem}\label{thm:slln:submartingale:differences}
Let\/ $\submartin$ be a submartingale such that $\Delta\submartin$ is uniformly bounded and let $\predictable$ be a real process that only assumes values in $\{0,1\}$.
Then strictly almost surely:
\begin{equation*}
\lim\summed{\predictable}=+\infty
\then
\liminf\pavg{\submartin}\geq0.
\end{equation*}
\end{theorem}
\noindent
If $\predictable$ is equal to $1$ in all situations, then $\lim\summed{\predictable}=+\infty$ on all paths, so the following special case is immediate.

\begin{corollary}[Strong law of large numbers for submartingale differences]\label{cor:slln:submartingale:differences}
Let\/ $\submartin$ be a submartingale such that $\Delta\submartin$ is uniformly bounded.
Then $\liminf\avg{\submartin}\geq0$ strictly almost surely.
\end{corollary}

\begin{proof}[Proof of Theorem~\ref{thm:slln:submartingale:differences}]
Consider the events $A\coloneqq\cset{\pth\in\pths}{\liminf\pavg{\submartin}(\pth)<0}$ and $D\coloneqq\cset{\pth\in\pths}{\lim\summed{\predictable}(\pth)=+\infty}$.
We have to show that there is some test supermartingale $\test$ that converges to $+\infty$ on the set $D\cap A$.
Let\/ $B>0$ be any uniform real bound on $\Delta\submartin$, meaning that $\abs{\Delta\submartin(s)}\leq B$ for all situations $s\in\sits$. 
We can always assume that $B>1$.
\par
For any $r\in\nats$, let $A_r\coloneqq\cset{\pth\in\pths}{\liminf\pavg{\submartin}(\pth)<-\frac{1}{2^r}}$, then $A=\bigcup_{r\in\nats}A_r$.
So fix any $r\in\nats$ and consider any $\pth\in D\cap A_r$, then
\begin{equation*}
\liminf_{n\to+\infty}\pavg{\submartin}(\pth^{n})<-\frac{1}{2^r}
\end{equation*}
and therefore
\begin{equation*}
(\forall m\in\nats)
(\exists n_m\geq m)
\pavg{\submartin}(\pth^{n_m})<-\frac{1}{2^r}=-\epsilon,
\end{equation*}
with $\epsilon\coloneqq\frac{1}{2^r}>0$. 
Consider now the positive supermartingale of Lemma~\ref{lem:martinwlln}, with in particular $\xi\coloneqq\frac{\epsilon}{2B^2}=\frac{1}{2^{r+1}B^2}$.\footnote{One of the requirements in Lemma~\ref{lem:martinwlln} is that $0<\epsilon<B$, and this is satisfied because we made sure that $B>1$.} 
Denote this test supermartingale by $\process_\submartin^{(r)}$.
It follows from Lemma~\ref{lem:martinwlln} that
\begin{equation}\label{eq:slln:intermediate}
\process_\submartin^{(r)}(\pth^{n_m})
\geq\exp\group[\bigg]{\summed{\predictable}(\pth^{n_m})\frac{\epsilon^2}{4B^2}}
=\exp\group[\bigg]{\summed{\predictable}(\pth^{n_m})\frac{1}{2^{2r+2}B^2}}
\quad\text{for all $m\in\nats$}.
\end{equation}
Consider any real $R>0$ and $m\in\nats$. 
Since $\pth\in D$, we know that $\lim_{n\to+\infty}\summed{\predictable}(\pth^{n})=+\infty$, so there is some natural number $m'\geq m$ such that $\exp\group[\big]{\summed{\predictable}(\pth^{m'})\frac{1}{2^{2r+2}B^2}}>R$.
Hence it follows from the statement in \eqref{eq:slln:intermediate} that there is some $n_{m'}\geq m'\geq m$---whence $\summed{\predictable}(\pth^{n_{m'}})\geq\summed{\predictable}(\pth^{m'})$---such that
\begin{equation*}
\process_\submartin^{(r)}(\pth^{n_{m'}})
\geq\exp\group[\bigg]{\summed{\predictable}(\pth^{n_{m'}})\frac{1}{2^{2r+2}B^2}}
\geq\exp\group[\bigg]{\summed{\predictable}(\pth^{m'})\frac{1}{2^{2r+2}B^2}}
>R,
\end{equation*}
which implies that $\limsup\process_\submartin^{(r)}(\pth)=+\infty$.
Observe that for this test supermartingale, $\process_\submartin^{(r)}(\xvalto{n})\leq\group{\frac{3}{2}}^{n}$ for all $n\in\natswithzero$ and $\xvalto{n}\in\statesto{n}$.

Now define the process $\process_\submartin\coloneqq\sum_{r\in\nats}w^{(r)}\process_\submartin^{(r)}$ as a countable convex combination of the $\process_\submartin^{(r)}$ constructed above, with positive weights $w^{(r)}>0$ that sum to one.
This is a real process, because each term in the series $\process_\submartin(\xvalto{n})$ is non-negative, and moreover 
\begin{equation*}
\process_\submartin(\xvalto{n})
\leq\sum_{r\in\nats}w^{(r)}\process_\submartin^{(r)}
\leq\sum_{r\in\nats}w^{(r)}\group[\Big]{\frac{3}{2}}^{n}
=\group[\Big]{\frac{3}{2}}^{n}
\text{ for all $n\in\natswithzero$ and $\xvalto{n}\in\statesto{n}$}.
\end{equation*}
This process is also positive, has $\process_\submartin(\init)=1$, and, for any $\pth\in D\cap A$, it follows from the argumentation above that there is some $r\in\nats$ such that $\pth\in D\cap A_r$ and therefore
\begin{equation*}
\limsup\process_\submartin(\pth)
\geq w^{(r)}\limsup\process_\submartin^{(r)}(\pth)
=+\infty,
\end{equation*}
so $\limsup\process_\submartin(\pth)=+\infty$.

We now prove that $\process_\submartin$ is a supermartingale.
Consider any $n\in\natswithzero$ and any $\xvalto{n}\in\statesto{n}$, then we have to prove that $\xtoinllocal{-\Delta\process_\submartin(\xvalto{n})}{n}\geq0$.
Since it follows from the argumentation in the proof of Lemma~\ref{lem:martinwlln} that
\begin{equation*}
-\Delta\process_\submartin^{(r)}(\xvalto{n})
=\frac{1}{2^{r+1}B^2}\process_\submartin^{(r)}(\xvalto{n})\Delta\submartin(\xvalto{n})
\quad\text{for all $r\in\nats$}, 
\end{equation*}
we see that
\begin{equation*}
-\Delta\process_\submartin(\xvalto{n})
=-\sum_{r\in\nats}w^{(r)}\Delta\process_\submartin^{(r)}
=\Delta\submartin(\xvalto{n})
\underbrace{\sum_{r\in\nats}\frac{w^{(r)}}{2^{r+1}B^2}\process_\submartin^{(r)}(\xvalto{n})}_{\eqqcolon c(\xvalto{n})},
\end{equation*}
where $c(\xvalto{n})\geq0$ must be a real number, because, using a similar argument as before
\begin{equation*}
c(\xvalto{n})
=\sum_{r\in\nats}\frac{w^{(r)}}{2^{r+1}B^2}\process_\submartin^{(r)}(\xvalto{n})
\leq L\sum_{r\in\nats}w^{(r)}\process_\submartin^{(r)}(\xvalto{n})
\leq L\group[\Big]{\frac{3}{2}}^{n}
\end{equation*}
for some real $L>0$.
Therefore indeed, using the non-negative homogeneity of lower expectations [\ref{it:lex:homo}]:
\begin{equation*}
\xtoinllocal{-\Delta\process_\submartin(\xvalto{n})}{n}
=\xtoinllocal{c(\xvalto{n})\Delta\submartin(\xvalto{n})}{n}
=c(\xvalto{n})\xtoinllocal{\Delta\submartin(\xvalto{n})}{n}
\geq0,
\end{equation*}
because $\submartin$ is a submartingale.

Since we now know that $\process_\submartin$ is a supermartingale that is furthermore bounded below (by $0$) it follows from the supermartingale convergence theorem (Theorem~\ref{thm:supermartingale:convergence}) that there is some test supermartingale $\test_\submartin$ that converges to $+\infty$ on all paths where $\process_\submartin$ does not converge to a real number, and therefore in particular on all paths in $D\cap A$. 
Hence $D\cap A$ is indeed strictly null.
\end{proof}

\begin{lemma}\label{lem:martinwlln}
Consider any real $B>0$ and any $0<\xi<\frac{1}{B}$.
Let\/ $\submartin$ be any submartingale such that $\vert\Delta\submartin\vert\leq B$.
Let $\predictable$ be any real process that only assumes values in $\{0,1\}$.
Then the process $\process_\submartin$ defined by:
\begin{equation*}
\process_\submartin(\xvalto{n})
\coloneqq\prod_{k=0}^{n-1} 
\sqgroup[\big]{1-\xi\predictable(\xvalto{k})\Delta\submartin(\xvalto{k})(\xval{k+1})}
\quad\text{for all $n\in\natswithzero$ and $\xvalto{n}\in\statesto{n}$}
\end{equation*}
is a positive supermartingale with $\process_\submartin(\init)=1$, and therefore in particular a test supermartingale.
Moreover, for $\xi\coloneqq\frac{\epsilon}{2B^2}$, with $0<\epsilon<B$, we have that 
\begin{equation*}
\pavg{\submartin}(\xvalto{n})\leq-\epsilon
\then
\process_\submartin(\xvalto{n})\geq\exp\group[\bigg]{\summed{\predictable}(\xvalto{n})\frac{\epsilon^2}{4B^2}}
\quad\text{for all $n\in\natswithzero$ and $\xvalto{n}\in\statesto{n}$}.
\end{equation*}
\end{lemma}

\begin{proof}
$\process_\submartin(\init)=1$ trivially.
To prove that $\process_\submartin$ is positive, consider any $n\in\nats$ and any $\xvalto{n}\in\statesto{n}$.
Since it follows from $0<\xi B<1$, $\vert\Delta\submartin\vert\leq B$ and $\predictable\in\{0,1\}$ that $1-\xi\predictable(\xvalto{k})\Delta\submartin(\xvalto{k})(\xval{k+1})\geq 1-\xi B>0$ for all $0\leq k\leq n-1$, we see that indeed:
\begin{equation*}
\process_\submartin(\xvalto{n})
=\prod_{k=0}^{n-1}\sqgroup[\big]{1-\xi\predictable(\xvalto{k})\Delta\submartin(\xvalto{k})(\xval{k+1})}
>0.
\end{equation*}

Consider any $n\in\natswithzero$ and any $\xvalto{n}\in\statesto{n}$.
For any $\xval{n+1}\in\states{n+1}$:
\begin{align*}
-\Delta\process_\submartin(\xvalto{n})(\xval{n+1})
&=\process_\submartin(\xvalto{n})-\process_\submartin(\xvalto{n+1})\\
&=\xi\predictable(\xvalto{n})\Delta\submartin(\xvalto{n})(\xval{n+1})
\prod_{k=0}^{n-1} 
\sqgroup[\big]{1-\xi\predictable(\xvalto{k})\Delta\submartin(\xvalto{k})(\xval{k+1})}\\
&=\xi\process_\submartin(\xvalto{n})\predictable(\xvalto{n})\Delta\submartin(\xvalto{n})(\xval{n+1}),
\end{align*}
implying that $-\Delta\process_\submartin(\xvalto{n})=\xi\predictable(\xvalto{n})\process_\submartin(\xvalto{n})\Delta\submartin(\xvalto{n})$.
Since $\process_\submartin(\xvalto{n})>0$, $\xi>0$ and $\summed{\predictable}(\xvalto{n})\in\{0,1\}$, it follows directly from $\xtoinllocal{\Delta\submartin(\xvalto{n})}{n}\geq0$ and the non-negative homogeneity property [\ref{it:lex:homo}] of a lower expectation that $\xtoinllocal{-\Delta\process_\submartin(\xvalto{n})}{n}\geq0$---implying that $\process_\submartin$ is a supermartingale. 

For the second statement, consider any $0<\epsilon<B$ and let $\xi\coloneqq\frac{\epsilon}{2B^2}$. 
Then for any $n\in\natswithzero$ and $\xvalto{n}\in\statesto{n}$ such that $\pavg{\submartin}(\xvalto{n})\leq-\epsilon$, we have for all real $K$:
\begin{align}
\process_\submartin(\xvalto{n})\geq\exp(K) 
&\ifandonlyif
\prod_{k=0}^{n-1} 
\sqgroup[\big]{1-\xi\predictable(\xvalto{k})\Delta\submartin(\xvalto{k})(\xval{k+1})}\geq\exp(K)\notag\\
&\ifandonlyif
\sum_{k=0}^{n-1} 
\ln\sqgroup[\big]{1-\xi\predictable(\xvalto{k})\Delta\submartin(\xvalto{k})(\xval{k+1})}\geq K.\label{eq:K}
\end{align}
Since $\vert\Delta\submartin\vert\leq B$, $\predictable\in\{0,1\}$ and $0<\epsilon<B$, we know that
\begin{equation*}
-\xi\predictable(\xvalto{k})\Delta\submartin(\xvalto{k})
\geq-\xi B
=-\frac{\epsilon}{2B}
>-\frac{1}{2}
\quad\text{ for $0\leq k\leq n-1$}. 
\end{equation*} 
As $\ln(1+x)\ge x-x^2$ for $x>-\frac{1}{2}$, this allows us to infer that
\begin{multline*}
\sum_{k=0}^{n-1} 
\ln\sqgroup[\big]{1-\xi\predictable(\xvalto{k})\Delta\submartin(\xvalto{k})(\xval{k+1})}\\
\begin{aligned}
&\geq
\sum_{k=0}^{n-1} 
\sqgroup[\big]{-\xi\predictable(\xvalto{k})\Delta\submartin(\xvalto{k})(\xval{k+1})
-\xi^2\predictable(\xvalto{k})^2(\Delta\submartin(\xvalto{k})(\xval{k+1}))^2}\\
&=
-\xi\summed{\predictable}(\xvalto{n})\pavg{\submartin}(\xvalto{n})
-\xi^2\sum_{k=0}^{n-1} 
\predictable(\xvalto{k})(\Delta\submartin(\xvalto{k})(\xval{k+1}))^2\\
&\geq\xi\summed{\predictable}(\xvalto{n})\epsilon-\xi^2\summed{\predictable}(\xvalto{n})B^2
=\summed{\predictable}(\xvalto{n})\xi(\epsilon-\xi B^2)
=\summed{\predictable}(\xvalto{n})\frac{\epsilon^2}{4B^2},
\end{aligned}
\end{multline*}
where the first equality holds because $\predictable^2=\predictable$. 
Now choose $K\coloneqq\summed{\predictable}(\xvalto{n})\frac{\epsilon^2}{4B^2}$ in Equation~\eqref{eq:K}.
\end{proof}

\section{Properties of the global models}\label{sec:properties:of:global:models}
In this section, we first consider the extension to extended real variables of the global lower and upper expectations introduced in Section~\ref{sec:global:models}, and then prove a number of very general and useful results for these extensions. 
Indeed, for a number of results and applications, it will be useful to extend the global models, introduced in formulae~\eqref{eq:global:lower:gambles} and~\eqref{eq:global:upper:gambles}, from bounded real variables (gambles on $\pths$) to extended ones; see for example the discussion in Section~\ref{sec:transition:and:return:times}, where we discuss transition and return times, which are unbounded and may even become infinite.

Nevertheless, it should be stressed here that most of the discussion in this paper  deals only with bounded real variables.
In particular, our results on ergodic theorems in Sections~\ref{sec:interesting:equality} and~\ref{sec:perron:frobenius} do not rely on this extension.

We begin by proving alternative expressions for the global models for gambles.

\begin{proposition}\label{prop:global:models:gambles:alternative}
For any gamble $g$ on $\pths$, and any situation $s\in\sits$:
\begin{align}
\sitinlglobal{g}{s}
&=\sup\cset{\submartin(s)}
{\submartin\in\basubmartins\text{ and }\limsup\submartin(\pth)\leq g(\pth)\text{ for all $\pth\in\exact{s}$}}
\label{eq:global:lower:gambles:alternative}\\
\sitinuglobal{g}{s}
&=\inf\cset{\supermartin(s)}
{\supermartin\in\bbsupermartins\text{ and }\liminf\supermartin(\pth)\geq g(\pth)\text{ for all $\pth\in\exact{s}$}}.
\label{eq:global:upper:gambles:alternative}
\end{align}
\end{proposition}

\begin{proof}
We only give the proof for the lower expectations, as the proof for the upper expectations is completely similar.
If we denote the right-hand side in Equation~\eqref{eq:global:lower:gambles:alternative} by $\altsitinlglobal{g}{s}$, then it follows trivially from $\basubmartins\subseteq\submartins$ that $\sitinlglobal{g}{s}\geq\altsitinlglobal{g}{s}$, so we concentrate on proving the converse inequality $\sitinlglobal{g}{s}\leq\altsitinlglobal{g}{s}$.

If $\sitinlglobal{g}{s}=-\infty$ then this inequality is trivially satisfied,\footnote{Note, by the way, that it will follow from~\ref{it:elex:bounds} in Proposition~\ref{prop:properties:of:global:expectations} that this cannot actually happen.} so we may assume without loss of generality that there is some $\submartin\in\submartins$ such that $\limsup\submartin(s\andpath)\leq g(s\andpath)$.
Consider any such submartingale $\submartin$ for which also $\submartin(t)=\submartin(s)$ in any situation $t$ that does not follow $s$ [to see that such a submartingale exists, simply consider that if we change the values $\submartin'(t)$ of any submartingale $\submartin'$ to $\submartin'(s)$ in such situations $t$, the result is still a submartingale]. It then follows from Lemma~\ref{lem:very:basic:inequality} that 
\begin{equation*}
\submartin(v)\leq
\sup_{\pth\in\exact{s}}
\limsup\submartin(\pth)
\leq
\sup_{\pth\in\exact{s}}
g(\pth)
\leq
\sup g
\end{equation*}
for all situations $v$ that follow $s$, and, in particular, that $\submartin(s)\leq\sup g$. 
For any situation that does not follow $s$, this implies that $\submartin(v)=\submartin(s)\leq\sup g$. Hence, $\submartin(v)\leq\sup g$ for all $v\in\sits$. 
Since $\sup g\in\reals$ because $g$ is a gamble and therefore bounded, this implies that $\submartin\in\basubmartins$, and the proof is complete.
\end{proof}

If we now simply replace the gambles `$g$' in Equations~\eqref{eq:global:lower:gambles} and~\eqref{eq:global:lower:gambles:alternative} by extended real variables `$f$', we get two obvious candidate definitions for the conditional lower expectation $\sitinlglobal{f}{s}$ of such $f$. 
The following example shows that the first candidate, which seems to be the one suggested by Shafer and Vovk in their earlier work \cite[Chapter~8.3]{shafer2001}, may have a rather undesirable property.

\begin{example}\label{ex:undesirable:behaviour}
Consider the precise probability tree that corresponds to repeatedly flipping a fair coin, where all coin flips are independent.
That is, let $\states{k}\coloneqq\{0,1\}$ for all $k\in\nats$ and let
\begin{equation}\label{eq:faircointree}
\xtoinllocal{h}{n}
=\xtoinulocal{h}{n}
\coloneqq\frac{1}{2}h(0)+\frac{1}{2}h(1)
\text{ for all $n\in\natswithzero$, $h\in\gambleson{n+1}$ and $\xvalto{n}\in\statesto{n}$.}
\end{equation}
For any real $\alpha>0$, we consider a corresponding real process $\submartin_\alpha$, defined by $\submartin_\alpha(\init)\coloneqq\alpha$, $\Delta\submartin_\alpha(\init)\coloneqq 2\alpha(\indsing{0}-\indsing{1})$ and, for all $n\in\nats$ and $\xvalto{n}\in\statesto{n}$:
\begin{equation*}
\Delta\submartin_\alpha(\xvalto{n})\coloneqq
\begin{cases}
3\alpha 2^{n-1}(\indsing{0}-\indsing{1})
&\text{ if $x_k=0$ for all $k\in\{1,\dots,n\}$}\\
\alpha 2^{n-1}(\indsing{0}-\indsing{1})
&\text{ if $x_k=1$ for all $k\in\{1,\dots,n\}$}\\
0
&\text{ otherwise.}
\end{cases}
\end{equation*}
It follows trivially from Equation~\eqref{eq:faircointree} that this real process is both a sub- and a supermartingale, and therefore also a martingale.
For any given natural $n\geq2$ and $\xvalto{n}\in\statesto{n}$, we now set out to find a closed-form expression for $\submartin_\alpha(\xvalto{n})$. 
We consider two cases: $x_1=0$ and $x_1=1$. If $x_1=0$, then there is at least one $i\in\{1,\dots,n\}$ such that $x_k=0$ for all $k\in\{1,\dots,i\}$.
Let $i_{\max}$ be the largest such $i\in\{1,\dots,n\}$, and let $i^*\coloneqq\min\{i_{\max},n-1\}$. 
Then 
\begin{align*}
\submartin_\alpha(\xvalto{n})
&=\submartin_\alpha(\init)
+\sum_{k=0}^{n-1}\Delta\submartin_\alpha(\xvalto{k})(\xval{k+1})\\
&=\submartin_\alpha(\init)
+\Delta\submartin_\alpha(\init)(0)
+\sum_{k=1}^{i^*-1}\Delta\submartin_\alpha(\xvalto{k})(\xval{k+1})
+\Delta\submartin_\alpha(\xvalto{i^*})(\xval{i^*+1})\\
&=\alpha
+2\alpha
+\sum_{k=1}^{i^*-1}3\alpha2^{k-1}
+3\alpha2^{i^*-1}\big(\indsing{0}(\xval{i^*+1})-\indsing{1}(\xval{i^*+1})\big)\\
&=3\alpha2^{i^*-1}+3\alpha2^{i^*-1}\big(\indsing{0}(\xval{i^*+1})-\indsing{1}(\xval{i^*+1})\big)
=3\alpha2^{i^*}\indsing{0}(\xval{i^*+1}).
\end{align*}
If $x_1=1$, then there is at least one $j\in\{1,\dots,n\}$ such that $x_k=1$ for all $k\in\{1,\dots,j\}$. 
Let $j_{\max}$ be the largest such $j\in\{1,\dots,n\}$, and let $j^*\coloneqq\min\{j_{\max},n-1\}$. 
Then, using an argument similar to the one for the case $x_1=0$, we find that
\begin{align*}
\submartin_\alpha(\xvalto{n})
&=\alpha
-2\alpha
-\sum_{k=1}^{j^*-1}\alpha2^{k-1}
+\alpha2^{j^*-1}\big(\indsing{0}(\xval{j^*+1})-\indsing{1}(\xval{j^*+1})\big)\\
&=-\alpha2^{j^*-1}+\alpha2^{j^*-1}\big(\indsing{0}(\xval{j^*+1})-\indsing{1}(\xval{j^*+1})\big)
=-\alpha2^{j^*}\indsing{1}(\xval{j^*+1}).
\end{align*}
Since $x_{i^*+1}=0$ if and only if $i_{\max}=n$, and similarly, $x_{j^*+1}=1$ if and only if $j_{\max}=n$, we can combine the two cases above to find that, for all natural $n\geq2$ and $\xvalto{n}\in\statesto{n}$:
\begin{equation*}
\submartin_\alpha(\xvalto{n})
=
\begin{cases}
3\alpha2^{n-1}
&\text{ if $x_k=0$ for all $k\in\{1,\dots,n\}$}\\
-\alpha2^{n-1}
&\text{ if $x_k=1$ for all $k\in\{1,\dots,n\}$}\\
0
&\text{ otherwise.}
\end{cases}
\end{equation*}
Now let $f$ be the extended real variable that is defined by
\begin{equation*}
f(\pth)\coloneqq
\begin{cases}
+\infty & \text{if $\pth=000000 \dots$} \\
-\infty & \text{if $\pth=111111 \dots$} \\
0 & \text{otherwise}
\end{cases}
\qquad\text{for all $\pth\in\pths$.}
\end{equation*}
Then clearly $\liminf\submartin_\alpha=\limsup\submartin_\alpha=\lim\submartin_\alpha=f$. 

We conclude from all of the above that, for any $\alpha>0$, we can construct a submartingale $\submartin_\alpha\in\submartins$ such that $\submartin_\alpha(\init)=\alpha$ and $\limsup\submartin_\alpha\leq f$.
Therefore, if we were to apply Equations~\eqref{eq:global:lower:gambles} and~\eqref{eq:global:upper:gambles} to $f$, with $s=\init$, we would find that
\begin{equation*}
\ljoint(f)
=
\sup\cset{\submartin(\init)}
{\submartin\in\submartins\text{ and }\limsup\submartin\leq f}=+\infty
\end{equation*}
and that $\ujoint(f)=-\ljoint(-f)=-\ljoint(f)=-\infty$, where the first equality follows from conjugacy and the second equality follows from a symmetry argument: $\ljoint(-f)$ is equal to $\ljoint(f)$ because exchanging zeroes and ones in the tree (a) turns $f$ into $-f$, and (b) leaves the probability tree unchanged. 
We conclude that if we were to apply Equations~\eqref{eq:global:lower:gambles} and~\eqref{eq:global:upper:gambles} to the extended real variable $f$, we would find that $+\infty=\ljoint(f)>\ujoint(f)=-\infty$. 
We consider this to be undesirable: any reasonable definition of lower and upper expectation should at the very least guarantee that a lower expectation can never exceed the corresponding upper expectation.\qed
\end{example}

This leaves us with the second candidate formula for extension, which is the one we will use in this paper, and which is related to the one used in more recent work by Shafer and Vovk~\cite[Section~2]{shafer2012:zero-one}:
\begin{align}
\sitinlglobal{f}{s}
\coloneqq&\sup\cset{\submartin(s)}
{\submartin\in\basubmartins\text{ and }\limsup\submartin(\pth)\leq f(\pth)\text{ for all $\pth\in\exact{s}$}}
\label{eq:global:lower:extended}\\
\sitinuglobal{f}{s}
\coloneqq&\inf\cset{\supermartin(s)}
{\supermartin\in\bbsupermartins\text{ and }\liminf\supermartin(\pth)\geq f(\pth)\text{ for all $\pth\in\exact{s}$}}
\label{eq:global:upper:extended}\\
=&-\sitinlglobal{-f}{s},\notag
\end{align}
where $f$ is any extended real variable, and $s\in\sits$ any situation.
We will see further on in Proposition~\ref{prop:properties:of:global:expectations} [in particular~\ref{it:elex:more:bounds}] that this definition does not lead to the undesirable behaviour that Example~\ref{ex:undesirable:behaviour} warns us about.

To investigate the properties of these extended global models, we first look at their behaviour on $n$-measurable extended real variables.

\begin{proposition}\label{prop:finite:horizon}
For any situation $\xvalto{m}\in\sits$ and any $n$-measurable extended real variable~$f$, with $n,m\in\natswithzero$ such that $n\geq m$:\footnote{In these expressions, $\submartins$ may be replaced by $\basubmartins$, and $\supermartins$ by $\bbsupermartins$, the proof remains essentially the same.}
\begin{align*}
\xtoinlglobal{f}{m}
&=\sup\cset{\submartin(\xvalto{m})}
{\submartin\in\submartins
\text{ and }
(\forall\xvalfromto{m+1}{n}\in\statesfromto{m+1}{n})\submartin(\xvalto{n})\leq f(\xvalto{n})}\\
\xtoinuglobal{f}{m}
&=\inf\cset{\supermartin(\xvalto{m})}
{\supermartin\in\supermartins
\text{ and }
(\forall\xvalfromto{m+1}{n}\in\statesfromto{m+1}{n})\supermartin(\xvalto{n})\geq f(\xvalto{n})}.
\end{align*}
\end{proposition}

\begin{proof}
We sketch the idea of the proof of the equality for the lower expectations; the proof for the upper expectations is completely similar.
Denote, for simplicity of notation, the right-hand side of the first equality by $R$.

First, consider any submartingale $\submartin$ such that $\submartin(\xvalto{n})\leq f(\xvalto{n})$ for all $\xvalfromto{m+1}{n}\in\statesfromto{m+1}{n}$.
Consider the submartingale $\submartin'$ derived from $\submartin$ by keeping it constant as soon as any situation in $\set{\xvalto{m}}\times\statesfromto{m+1}{n}$ is reached and letting $\submartin'(t)=\submartin(s)$ for any situation $t$ that does not follow $s$, then clearly $\submartin'$ is bounded above, $\limsup\submartin'(\pth)\leq f(\pth)$ for all $\pth\in\exact{\xvalto{m}}$, and $\submartin(\xvalto{m})=\submartin'(\xvalto{m})$.
Hence it follows from Equation~\eqref{eq:global:lower:extended} that $\submartin(\xvalto{m})\leq\xtoinlglobal{f}{m}$, whence also $R\leq\xtoinlglobal{f}{m}$.

For the converse inequality, consider any bounded above submartingale $\submartin$ such that $\limsup\submartin(\pth)\leq f(\pth)$ for all $\pth\in\exact{\xvalto{m}}$.
Fix any $\xvalfromto{m+1}{n}$, then it follows from the $n$-measurability of $f$ that $\limsup\submartin(\pth)\leq f(\xvalto{n})$ for all $\pth\in\exact{\xvalto{n}}$, whence 
\begin{equation*}
\submartin(\xvalto{n})
\leq\sup_{\pth\in\exact{\xvalto{n}}}
\limsup\submartin(\pth)
\leq f(\xvalto{n}),
\end{equation*}
where the first inequality follows from Lemma~\ref{lem:very:basic:inequality} with $s\coloneqq\xvalto{n}$.
This implies that $\submartin(\xvalto{m})\leq R$, and therefore also $\xtoinlglobal{f}{m}\leq R$.
\end{proof}

\begin{corollary}\label{cor:finite:horizon}
For any situation $\xvalto{m}\in\sits$ and any $n$-measurable extended real variable~$f$, with $n,m\in\natswithzero$ such that $n\geq m$:
\begin{align*}
\xtoinlglobal{f}{m}
&=\sup\cset{\xtoinlglobal{g}{m}}
{g\in\gamblesto{n}
\text{ and }
(\forall\xvalfromto{m+1}{n}\in\statesfromto{m+1}{n})g(\xvalto{n})\leq f(\xvalto{n})}\\
\xtoinuglobal{f}{m}
&=\inf\cset{\xtoinuglobal{g}{m}}
{g\in\gamblesto{n}
\text{ and }
(\forall\xvalfromto{m+1}{n}\in\statesfromto{m+1}{n})g(\xvalto{n})\geq f(\xvalto{n})}.
\end{align*}
\end{corollary}

\begin{proof}
We give the proof for the lower expectations; the proof for the upper expectations is completely similar.
Denote the right-hand side of the first equality by $R$, for notational simplicity. 
It follows at once from Proposition~\ref{prop:finite:horizon} that $\xtoinlglobal{g}{m}\leq\xtoinlglobal{f}{m}$ for all $g\in\gamblesto{n}$ such that $g(\xvalto{m},\xvalfromto{m+1}{n})\leq f(\xvalto{m},\xvalfromto{m+1}{n})$ for all $\xvalfromto{m+1}{n}\in\statesfromto{m+1}{n}$, and therefore also $R\leq\xtoinlglobal{f}{m}$.
Conversely, consider any submartingale $\submartin$ such that $\submartin(\xvalto{m},\xvalfromto{m+1}{n})\leq f(\xvalto{m},\xvalfromto{m+1}{n})$ for all $\xvalfromto{m+1}{n}\in\statesfromto{m+1}{n}$.
If we define the $n$-measurable gamble $g$ on $\pths$ by letting $g(\xvalto{n})\coloneqq\submartin(\xvalto{n})$ for all $\xvalto{n}\in\statesto{n}$, then it follows from Proposition~\ref{prop:finite:horizon} that $\submartin(\xvalto{m})\leq\xtoinlglobal{g}{m}$, and since by assumption $g(\xvalto{m},\xvalfromto{m+1}{n})\leq f(\xvalto{m},\xvalfromto{m+1}{n})$ for all $\xvalfromto{m+1}{n}\in\statesfromto{m+1}{n}$, that $\xtoinlglobal{g}{m}\leq R$.
Hence $\submartin(\xvalto{m})\leq R$, and therefore, by Proposition~\ref{prop:finite:horizon}, $\xtoinlglobal{f}{m}\leq R$.
\end{proof}

The following result extends Corollary~\ref{cor:global:and:local:gambles}. 

\begin{corollary}\label{cor:global:and:local}
Consider any $n\in\natswithzero$, any situation $\xvalto{n}\in\sits$ and any $(n+1)$-measurable extended real variable $f$. 
Then
\begin{align*}
\xtoinlglobal{f}{n}
&=\sup\cset{\xtoinllocal{h}{n}}{h\in\gambles\text{ and }h\leq f(\xvalto{n}\andstate)}\\
\xtoinuglobal{f}{n}
&=\inf\cset{\xtoinulocal{h}{n}}{h\in\gambles\text{ and }h\geq f(\xvalto{n}\andstate)}.
\end{align*}
\end{corollary}

\begin{proof}
We give the proof for the lower expectation; the proof for the upper expectation is completely similar.
We infer from Proposition~\ref{prop:finite:horizon}, the argumentation above and the definition of a submartingale that, indeed,
\begin{align*}
\xtoinlglobal{f}{n}
&=\sup\cset{\alpha\in\reals}{\alpha+h\leq f(\xvalto{n}\andstate)\text{ for some $h\in\gambles$ such that $\xtoinllocal{h}{n}\geq0$}}\\
&=\sup\cset{\alpha\in\reals}{h\leq f(\xvalto{n}\andstate)\text{ for some $h\in\gambles$ such that $\alpha\leq\xtoinllocal{h}{n}$}}\\
&=\sup\cset{\xtoinllocal{h}{n}}{h\in\gambles\text{ and }h\leq f(\xvalto{n}\andstate)},
\end{align*}
where the second equality follows from coherence [\ref{it:lex:constant:additivity}].
\end{proof}

We end this section by proving three interesting and very useful results about the global models.
The first summarises and extends properties first proved by Shafer and Vovk (see for instance Refs.~\cite[Chapter~8.3]{shafer2001}, \cite[Section~2]{shafer2012:zero-one} and~\cite[Section 6.3]{vovk2014:itip}), in showing that these global models satisfy properties that extend the basic coherence axioms/properties \ref{it:lex:bounds}--\ref{it:lex:constant:additivity} for lower and upper expectations from gambles to extended real maps.
We provide, for the sake of completeness, proofs that are very close to the ones given by Shafer and Vovk \cite[Section~2]{shafer2012:zero-one}.\footnote{Our proof of~\ref{it:elex:more:bounds} corrects a small glitch in theirs. Our monotonicity property~\ref{it:elex:monotone} is stronger, because it only requires strictly almost sure, rather than point-wise, dominance.}

\begin{proposition}\label{prop:properties:of:global:expectations}
Consider any situation $s$, any extended real variables $f$ and $g$, and any real numbers $\lambda\geq0$ and $\mu$.
Then 
\begin{enumerate}[label=\upshape LE*\arabic*.,ref=\upshape LE*\arabic*,leftmargin=*]
\item\label{it:elex:bounds} $\sitinlglobal{f}{s}\geq\inf\cset{f(\pth)}{\pth\in\exact{s}}$;
\item\label{it:elex:super} $\sitinlglobal{f+g}{s}\geq\sitinlglobal{f}{s}+\sitinlglobal{g}{s}$;
\item\label{it:elex:homo} $\sitinlglobal{\lambda f}{s}=\lambda\sitinlglobal{f}{s}$;
\item\label{it:elex:monotone} if $f\leq g$ on $\exact{s}$ strictly almost surely, then $\sitinlglobal{f}{s}\leq\sitinlglobal{g}{s}$ and $\sitinuglobal{f}{s}\leq\sitinuglobal{g}{s}$; as a consequence, if $f=g$ on $\exact{s}$ strictly almost surely, then $\sitinlglobal{f}{s}=\sitinlglobal{g}{s}$ and $\sitinuglobal{f}{s}=\sitinuglobal{g}{s}$;
\item\label{it:elex:more:bounds} $\inf\cset{f(\pth)}{\pth\in\exact{s}}\leq\sitinlglobal{f}{s}\leq\sitinuglobal{f}{s}\leq\sup\cset{f(\pth)}{\pth\in\exact{s}}$;
\item\label{it:elex:constant:additivity} $\sitinlglobal{f+\mu}{s}=\sitinlglobal{f}{s}+\mu$ and $\sitinuglobal{f+\mu}{s}=\sitinuglobal{f}{s}+\mu$.
\end{enumerate}
\end{proposition}
\noindent
In these expressions, as well as further on, we use the convention that $\infty+\infty=\infty$, $-\infty+(-\infty)=-\infty$, $-\infty+\infty=\infty+(-\infty)=-\infty$, $a+\infty=\infty+a=\infty$, $a+(-\infty)=-\infty+a=-\infty$ for all real $a$, and $0\cdot\pm\infty=\pm\infty\cdot0=0$.\footnote{This is the extended addition that is convenient for working with lower expectations; for the dual upper expectations, we need to introduce a dual operator, defined by $a\upplus b\coloneqq-[(-a)+(-b)]$ for all extended real $a$ and $b$.}

\begin{proof}
\ref{it:elex:bounds}.
If $\inf\cset{f(\pth)}{\pth\in\exact{s}}=-\infty$, then the inequality is trivially satisfied.
Consider therefore any real $L\leq\inf\cset{f(\pth)}{\pth\in\exact{s}}$, and the submartingale $\submartin$ that assumes the constant value $L$ everywhere. 
Then surely $\submartin$ is bounded above, $\limsup\submartin(s\andpath)=L\leq f(s\andpath)$ and $\submartin(s)=L$, so Equation~\eqref{eq:global:lower:extended} guarantees that indeed $L\leq\sitinlglobal{f}{s}$.

\ref{it:elex:super}.
When $\sitinlglobal{f}{s}$ or $\sitinlglobal{g}{s}$ are equal to $-\infty$, so is their sum, and the inequality holds trivially.
Assume therefore that both $\sitinlglobal{f}{s}>-\infty$ and $\sitinlglobal{g}{s}>-\infty$.
This implies that there are bounded above submartingales $\submartin_1$ and $\submartin_2$ such that $\limsup\submartin_1(s\andpath)\leq f(s\andpath)$ and $\limsup\submartin_2(s\andpath)\leq g(s\andpath)$.
Consider any such submartingales $\submartin_1$ and $\submartin_2$, then it follows from the coherence [\ref{it:lex:super}] of the local models that $\submartin\coloneqq\submartin_1+\submartin_2$ is a bounded above submartingale as well.
Since $\limsup\submartin(s\andpath)\leq\limsup\submartin_1(s\andpath)+\limsup\submartin_2(s\andpath)\leq f(s\andpath)+g(s\andpath)$,\footnote{The first inequality holds for bounded above submartingales, but may fail for more general ones. Indeed, assume that on some path $\pth$, $\submartin_1(\pth^n)=2n$ and $\submartin_2(\pth^n)=-n$. Then $\limsup\submartin_1(\pth)=+\infty$, $\limsup\submartin_2(\pth)=-\infty$, and $\limsup[\submartin_1(\pth)+\submartin_2(\pth)]=+\infty$, so the inequality is violated.} we infer from Equation~\eqref{eq:global:lower:extended} that indeed $\sitinlglobal{f+g}{s}\geq\sitinlglobal{f}{s}+\sitinlglobal{g}{s}$.

\ref{it:elex:homo}.
For $\lambda>0$, simply observe that if $\submartin$ is a bounded above submartingale such that $\limsup\submartin(s\andpath)\leq f(s\andpath)$, then also $\lambda\submartin$ is a bounded above submartingale such that $\limsup[\lambda\submartin(s\andpath)]\leq\lambda f(s\andpath)$, and \emph{vice versa}. 
For $\lambda=0$, we infer on the one hand from~\ref{it:elex:bounds} and Lemma~\ref{lem:very:basic:inequality} that $\sitinlglobal{\lambda f}{s}=\sitinlglobal{0}{s}=0$, and on the other hand we also know that $0\cdot\sitinlglobal{f}{s}=0$.

\ref{it:elex:monotone}.
Due to conjugacy, it suffices to prove the first inequality.
It is trivially satisfied if $\sitinlglobal{f}{s}=-\infty$.
Assume therefore that $\sitinlglobal{f}{s}>-\infty$, meaning that there is some bounded above submartingale $\submartin$ such that $\limsup\submartin(s\andpath)\leq f(s\andpath)$.
Consider any such submartingale $\submartin$ and any real $\epsilon>0$.
It follows from the assumption and Theorem~\ref{thm:supermartingale:convergence} that there is some test supermartingale $\test\geq0$ with $\test(\init)=1$ that converges to $+\infty$ on all paths $\pth\in\exact{s}$ for which $f(\pth)>g(\pth)$.
If we let $\submartin'\coloneqq\submartin-\epsilon\test$, then $\submartin'$ is a bounded above submartingale, $\submartin'\leq\submartin$ and $\submartin'(s)=\submartin(s)-\epsilon\test(s)$.
Moreover, for any $\pth\in\exact{s}$, $\limsup\submartin'(\pth)=-\infty\leq g(\pth)$ if $g(\pth)<f(\pth)$, and $\limsup\submartin'(\pth)\leq\limsup\submartin(\pth)\leq f(\pth)\leq g(\pth)$ otherwise.
Hence $\limsup\submartin'(s\andpath)\leq g(s\andpath)$, so we infer from Equation~\eqref{eq:global:lower:extended} that $\submartin(s)-\epsilon\test(s)\leq\sitinlglobal{g}{s}$, and therefore also $\sitinlglobal{f}{s}-\epsilon\test(s)\leq\sitinlglobal{g}{s}$.
Since this inequality holds for all $\epsilon>0$, we find that indeed $\sitinlglobal{f}{s}\leq\sitinlglobal{g}{s}$.

\ref{it:elex:more:bounds}.
Suppose {\itshape ex absurdo} that $\sitinlglobal{f}{s}>\sitinuglobal{f}{s}=-\sitinlglobal{-f}{s}$.
This implies that also $\sitinlglobal{f}{s}+\sitinlglobal{-f}{s}>0$, but then \ref{it:elex:super} tells us that also $\sitinlglobal{f+(-f)}{s}>0$.
Now the extended real map $f+(-f)$ assumes only the values $0$ and $-\infty$, and therefore $f+(-f)\leq0$, so we infer from~\ref{it:elex:monotone} that $\sitinlglobal{f+(-f)}{s}\leq\sitinlglobal{0}{s}=0$, where the last equality follows from~\ref{it:elex:homo}.
This is a contradiction.
The remaining inequalities are now trivial.

\ref{it:elex:constant:additivity}.
Due to conjugacy, it suffices to prove the first equality.
If $\submartin$ is a bounded above submartingale such that $\limsup\submartin(s\andpath)\leq f(s\andpath)+\mu$, then $\submartin-\mu$ is a bounded above submartingale such that $\limsup[\submartin(s\andpath)-\mu]\leq f(s\andpath)$, and \emph{vice versa}. 
\end{proof}

Our second result follows immediately from the definition of the global lower expectations in Equation~\eqref{eq:global:lower:extended}. 

\begin{proposition}\label{prop:plug:in:the:conditional}
Consider any extended real variable~$f$ and any $n\in\natswithzero$. 
Then
\begin{equation*}
\xtoinlglobal{f(\state{1}\state{2}\dots)}{n}
=\xtoinlglobal{f(\xvalto{n}\state{n+1}\state{n+2}\dots)}{n}.
\end{equation*}
\end{proposition}

Our third result can be seen as a generalisation of the law of iterated expectations---or the law of total probability in expectation form---in classical probability theory.
Our formulation generalises a result by Shafer and Vovk~\cite[Proposition~8.7]{shafer2001}, whose proof can only be guaranteed to work for bounded real variables; we provide a proof that is better suited for dealing with extended real variables.
In accordance with the notational convention introduced in Section~\ref{sec:basic:iptrees}, we denote by $\statetoinlglobal{f}{n}$ the extended real variable that assumes the same value $\xtoinlglobal{f}{n}$ on all paths $\pth=\xvalto{n}\andpath$ that go through $\xvalto{n}$.
It is clearly $n$-measurable, and therefore only depends on the values of the first $n$ states $\stateto{n}$.

\begin{theorem}[Law of iterated lower expectations]\label{thm:law:of:iterated:expectations:general}
Consider any extended real variable~$f$ and any $n,m\in\natswithzero$ such that $n\geq m$. 
Then
\begin{equation*}
\statetoinlglobal{f}{m}
=\statetoinlglobal{\statetoinlglobal{f}{n}}{m}.
\end{equation*}
\end{theorem}

\begin{proof}
Fix any $\zvalto{m}\in\statesto{m}$. 
We prove that $\ztoinlglobal{f}{m}=\ztoinlglobal{\statetoinlglobal{f}{n}}{m}$, or equivalently, by Proposition~\ref{prop:plug:in:the:conditional}, that $\ztoinlglobal{f}{m}=\ztoinlglobal{\zstatetoinlglobal{f}{m}{n}}{m}$.

First, consider any bounded above submartingale $\submartin$ such that $\limsup\submartin(\zvalto{m}\andpath)\leq f(\zvalto{m}\andpath)$.
Then also, for any $\xvalfromto{m+1}{n}\in\statesfromto{m+1}{n}$, $\limsup\submartin(\zvalto{m}\xvalfromto{m+1}{n}\andpath)\leq g(\zvalto{m}\xvalfromto{m+1}{n}\andpath)$, which implies that, by Equation~\eqref{eq:global:lower:extended} for $s\coloneqq\zvalto{m}\xvalfromto{m+1}{n}$, $\submartin(\zvalto{m}\xvalfromto{m+1}{n})\leq\zxtoinlglobal{f}{m}{n}$. 
Hence $\submartin(\zstateto{m}{n})\leq\zstatetoinlglobal{f}{m}{n}$, and therefore we can infer from Proposition~\ref{prop:properties:of:global:expectations} [\ref{it:elex:monotone} for $s\coloneqq\zvalto{m}$] that $\ztoinlglobal{\submartin(\zstateto{m}{n})}{m}\leq\ztoinlglobal{\zstatetoinlglobal{f}{m}{n}}{m}$.
Since it follows almost trivially from Proposition~\ref{prop:finite:horizon} that $\submartin(\zvalto{m})\leq\ztoinlglobal{\submartin(\zstateto{m}{n})}{m}$, this also implies that $\submartin(\zvalto{m})\leq\ztoinlglobal{\zstatetoinlglobal{f}{m}{n}}{m}$. 
If we now use Equation~\eqref{eq:global:lower:extended} for $s\coloneqq\zvalto{m}$, we find that $\ztoinlglobal{f}{m}\leq\ztoinlglobal{\zstatetoinlglobal{f}{m}{n}}{m}$.

To prove the converse inequality, consider any $h\in\gamblesto{n}$ such that $h(\zvalto{m}\xvalfromto{m+1}{n})\leq\zxtoinlglobal{f}{m}{n}$ for all $\xvalfromto{m+1}{n}\in\statesfromto{m+1}{n}$. 
Fix any $\epsilon>0$. 
It then follows from Equation~\eqref{eq:global:lower:extended} [with $s\coloneqq\zvalto{m}\xvalfromto{m+1}{n}$] that, for any $\xvalfromto{m+1}{n}\in\statesfromto{m+1}{n}$, there is some bounded above submartingale $\submartin_{\xvalfromto{m+1}{n}}$ such that
\begin{multline*}
\submartin_{\xvalfromto{m+1}{n}}(\zvalto{m}\xvalfromto{m+1}{n})
\geq h(\zvalto{n}\xvalfromto{m+1}{n})-\frac{\epsilon}{2}\\
\text{ and }
\limsup\submartin_{\xvalfromto{m+1}{n}}(\zvalto{m}\xvalfromto{m+1}{n}\andpath)
\leq f(\zvalto{m}\xvalfromto{m+1}{n}\andpath).
\end{multline*}
Now consider any $n$-measurable real variable $g$ such that
\begin{equation*}
g(\zvalto{m}\xvalfromto{m+1}{n})
=\submartin_{\xvalfromto{m+1}{n}}(\zvalto{m}\xvalfromto{m+1}{n})
\geq h(\zvalto{m}\xvalfromto{m+1}{n})-\frac{\epsilon}{2}
\text{ for all $\xvalfromto{m+1}{n}\in\statesfromto{m+1}{n}$},
\end{equation*}
then it follows from Proposition~\ref{prop:finite:horizon} that there is a submartingale $\submartin'$ such that $\submartin'(\zvalto{m})>\ztoinlglobal{g}{m}-\frac{\epsilon}{2}$ and $\submartin'(\zvalto{m}\xvalfromto{m+1}{n})\leq g(\zvalto{m}\xvalfromto{m+1}{n})$ for all $\xvalfromto{m+1}{n}\in\statesfromto{m+1}{n}$.  
Now consider a submartingale $\submartin$ that assumes the constant value $\submartin'(\zvalto{m})$ in all situations $t$ that do not strictly follow $\zvalto{m}$---so $\submartin(t)=\submartin'(\zvalto{m})$ for all $t\in\sits$ such that $\zvalto{m}\not\sprecedes t$ and, in particular, $\submartin(\zvalto{m})=\submartin'(\zvalto{m})$---and such that moreover
\begin{multline*}
\Delta\submartin(\zvalto{m}\xvalfromto{m+1}{k})
=
\begin{cases}
\Delta\submartin'(\zvalto{m}\xvalfromto{m+1}{k})
&\text{ if $k<n$}\\
\Delta\submartin_{\xvalfromto{m+1}{n}}(\zvalto{m}\xvalfromto{m+1}{k})
&\text{ if $k\geq n$}
\end{cases}\\
\quad\text{for all $k\geq m$ and $\xvalfromto{m+1}{k}\in\statesfromto{m+1}{k}$.}
\end{multline*}
It then follows that $\submartin(\zvalto{m}\xvalfromto{m+1}{k})\leq\submartin_{\xvalfromto{m+1}{n}}(\zvalto{m}\xvalfromto{m+1}{k})$ for all $k\geq n$ and $\xvalfromto{m+1}{k}\in\statesfromto{m+1}{k}$ and, therefore, we find that $\submartin$ is bounded above and that $\limsup\submartin(\zvalto{m}\xvalfromto{m+1}{n}\andpath)\leq f(\zvalto{m}\xvalfromto{m+1}{n}\andpath)$ for all $\xvalfromto{m+1}{n}\in\statesfromto{m+1}{n}$, which implies that $\limsup\submartin(\zvalto{m}\andpath)\leq f(\zvalto{m}\andpath)$, and therefore also that $\ztoinlglobal{f}{m}\geq\submartin(\zvalto{m})$, by applying Equation~\eqref{eq:global:lower:extended} for $s=\zvalto{m}$.
Since also $\submartin(\zvalto{m})=\submartin'(\zvalto{m})>\ztoinlglobal{g}{m}-\frac{\epsilon}{2}$, we find that $\ztoinlglobal{f}{m}>\ztoinlglobal{g}{m}-\frac{\epsilon}{2}$. 
Furthermore, since $g(\zvalto{m}\statefromto{m+1}{n})\geq h(\zvalto{m}\statefromto{m+1}{n})-\frac{\epsilon}{2}$, it follows from Proposition~\ref{prop:plug:in:the:conditional} and~\ref{it:elex:monotone} that $\ztoinlglobal{g}{m}=\ztoinlglobal{g(\zvalto{m}\statefromto{m+1}{n})}{m}\geq\ztoinlglobal{h(\zvalto{m}\statefromto{m+1}{n})-\frac{\epsilon}{2}}{m}=\ztoinlglobal{h-\frac{\epsilon}{2}}{m}$, which, due to~\ref{it:elex:constant:additivity}, implies that $\ztoinlglobal{g}{m}\geq\ztoinlglobal{h}{m}-\frac{\epsilon}{2}$. 
Hence, we find that $\ztoinlglobal{f}{m}>\ztoinlglobal{h}{m}-\epsilon$.
Since this holds for any $\epsilon>0$, we find that $\ztoinlglobal{f}{m}\geq\ztoinlglobal{h}{m}$. 
Since this holds for any $h\in\gamblesto{n}$ such that $h(\zvalto{m}\xvalfromto{m+1}{n})\leq\zxtoinlglobal{f}{m}{n}$ for all $\xvalfromto{m+1}{n}\in\statesfromto{m+1}{n}$, it follows from Corollary~\ref{cor:finite:horizon} that $\ztoinlglobal{f}{m}\geq\ztoinlglobal{\zstatetoinlglobal{f}{m}{n}}{m}$.
\end{proof}

\section{Imprecise Markov chains}\label{sec:markov}
We are now ready to apply what we have learned in the previous sections to the special case of (time-homogeneous) imprecise Markov chains.
These are imprecise probability trees where (i) all states $\state{k}$ assume values in the same \emph{finite} set $\states{k}=\stateset$, called the \emph{state space}, and (ii) all local uncertainty models satisfy the so-called (time-homogeneous) \emph{Markov condition}:
\begin{equation}\label{eq:markov:condition}
\xtoinllocal{\cdot}{n}=\ltransition{\cdot}{\xval{n}}
\quad\text{for all situations $\xvalto{n}\in\sits$},
\end{equation}
meaning that these local models only depend on the last observed state; see Figure~\ref{fig:markovtree}.

\begin{figure}[ht]
\centering\footnotesize
\begin{tikzpicture}
\tikzstyle{level 1}=[sibling distance=20em]
\tikzstyle{level 2}=[sibling distance=10em]
\tikzstyle{level 3}=[sibling distance=5em]
\tikzstyle{level 4}=[level distance=2em]
\node[root] (root) {} [grow=down,level distance=8ex]
child {node[nonterminal] (a) {$a\vphantom{)}$}
  child {node[nonterminal] (aa) {$(a,a)$}
    child {node[nonterminal] (aaa) {$(a,a,a)$}
      child[black,dotted]}
    child {node[nonterminal] (aab) {$(a,a,b)$}
      child[black,dotted]}
  }
  child {node[nonterminal] (ab) {$(a,b)$}
    child {node[nonterminal] (aba) {$(a,b,a)$}
      child[black,dotted]}
    child {node[nonterminal] (abb) {$(a,b,b)$}
      child[black,dotted]}
  }
}
child {node[nonterminal] (b) {$b\vphantom{)}$}
  child {node[nonterminal] (ba) {$(b,a)$}
    child {node[nonterminal] (baa) {$(b,a,a)$}
      child[black,dotted]}
    child {node[nonterminal] (bab) {$(b,a,b)$}
      child[black,dotted]}
  }
  child {node[nonterminal] (bb) {$(b,b)$}
    child {node[nonterminal] (bba) {$(b,b,a)$}
      child[black,dotted]}
    child {node[nonterminal] (bbb) {$(b,b,b)$}
      child[black,dotted]}
  }
};
\draw[local,thick] (root) +(190:1.5em) arc (190:350:1.5em);
\draw[local,thick] (b) +(210:2em) arc (210:330:2em);
\draw[local,thick] (a) +(210:2em) arc (210:330:2em);
\draw[local,thick] (bb) +(230:2.25em) arc (230:310:2.25em);
\draw[local,thick] (ba) +(230:2.25em) arc (230:310:2.25em);
\draw[local,thick] (ab) +(230:2.25em) arc (230:310:2.25em);
\draw[local,thick] (aa) +(230:2.25em) arc (230:310:2.25em);
\path (root) +(275:2.35em) node[local] {$\lmargin{\cdot}$};
\path (a) +(270:2.95em) node[local] {$\ltransition{\cdot}{a}$};
\path (b) +(270:2.95em) node[local] {$\ltransition{\cdot}{b}$};
\path (aa) +(300:2.8em) node[local,above right] {$\ltransition{\cdot}{a}$};
\path (bb) +(300:2.8em) node[local,above right] {$\ltransition{\cdot}{b}$};
\path (ba) +(300:2.8em) node[local,above right] {$\ltransition{\cdot}{a}$};
\path (ab) +(300:2.8em) node[local,above right] {$\ltransition{\cdot}{b}$};
\end{tikzpicture}
\caption{The (initial part of the) imprecise probability tree for an imprecise Markov process whose states can assume two values, $a$ and~$b$, and can change at time instants $n=1,2,3,\dots$}
\label{fig:markovtree}
\end{figure}
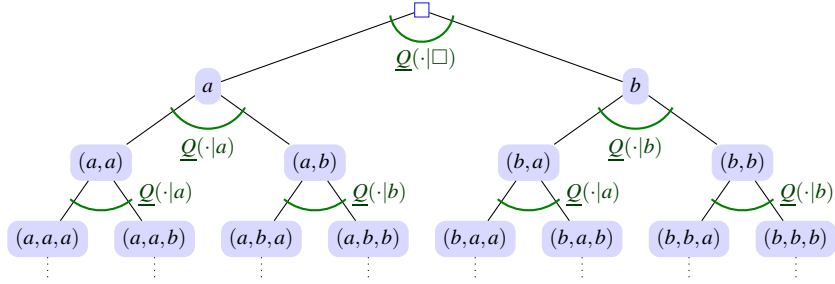

We refer to Refs.~\cite{cooman2008,hermans2012,skulj2013} for detailed studies of the behaviour of these processes.
We restrict ourselves here in Section~\ref{sec:lower-transition} to a summary of the existing material in the literature that is relevant for the discussion of ergodicity in later sections.
As far as we can tell, all of the current discussions and results about imprecise Markov chains deal with a finite time horizon, and consider only bounded real variables (gambles).
For this reason, we devote some effort in Section~\ref{sec:properties:of:global:expectations} to broadening the discussion to an infinite time horizon, using the expressions for the global conditional lower expectations that we introduced in Section~\ref{sec:global:models} and extended to extended real variables in Section~\ref{sec:properties:of:global:models}.
In Section~\ref{sec:shift-invariance}, we discuss the relationship between time shifts and lower expectation operators in imprecise Markov chains and use these results to characterise their potential stationarity.

We believe it is important to explain at this point how our imprecise Markov chains are related to their precise counterparts.
Recall from the discussion in Section~\ref{sec:global:models} that the expressions for the lower and upper expectations in Proposition~\ref{prop:finite:horizon:gambles} have an interesting interpretation in terms of (precise) probability trees~\cite{cooman2007d}: (i) the imprecise probability tree for an imprecise Markov chain corresponds to a collection of \emph{compatible} precise probability trees with the same event tree, by associating with each situation $\xvalto{n}$ in the event tree some arbitrarily chosen precise local expectation $\sitinlocal{\cdot}{\xvalto{n}}$ in the convex closed set $\solp(\sitinllocal{\cdot}{\xval{n}})$ of expectations that are compatible with the local lower expectation $\sitinllocal{\cdot}{\xval{n}}$; and (ii) for any $n$-measurable gamble $f$ on $\pths$, the global precise expectations in the compatible precise probability trees will then range over a closed interval whose lower and upper bounds are given by the expressions in Proposition~\ref{prop:finite:horizon}.
\emph{It should be clear that the local precise models $\sitinlocal{\cdot}{\xvalto{n}}$ need not satisfy the Markov condition,\footnote{\dots\ in either its time-homogeneous or time-inhomogeneous form.} in contradistinction with the \emph{collections} $\solp(\sitinllocal{\cdot}{\xval{n}})$ of precise models they are chosen from.}
In other words, imprecise Markov chains are \emph{not} simply collections of precise Markov chains, but rather correspond to collections of general stochastic processes whose local models belong to sets whose lower and upper envelopes satisfy a Markov condition.

\subsection{Lower transition operators and {\pflike} behaviour}\label{sec:lower-transition}
We can use the local uncertainty models to introduce a (generally non-linear) transformation $\ltrans$ of the set $\gambles$ of all gambles on the state space $\stateset$.
The so-called \emph{lower transition operator} of the imprecise Markov chain is given by:
\begin{equation*}
\ltrans\colon\gambles\to\gambles\colon f\mapsto\ltrans f,
\end{equation*}
where $\ltrans f$ is the gamble on $\stateset$ defined by
\begin{equation*}
\ltrans f(x)\coloneqq\ltransition{f}{x}
\quad\text{for all $x\in\stateset$}.
\end{equation*}
The conjugate \emph{upper transition operator} $\utrans$ is defined by $\utrans f\coloneqq-\ltrans(-f)$ for all $f\in\gambles$.
In particular, $\ltrans\ind{\set{y}}(x)$ is the lower probability to go from state value $x$ to state value $y$ in one time step, and $\utrans\ind{\set{y}}(x)$ the conjugate upper probability.
This seems to suggest that the lower/upper transition operators $\ltrans$ are generalisations of the concept of a Markov transition matrix for ordinary Markov chains.
This is confirmed by the following result, proved in Ref.~\cite[Corollary~3.3]{cooman2008} as a special case of the law of iterated (lower) expectations~\cite{cooman2007d,shafer2001}; see also Corollary~\ref{corol:LILE3} further on for a more general formulation.
If, for any $n\in\nats$ and any $f\in\gambles$, we denote by $\ljoint[n](f)$ the value of the (global) lower expectation $\ljoint(f(\state{n}))$ of the real variable $f(\state{n})$ that only depends on the state $\state{n}$ at time $n$, then
\begin{equation*}
\ljoint[n](f)=\ljoint[1](\ltrans^{n-1}f),
\text{ with }
\ltrans^{n-1}f\coloneqq\underset{\text{$n-1$ times}}{\underbrace{\ltrans\ltrans\dots\ltrans}}f,
\end{equation*}
and where, of course, $\ljoint[1]=\sitinllocal{\cdot}{\init}$ is the marginal local model for the state $\state{1}$ at time $1$.
In a similar vein, for any $n\in\natswithzero$, $\ltrans^{n}\ind{\set{y}}(x)$ is the lower probability to go from state value $x$ to state value $y$ in $n$ time steps, and $\utrans^{n}\ind{\set{y}}(x)$ the conjugate upper probability.

We can formally call \emph{lower transition operator} any transformation $\ltrans$ of $\gambles$ such that for any $x\in\stateset$, the real functional $\ltrans_x$ on $\gambles$, defined by $\ltrans_x(f)\coloneqq\ltrans f(x)$ for all $f\in\gambles$, is a lower expectation---satisfies the coherence axioms~\ref{it:lex:bounds}--\ref{it:lex:homo}.
The composition of any two lower transition operators is again a lower transition operator. 
See Ref.~\cite{cooman2008} for more details on the definition and properties of such lower transition operators, and Ref.~\cite{cooman2014:filtermaps} for a mathematical discussion of the general role of these  operators in imprecise probabilities.

We call a lower transition operator $\ltrans$ \emph{{\pflike}} if for all $f\in\gambles$, the sequence of gambles $\ltrans^{n}f$ converges point-wise to a constant real number, which we will then denote by $\ljoint[\pf](f)$.
An imprecise Markov chain is said to be \emph{\pflike} if its lower transition operator is.

The following result was proved in Ref.~\cite[Theorem~5.1]{cooman2008}, together with a simple sufficient (and quite weak) condition on $\ltrans$ for a Markov chain to be {\pflike}: there is some $n\in\nats$ such that $\min\utrans^{n}\ind{\set{y}}>0$ for all $y\in\stateset$, or in other words, all state values can be reached from any state value with positive upper probability in (precisely) $n$ time steps. 
More involved necessary \emph{and} sufficient conditions were given later in Refs.~\cite{hermans2012,skulj2013}; see also Theorem~\ref{thm:coe:properties}\ref{thm:coe:ergodicity} further on.

\begin{proposition}[\cite{cooman2008}]\label{prop:perron:frobenius}
A lower transition operator\/ $\ltrans$ is {\pflike} if and only if there is some real functional $\ljoint[\infty]$ on $\gambles$ such that for any initial model $\ljoint[1]$ and any $f\in\gambles$, it holds that $\ljoint[1](\ltrans^{n-1}f)\to\ljoint[\infty](f)$. 
Moreover, in that case the functional $\ljoint[\infty]$ is a lower expectation on $\gambles$, called the \emph{stationary lower expectation}, it coincides with $\ljoint[\pf]$, and it is the only lower expectation that is\/ $\ltrans$-invariant in the sense that $\ljoint[\infty]\circ\ltrans=\ljoint[\infty]$.
\end{proposition}

\subsection{Properties of the global lower expectations}\label{sec:properties:of:global:expectations}
The global lower and upper expectations introduced in Section~\ref{sec:global:models} and extended in Section~\ref{sec:properties:of:global:models} have special properties when we restrict ourselves to imprecise Markov chains.
We explore them in this section.

We begin with a few preliminary remarks.
In this context, we can identify $\pths$ with $\stateset^\nats$ and paths $\pth$ with elements of $\stateset^\nats$.
We will do so freely from now on.
Similarly, any situation $s\neq\init$ can be identified with some sequence of states $\xvalto{n}\in\statesto{n}$ for some $n\in\nats$.

Recalling that $\statesto{n}=\stateset^{n}$ allows us to concatenate situations $s$ with other situations $t$ into new situations $st$; the initial situation $\init$ works as the neutral element for this operation.

We can also concatenate situations $s$ and paths $\pth$ into new paths $s\pth$.
This allows us to use a situation $s$ to construct a new variable $g\coloneqq f(s\andpath)$ from a variable $f$ by letting
\begin{equation*}
g(\pth)\coloneqq f(s\pth)\quad\text{for all $\pth\in\pths$}
\end{equation*}
We say that a variable $g$ \emph{does not depend on the first $n$ states $\stateto{n}$}---with $n\in\natswithzero$ if
\begin{equation*}
g(s\andpath)=g(t\andpath)
\quad\text{for all $s,t\in\statesto{n}$},
\end{equation*}
which of course implies that there is some variable $f$ such that $g(s\andpath)=f$ for all $s\in\statesto{n}$.

We assume that we have an imprecise Markov chain with marginal model $\lmargin{\cdot}$ and transition models $\ltransition{\cdot}{x}$, $x\in\stateset$, or equivalently, a lower transition operator $\ltrans$.

We first extend these local transition models from bounded to extended real maps.
In accordance with what we have found in Corollary~\ref{cor:global:and:local}, we extend the local models $\ltransition{\cdot}{x}$ and the corresponding lower transition operator $\ltrans$ to extended real maps $g\colon\stateset\to\xreals$ on $\stateset$ by letting
\begin{equation}\label{eq:ltrans:extended}
\ltrans g(x)
\coloneqq\ltransition{g}{x}
\coloneqq\sup\cset{\ltransition{h}{x}}{h\in\gambles\text{ and }h\leq g}
\text{ for all $x\in\stateset$}.
\end{equation}
That this is indeed an extension follows from the monotonicity [\ref{it:lex:monotone}] of $\ltransition{\cdot}{x}$.

Similarly to what we did in the previous section, for any $n\in\nats$ and any extended real map $g$ on $\stateset$, we denote by $\ljoint[n](g)$ the value of the (global) lower expectation $\ljoint(g(\state{n}))$ of the extended real variable $g(\state{n})$ that only depends on the state $\state{n}$ at time $n$: $\ljoint[n](g)\coloneqq\ljoint(g(\state{n}))=\sitinlglobal{g(\state{n})}{\init}$.
Recall as a special case of Corollary~\ref{cor:global:and:local} [for $n=0$] that for any extended real map $g$ on $\stateset$:
\begin{equation}\label{eq:marginal:extended}
\ljoint[1](g)
=\sitinlglobal{g(\state{1})}{\init}
=\sup\cset{\sitinllocal{h}{\init}}{h\in\gambles\text{ and }h\leq g}.
\end{equation}

Also, when the lower transition operator $\ltrans$ is {\pflike}, and therefore has a unique stationary lower expectation $\ljoint[\infty]$ on $\gambles$, we can extend this lower expectation to extended real maps $g$ on $\stateset$ by similarly letting
\begin{equation}\label{eq:stationary:expectation:extended}
\ljoint[\infty](g)
\coloneqq\sup\cset{\ljoint[\infty](h)}{h\in\gambles\text{ and }h\leq g}.
\end{equation}
This extended functional then satisfies a similar invariance property:

\begin{proposition}\label{prop:extended:remains:stationary}
Assume that\/ $\ltrans$ is {\pflike}.
Then $\ljoint[\infty](g)=\ljoint[\infty](\ltrans g)$ for any extended real map $g$ on $\stateset$.  
\end{proposition}

\begin{proof}
Observe that
\begin{align*}
\ljoint[\infty](g)
&=\sup\cset{\ljoint[\infty](h)}{h\in\gambles\text{ and }h\leq g}\\
&\leq\sup\cset{\ljoint[\infty](h)}{h\in\gambles\text{ and }\ltrans h\leq\ltrans g}\\
&=\sup\cset{\ljoint[\infty](\ltrans h)}{h\in\gambles\text{ and }\ltrans h\leq\ltrans g}\\
&\leq\sup\cset{\ljoint[\infty](h')}{h'\in\gambles\text{ and }h'\leq\ltrans g}
=\ljoint[\infty](\ltrans g),
\end{align*}
where the first inequality follows because $h\leq g$ implies that $\ltrans h\leq\ltrans g$ [use Equation~\eqref{eq:ltrans:extended}], the second equality because $\ljoint[\infty](h)=\ljoint[\infty](\ltrans h)$ [use Proposition~\ref{prop:perron:frobenius}], and the last inequality because $\ltrans h\in\gambles$ [\ref{it:lex:more:bounds}].

For the converse inequality $\ljoint[\infty](\ltrans g)\leq\ljoint[\infty](g)$, fix any $\epsilon>0$. 
We may assume without loss of generality that there is some $h\in\gambles$ such that $h\leq\ltrans g$: otherwise $\ljoint[\infty](\ltrans g)=-\infty$ by Equation~\eqref{eq:stationary:expectation:extended}, and the converse inequality holds trivially.
So consider any such $h$. 
Since $\stateset$ is finite, it follows from Equation~\eqref{eq:ltrans:extended} and the monotonicity of $\ltrans$ [which follows easily from Equation~\eqref{eq:ltrans:extended}] that there is some $h_\epsilon\in\gambles$ such that $h_\epsilon\leq g$ and $h\leq\ltrans h_\epsilon+\epsilon$.
The monotonicity and constant additivity of $\ljoint[\infty]$ [which follow easily from Equation~\eqref{eq:stationary:expectation:extended}] then imply that $\ljoint[\infty](h)\leq\ljoint[\infty](\ltrans h_\epsilon)+\epsilon$ and $\ljoint[\infty](h_\epsilon)\leq\ljoint[\infty](g)$, whence, since $\ljoint[\infty](\ltrans h_\epsilon)=\ljoint[\infty](h_\epsilon)$ [use Proposition~\ref{prop:perron:frobenius}] also $\ljoint[\infty](h)\leq\ljoint[\infty](g)+\epsilon$. 
Since this inequality holds for any $h\in\gambles$ such that $h\leq\ltrans g$, it follows from Equation~\eqref{eq:stationary:expectation:extended} that $\ljoint[\infty](\ltrans g)\leq\ljoint[\infty](g)+\epsilon$.
Since this inequality holds for all $\epsilon>0$, we are done.
\end{proof}

We are now ready to start our analysis.
Our first, basic result is a Markov property for the global models.
It states that all global conditional models are completely determined by the global conditional models $\sitinlglobal{\cdot}{x}$, $x\in\stateset$:

\begin{proposition}[Markov property for global models]\label{prop:markov:property:essential}
Consider any extended real variable $f$, any situation $s\in\sits$ and any $x\in\stateset$, then
\begin{equation*}
\sitinlglobal{f}{sx}
=\sitinlglobal{f(s\andpath)}{x}.
\end{equation*}
\end{proposition}
\noindent
A perhaps more familiar way of writing this is $\sitinlglobal{f(\state{1}\state{2}\dots)}{sx}=\sitinlglobal{f(s\state{1}\state{2}\dots)}{x}$.

\begin{proof}
Consider, for ease of notation, the extended real variable $g\coloneqq f(s\andpath)$. 
Consider any bounded above submartingale $\submartin$ such that $\limsup\submartin(sx\andpath)\leq f(sx\andpath)$, and let $\submartin'$ be the real process defined by $\submartin'(u)\coloneqq\submartin(su)$ for all $u\in\sits$.
$\submartin'$ is clearly a bounded above submartingale because $\submartin$ is, and moreover $\submartin'(x)=\submartin(sx)$ and $\limsup\submartin'(x\andpath)=\limsup\submartin(sx\andpath)\leq f(sx\andpath)=g(x\andpath)$, whence, by Equation~\eqref{eq:global:lower:extended}
\begin{align*}
\sitinlglobal{f}{sx}
&=\sup\cset{\submartin(sx)}
{\submartin\in\basubmartins\text{ and }\limsup\submartin(sx\andpath)\leq f(sx\andpath)}\\
&\leq\sup\cset{\submartin'(x)}
{\submartin'\in\basubmartins\text{ and }\limsup\submartin'(x\andpath)\leq g(x\andpath)}
=\sitinlglobal{g}{x}.
\end{align*}
Conversely, consider any bounded above submartingale $\submartin$ such that $\limsup\submartin(x\andpath)\leq g(x\andpath)$, and let $\submartin'$ be the real process defined by letting $\submartin'(sxu)\coloneqq\submartin(xu)$ for all $u\in\sits$, and letting $\submartin'(t)\coloneqq\submartin(sx)$ in all situations $t$ that do not follow $sx$.
Then $\submartin'$ is clearly a bounded above submartingale because $\submartin$ is, and moreover $\submartin'(sx)=\submartin(x)$ and $\limsup\submartin'(sx\andpath)=\limsup\submartin(x\andpath)\leq g(x\andpath)=f(sx\andpath)$, whence, again by Equation~\eqref{eq:global:lower:extended}
\begin{align*}
\sitinlglobal{g}{x}
&=\sup\cset{\submartin(x)}
{\submartin\in\basubmartins\text{ and }\limsup\submartin(x\andpath)\leq g(x\andpath)}\\
&\leq\sup\cset{\submartin'(sx)}
{\submartin'\in\basubmartins\text{ and }\limsup\submartin'(sx\andpath)\leq f(sx\andpath)}
=\sitinlglobal{f}{sx}.\qedhere
\end{align*}
\end{proof}
\noindent
This allows us to introduce a new notation $\newlglobal{n}{g}{x}$ for conditional lower expectations of extended real variables $g$ that do not depend on the first $n-1$ states $\stateto{n-1}$, with $n\in\nats$:
\begin{equation*}
\newlglobal{n}{g}{x}
\coloneqq\sitinlglobal{g}{sx}
=\sitinlglobal{g(s\andpath)}{x},
\end{equation*}
where $s$ is any situation of length $n-1$.
Obviously, $\newlglobal{n}{g}{\cdot}$ is an extended real-valued map on $\stateset$.

We can now prove a number of related corollaries to our general law of iterated lower expectations, formulated in Theorem~\ref{thm:law:of:iterated:expectations:general}.

\begin{corollary}\label{corol:LILE1}
Let $n\in\nats$ and $k\in\natswithzero$, and consider any extended real variable $g$ that does not depend on the first $n+k-1$ states.
Then
\begin{equation*}
\newlglobal{n}{g}{\cdot}=\ltrans^k\newlglobal{n+k}{g}{\cdot}.
\end{equation*}  
\end{corollary} 

\begin{proof}
It clearly suffices to give the proof for $k=1$.
So consider any extended real variable $g$ that does not depend on the first $n$ states, and any $x\in\stateset$.  
We prove that $\newlglobal{n}{g}{x}=\ltransition{\newlglobal{n+1}{g}{\cdot}}{x}$.

Consider any $\xvalto{n-1}\in\statesto{n-1}$, then it follows from Theorem~\ref{thm:law:of:iterated:expectations:general} and Proposition~\ref{prop:plug:in:the:conditional} that
\begin{equation}\label{eq:LILE1:intermediate}
\sitinlglobal{g}{\xvalto{n-1}x}
=\sitinlglobal{\sitinlglobal{g}{\stateto{n+1}}}{\xvalto{n-1}x}
=\sitinlglobal{\sitinlglobal{g}{\xvalto{n-1}x\state{n+1}}}{\xvalto{n-1}x}.
\end{equation}
Because $g$ in particular does not depend on the first $n-1$ states, we see that for the left-hand side of this equality: $\sitinlglobal{g}{\xvalto{n-1}x}=\newlglobal{n}{g}{x}$.
We now look at the right-hand side.
Consider the extended real map $h\coloneqq\sitinlglobal{g}{\xvalto{n-1}x\andstate}$ on $\stateset$ and the $n+1$-measurable extended real variable $h(\state{n+1})=\sitinlglobal{g}{\xvalto{n-1}x\state{n+1}}$, then we see that
\begin{align*}
\sitinlglobal{\sitinlglobal{g}{\xvalto{n-1}x\state{n+1}}}{\xvalto{n-1}x}
&=\sitinlglobal{h(\state{n+1})}{\xvalto{n-1}x}\\
&=\sup\cset{\sitinllocal{h'}{\xvalto{n-1}x}}{h'\in\gambles\text{ and }h'\leq h}\\
&=\sup\cset{\sitinllocal{h'}{x}}{h'\in\gambles\text{ and }h'\leq h}
=\sitinllocal{h}{x},
\end{align*}
where the second equality follows from Corollary~\ref{cor:global:and:local}, the third from the Markov property~\eqref{eq:markov:condition} of the local models, and the last from Equation~\eqref{eq:ltrans:extended}.
To complete the proof, consider that, since $g$ does not depend on the first $n$ states, $h=\sitinlglobal{g}{\xvalto{n-1}x\andstate}=\newlglobal{n+1}{g}{\andstate}$.
\end{proof}

\begin{corollary}\label{corol:LILE2}
Let $\ell\in\nats$ and consider any extended real variable $g$ that does not depend on the first $\ell-1$ states.
Then:
\begin{equation*}
\ljoint(g)=\ljoint[1](\ltrans^{\ell-1}\newlglobal{\ell}{g}{\cdot}).
\end{equation*}  
\end{corollary}

\begin{proof}
It follows from Corollary~\ref{corol:LILE1} with $n=1$ and $k=\ell-1$ that $\newlglobal{1}{g}{\cdot}=\ltrans^{\ell-1}\newlglobal{\ell}{g}{\cdot}$. 
Furthermore, by applying Theorem~\ref{thm:law:of:iterated:expectations:general} [for $m\coloneqq0$ and $n\coloneqq1$], we find that $\ljoint(g)=\ljoint(\stateinlglobal{g}{1})$. This establishes the proof because it follows from the definition of $\ljoint[1]$ and $\newlglobal{1}{g}{\cdot}$ that $\ljoint(\stateinlglobal{g}{1})=\ljoint[1](\newlglobal{1}{g}{\cdot})$.
\end{proof}

\begin{corollary}\label{corol:LILE3}
Consider any $n\in\nats$ and any extended real map $f$ on $\stateset$. 
Then $\ljoint[n](f)=\ljoint[k](\ltrans^{n-k}f)=\ljoint[1](\ltrans^{n-1}f)$ for any $1\leq k\leq n$.
\end{corollary}

\begin{proof}
We use Corollary~\ref{corol:LILE2} with $\ell\coloneqq n$ and $g\coloneqq f(\state{n})$, leading to
\begin{equation*}
\ljoint[n](f)
=\ljoint(g)
=\ljoint[1](\ltrans^{n-1}\newlglobal{n}{g}{\cdot}),
\end{equation*}
since the extended real variable $g$ only depends on the $n$-the state $\state{n}$, and therefore does not depend on the first $n-1$ states. 
For the same reason, we see that for any $x\in\stateset$, $\newlglobal{n}{g}{x}=\sitinlglobal{g}{sx}=\sitinlglobal{g(sx\state{n+1}\dots)}{sx}=\sitinlglobal{f(x)}{sx}=f(x)$, where $s$ is any situation of length $n-1$, and where the second equality follows from Proposition~\ref{prop:plug:in:the:conditional}, and the last from coherence property~\ref{it:elex:more:bounds}.
Hence $\ljoint[n](f)=\ljoint[1](\ltrans^{n-1}f)$.
Now consider any natural $k\leq n$, then we find in a similar manner that $\ljoint[k](f)=\ljoint[1](\ltrans^{k-1}f)$ and therefore also
\begin{equation*}
\ljoint[1](\ltrans^{n-1}f)
=\ljoint[1](\ltrans^{k-1}\ltrans^{n-k}f)
=\ljoint[k](\ltrans^{n-k}f).\qedhere
\end{equation*}
\end{proof}

\subsection{Shift invariance}\label{sec:shift-invariance}
We introduce the \emph{shift operator} $\theta$ on $\nats$ by letting $\theta(n)\coloneqq n+1$ for all $n\in\nats$.
This induces a shift operator on $\pths$: $\theta\pth$ is the path with $(\theta\pth)_n\coloneqq\pth_{\theta(n)}=\pth_{n+1}$ for all $n\in\nats$.
And this also induces a shift operation on variables $f$: $\theta f$ is the variable defined by $(\theta f)(\pth)\coloneqq f(\theta\pth)$ for all $\pth\in\pths$.

\begin{proposition}\label{prop:shifts:and:dependence}
Let $n\in\natswithzero$.
If the variable $g$ does not depend on the first $n$ states, then $\theta g$ does not depend on the first $n+1$ states.
\end{proposition}

\begin{proof}
Assume $g$ does not depend on the first $n$ states, so there is some variable $f$ such that $g(s\andpath)=f$ for all $s\in\statesto{n}$, where of course $\statesto{n}=\stateset^{n}$.
Then for all $x\in\stateset$, $s\in\stateset^{n}$ and all $\pth\in\pths$:
\begin{equation*}
(\theta g)(xs\pth)
=g(\theta(xs\pth))
=g(s\pth)
=f(\pth),
\end{equation*}
which concludes the proof.
\end{proof}

We call a variable $f$ \emph{shift invariant} if $\theta f=f$, meaning that
\begin{equation*}
f(\pth)=f(\theta\pth)\quad\text{for all $\pth\in\pths$}.
\end{equation*}

\begin{proposition}
A shift invariant variable $f$ does not depend on the first $n$ states $\stateto{n}$, for all $n\in\natswithzero$.
\end{proposition}

\begin{proof}
Immediate consequence of Proposition~\ref{prop:shifts:and:dependence}.
\end{proof}

Another way to understand that a variable $f$ does not depend on the first $n$ states, is that then $f(\pth)=f(s\theta^{n}\pth)$ for all $s\in\statesto{n}$ and $\pth\in\pths$, which we also write as $f=f(s\theta^{n}\andpath)$ for all $s\in\statesto{n}$.

The following propositions tell us that the global lower expectations satisfy a shift invariance property.

\begin{proposition}\label{prop:global:shift:invariance}
Let $n\in\nats$ and consider any extended real variable $g$ that does not depend on the first $n-1$ states.
Then for all $k\in\natswithzero$:
\begin{equation*}
\newlglobal{n}{g}{\cdot}=\newlglobal{n+k}{\theta^kg}{\cdot}.
\end{equation*}  
\end{proposition} 

\begin{proof}
It clearly suffices to prove the statement for $k=1$.
So consider any $s\in\statesto{n-1}$ [recall that $\statesto{n-1}=\stateset^{n-1}$] and any $x,y\in\stateset$, then it follows from Proposition~\ref{prop:markov:property:essential} that $\newlglobal{n}{g}{x}=\sitinlglobal{g}{sx}=\sitinlglobal{g(s\andpath)}{x}$ and $\newlglobal{n+1}{\theta g}{x}=\sitinlglobal{\theta g}{ysx}=\sitinlglobal{\theta g(ys\andpath)}{x}$.
Now observe that $\theta g(ys\andpath)=g(\theta(ys\andpath))=g(s\andpath)$.
\end{proof}

\begin{proposition}\label{prop:shift:global}
For any extended real variable $f$ and any $n\in\natswithzero$:
\begin{equation*}
\ljoint(\theta^{n}f)=\ljoint[n+1](\newlglobal{1}{f}{\cdot}).
\end{equation*}
\end{proposition}

\begin{proof}
Since $\theta^{n}f$ does not depend on the first $n$ states [see Proposition~\ref{prop:shifts:and:dependence}], we infer from Corollary~\ref{corol:LILE2} [with $\ell\coloneqq n+1$] and Proposition~\ref{prop:global:shift:invariance} that indeed
\begin{equation*}
\ljoint(\theta^{n}f)
=\ljoint[1](\ltrans^{n}\newlglobal{n+1}{\theta^{n}f}{\cdot})
=\ljoint[1](\ltrans^{n}\newlglobal{1}{f}{\cdot})
=\ljoint[n+1](\newlglobal{1}{f}{\cdot}), 
\end{equation*}
where the last equality follows from $\ljoint[n+1]=\ljoint[1]\circ\ltrans^{n}$ [see Corollary~\ref{corol:LILE3}].
\end{proof}

As a generalisation of the case for precise Markov chains, we can call an imprecise Markov chain \emph{stationary} or \emph{time invariant} if
\begin{equation*}
\ljoint(f)=\ljoint(\theta f)
\text{ for all extended real variables $f$}.
\end{equation*}
The following proposition gives a simple characterisation of stationarity.

\begin{proposition}[Stationarity]\label{prop:stationarity}
Consider an imprecise Markov chain with marginal lower expectation $\ljoint[1]$ and lower transition\/ operator $\ltrans$ that is {\pflike} with stationary lower expectation $\ljoint[\infty]$.
Then the imprecise Markov chain is stationary if and only if $\ljoint[\infty]=\ljoint[1]$.
\end{proposition}

\begin{proof}
Assume that $\ljoint[\infty]=\ljoint[1]$, then $\ljoint[2]=\ljoint[1]\circ\ltrans=\ljoint[\infty]\circ\ltrans=\ljoint[\infty]=\ljoint[1]$, where the first equality follows from Corollary~\ref{corol:LILE3}, and the one but last equality from Proposition~\ref{prop:extended:remains:stationary}.
Hence it follows from Proposition~\ref{prop:shift:global} and Corollary~\ref{corol:LILE2} [with $\ell\coloneqq1$] that for any extended real variable $f$, $\ljoint(\theta f)=\ljoint[2](\newlglobal{1}{f}{\cdot})=\ljoint[1](\newlglobal{1}{f}{\cdot})=\ljoint(f)$. 
Hence the imprecise Markov chain is stationary.

Assume, conversely, that the imprecise Markov chain is stationary.
Let $h$ be any gamble on $\stateset$, and consider the real variables $f\coloneqq h(\state{1})$ and $\theta f=h(\state{2})$. 
Then on the one hand $\ljoint(f)=\ljoint[1](h)$, and on the other hand $\ljoint(\theta f)=\ljoint[2](h)$, so it follows from Corollary~\ref{corol:LILE3} [with $n\coloneqq2$] and stationarity that $\ljoint[1]\circ\ltrans=\ljoint[2]=\ljoint[1]$.
So $\ljoint[1]$ is $\ltrans$-invariant, which implies that $\ljoint[1]=\ljoint[\infty]$ [use Proposition~\ref{prop:perron:frobenius} to get the equality for gambles, which also implies the equalities for their extensions, via Equations~\eqref{eq:marginal:extended} and~\eqref{eq:stationary:expectation:extended}]. 
\end{proof}
\noindent
We gather from Proposition~\ref{prop:stationarity} that, with any {\pflike} lower transition operator $\ltrans$ and associated stationary lower expectation $\ljoint[\infty]$, there always corresponds a unique stationary imprecise Markov chain; its initial model is given by $\sitinllocal{\cdot}{\init}\coloneqq\ljoint[\infty]$. 
We will denote its corresponding (shift-invariant) global lower expectation operator by $\statljoint$.

\section{Transition and return times}\label{sec:transition:and:return:times}
Let us now look at lower (and upper) expected transition and return times, as a simple and elegant example of what can be done using our extensions of the joint lower and upper expectations to an infinite time horizon and extended real variables, and their properties, discussed in the previous section.

Consider two (possibly identical) state values $x$ and $y$ in $\stateset$.
Suppose that the imprecise Markov chain starts out at time $n$ in state value $x$, then we can ask ourselves how long it will take for it to reach the state value $y$, and when $y=x$, for the imprecise Markov chain to return to the state value $x$.
To study this, we introduce the extended real variables $\timefromto[n]{x}{y}$ given by:
\begin{equation}\label{eq:transition:times:path}
\timefromto[n]{x}{y}(\pth)
\coloneqq
\begin{cases}
0
&\quad\text{ if $\pth_n\neq x$}\\
\inf\cset{m\in\nats}{\pth_{n+m}=y}
&\quad\text{ if $\pth_n=x$}.
\end{cases}
\end{equation}
Observe that $\theta\timefromto[n]{x}{y}=\timefromto[n+1]{x}{y}$.
Consider the lower expected time $\sitinlglobal{\timefromto[n]{x}{y}}{sx}$, where $s$ is any situation of length $n-1$.
Then, since $\timefromto[n]{x}{y}$ clearly does not depend on the first $n-1$ states, we infer from Proposition~\ref{prop:markov:property:essential} [the Markov property] that $\sitinlglobal{\timefromto[n]{x}{y}}{sx}=\sitinlglobal{\timefromto[n]{x}{y}(s\andpath)}{x}=\newlglobal{n}{\timefromto[n]{x}{y}}{x}$.
Moreover, we infer from Proposition~\ref{prop:global:shift:invariance} that
\begin{equation*}
\newlglobal{n+1}{\timefromto[n+1]{x}{y}}{x}
=\newlglobal{n+1}{\theta\timefromto[n]{x}{y}}{x}
=\newlglobal{n}{\timefromto[n]{x}{y}}{x},
\end{equation*}
so we conclude that $\sitinlglobal{\timefromto[n]{x}{y}}{sx}$ neither depends on the initial segment $s$, nor on its length $n-1$.
A similar conclusion holds for $\sitinuglobal{\timefromto[n]{x}{y}}{sx}$.
We therefore define the \emph{lower} and \emph{upper expected transition times} from $x$ to $y$ as
\begin{equation}\label{eq:transition:times:shift_inv}
\ltimefromto{x}{y}
\coloneqq\sitinlglobal{\timefromto[1]{x}{y}}{x}
=\newlglobal{n}{\timefromto[n]{x}{y}}{x}
\text{ and }
\utimefromto{x}{y}
\coloneqq\sitinuglobal{\timefromto[1]{x}{y}}{x}
=\newuglobal{n}{\timefromto[n]{x}{y}}{x}.
\end{equation}
When $y=x$, we talk about \emph{return times} rather than transition times.
It follows from~\ref{it:elex:more:bounds} that $\utimefromto{x}{y}\geq\ltimefromto{x}{y}\geq1$.

On any path $x\state{2}\andpath$ that starts in $x$, the following recursion equation is satisfied:
\begin{equation}\label{eq:transition:times:recursion}
\timefromto[1]{x}{y}(x\state{2}\andpath)
\coloneqq
\begin{cases}
1
&\text{ if $\state{2}=y$}\\
1+\timefromto[2]{z}{y}(xz\andpath)
&\text{ if $\state{2}=z\neq y$}
\end{cases}
=1+\sum_{z\neq y}\indsing{z}(\state{2})\timefromto[2]{z}{y}(xz\andpath),
\end{equation}
using our convention that $0\cdot+\infty=0$.
We know from Theorem~\ref{thm:law:of:iterated:expectations:general} and Proposition~\ref{prop:plug:in:the:conditional} that
\begin{equation}\label{eq:transition:intermediate}
\ltimefromto{x}{y}
=\sitinlglobal{\timefromto[1]{x}{y}}{x}
=\sitinlglobal{\sitinlglobal{\timefromto[1]{x}{y}}{x\state{2}}}{x}.
\end{equation}
Moreover, for any $z\in\stateset\setminus\set{y}$, we infer from Equation~\eqref{eq:transition:times:recursion}, Proposition~\ref{prop:plug:in:the:conditional} [repeatedly] and coherence [\ref{it:elex:constant:additivity}] that
\begin{align*}
\sitinlglobal{\timefromto[1]{x}{y}}{xz}
&=\sitinlglobal{\timefromto[1]{x}{y}(xz\state{3}\dots)}{xz}
=\sitinlglobal{1+\timefromto[2]{z}{y}(xz\state{3}\dots)}{xz}\\
&=1+\sitinlglobal{\timefromto[2]{z}{y}(xz\state{3}\dots)}{xz}
=1+\sitinlglobal{\timefromto[2]{z}{y}}{xz}
=1+\ltimefromto{z}{y},
\end{align*}
Similarly, we infer from Equation~\eqref{eq:transition:times:recursion}, Proposition~\ref{prop:plug:in:the:conditional} and coherence [\ref{it:elex:more:bounds}] that
\begin{align*}
\sitinlglobal{\timefromto[1]{x}{y}}{xy}
&=\sitinlglobal{\timefromto[1]{x}{y}(xy\state{3}\dots)}{xy}
=\sitinlglobal{1}{xy}
=1.
\end{align*}
Hence
\begin{equation*}
\sitinlglobal{\timefromto[1]{x}{y}}{x\state{2}}
=1+\sum_{z\in\stateset\setminus\set{y}}\indsing{z}(\state{2})\ltimefromto{z}{y},
\end{equation*} 
and if we plug this expression into the local conditional expectation on the right-hand side of Equation~\eqref{eq:transition:intermediate}, and use coherence [\ref{it:elex:constant:additivity}], Corollary~\ref{cor:global:and:local} and Equation~\ref{eq:ltrans:extended}, we are led to the following system of non-linear equations for the lower transition (and return) times:
\begin{equation}\label{eq:transition:times:lower}
\ltimefromto{x}{y}
=1+\ltrans\group[\bigg]{\sum_{z\in\stateset\setminus\set{y}}\indsing{z}\ltimefromto{z}{y}}(x)
\text{ for all $x,y\in\stateset$}.
\end{equation}
A completely analogous argument leads to the corresponding system for the upper transition and return times:
\begin{equation}\label{eq:transition:times:upper}
\utimefromto{x}{y}
=1+\utrans\group[\bigg]{\sum_{z\in\stateset\setminus\set{y}}\indsing{z}\utimefromto{z}{y}}(x) 
\text{ for all $x,y\in\stateset$}.
\end{equation}
Finding a general solution to these systems is a difficult task, which we will not tackle here; for the special case of imprecise birth-death chains, see Ref.~\cite{lopatatzid2015}. 










We end this section by solving the following simple binary case. Let $\stateset=\set{a,b}$ and let
\begin{equation}\label{eq:localmodelsexample}
\ltrans h(x)=\ljoint[1](h)=(1-\epsilon)\frac{h(a)+h(b)}{2}+\epsilon\min h
\text{ for all $h\in\gambles$ and $x\in\stateset$,}
\end{equation}
with $\epsilon\in(0,1)$.
It is clear that $\ltrans^{n}h=\ljoint[1](h)$ and therefore this imprecise Markov chain is {\pflike}, with $\ljoint[\infty]=\ljoint[1]$, so it is stationary as well; see Proposition~\ref{prop:stationarity}.
Since the transition model $\ltransition{\cdot}{x}=\ljoint[1]$ is the same for all state values $x\in\stateset$, this is an imprecise-probabilistic version of a Bernoulli (iid) process,\footnote{There are various ways to generalise a Bernoulli process to an imprecise probabilities context; see Ref.~\cite{debock2012:bernoulli} for discussion.} with
\begin{equation*}
\ltheta_a
=1-\utheta_b
\coloneqq\ljoint[1](\indsing{a})
=\frac{1}{2}-\frac{\epsilon}{2}
\text{ and }
\utheta_a
=1-\ltheta_b
\coloneqq\ujoint[1](\indsing{a})
=\frac{1}{2}+\frac{\epsilon}{2}.
\end{equation*}
In this simple binary case, Equation~\eqref{eq:transition:times:lower} can be significantly simplified. 
For example, for $y=a$ and any $x\in\{a,b\}$, we find that
\begin{equation}\label{eq:simplifiedsystembinary}
\ltimefromto{x}{a}
=1+\ltrans(\indsing{b}\ltimefromto{b}{a})(x)
=1+\ljoint[1](\indsing{b}\ltimefromto{b}{a})
=1+\ltimefromto{b}{a}\ljoint[1](\indsing{b})
=1+\ltimefromto{b}{a}\ltheta_b,
\end{equation}
where the second equality follows from Equation~\eqref{eq:localmodelsexample} and the third one from coherence [\ref{it:lex:homo}] and the following lemma.

\begin{lemma}\label{lemma:binaryexample}
Consider a binary imprecise Markov chain with state space $\states{}=\{a,b\}$ whose local models are given by Equation~\eqref{eq:localmodelsexample}, with $\epsilon\in(0,1)$. Then
$1\leq\ltimefromto{b}{a}\leq\utimefromto{b}{a}<+\infty$.
\end{lemma}
\begin{proof}
Since $\epsilon\in(0,1)$, it follows that $\ltheta_a>0$. Consider the real process $\submartin$ defined by $\submartin(\init)\coloneqq\nicefrac{1}{\ltheta_a}$, $\Delta\submartin(\init)\coloneqq 0$ and, for all $n\in\nats$ and $\xvalto{n}\in\statesto{n}$: 
\begin{equation}\label{eq:defsubmartinexample}
\Delta\submartin(\xvalto{n})\coloneqq
\begin{cases}
1-\indsing{a}\nicefrac{1}{\ltheta_a}
&\text{ if $x_k=b$ for all $k\in\{1,\dots,n\}$}\\
0
&\text{ otherwise.}
\end{cases}
\end{equation}
Since coherence implies that $\overline{E}_1(0)=0$ and
\begin{equation*}
\overline{E}_1(1-\indsing{a}\nicefrac{1}{\ltheta_a})
=1+\overline{E}_1(-\indsing{a}\nicefrac{1}{\ltheta_a})
=1-\underline{E}_1(\indsing{a}\nicefrac{1}{\ltheta_a})
=1-\nicefrac{1}{\ltheta_a}\underline{E}_1(\indsing{a})
=0,
\end{equation*}
it follows from from Equations~\eqref{eq:localmodelsexample} and~\eqref{eq:defsubmartinexample} that $\submartin$ is a supermartingale.

Consider now any $n\in\nats$ and $\xvalto{n}\in\statesto{n}$. If $x_k=b$ for all $k\in\{1,\dots,n\}$, we have that
\begin{equation*}
\submartin(\xvalto{n})=\submartin(\init)+\Delta\submartin(\init)+\sum_{k=1}^{n-1}\Delta\submartin(\xvalto{k})
=\nicefrac{1}{\ltheta_a}+0+(n-1)=\nicefrac{1}{\ltheta_a}+n-1.
\end{equation*}
Otherwise, we find that
\begin{equation*}
\submartin(\xvalto{n})
=\submartin(\init)+\Delta\submartin(\init)+\sum_{k=1}^{k^*-1}\Delta\submartin(\xvalto{k})
=\nicefrac{1}{\ltheta_a}+0+(k^*-1-\nicefrac{1}{\ltheta_a})=k^*-1,
\end{equation*}
where $k^*$ is the smallest index $k$ such that $x_{k}=a$. Hence, we find that the supermartingale $\submartin$ is bounded below by zero and therefore belongs to $\bbsupermartins$, and that
\begin{equation}\label{eq:liminfbinaryexample}
\liminf\submartin(\pth)=
\begin{cases}
+\infty
&\text{ if $\pth_k=b$ for all $k\in\nats$}\\
\inf\{k\in\nats\colon \pth_k=a\}-1
&\text{ otherwise}
\end{cases}
\end{equation}
for all $\pth\in\pths$.

For all $\pth\in\exact{b}$, Equations~\eqref{eq:transition:times:path} and~\eqref{eq:liminfbinaryexample} now imply that $\liminf\submartin(\pth)=\timefromto[1]{b}{a}(\pth)$. Hence, since $\submartin\in\bbsupermartins$, it follows from Equation~\eqref{eq:global:upper:extended} that
\begin{equation*}
\utimefromto{b}{a}
\coloneqq\sitinuglobal{\timefromto[1]{b}{a}}{b}\leq\submartin(b)=\submartin(\init)+\Delta\submartin(\init)(b)=\nicefrac{1}{\ltheta_a}<+\infty.
\end{equation*} 
We also already know that $1\leq\ltimefromto{b}{a}\leq\utimefromto{b}{a}$---see the text after Equation~\eqref{eq:transition:times:shift_inv}.
\end{proof}

Since the lemma tells us that $\ltimefromto{b}{a}$ is real-valued, we can now solve Equation~\eqref{eq:simplifiedsystembinary} for $x=b$ to find that $\ltimefromto{b}{a}=\nicefrac{1}{\utheta_a}$. 
Equation~\eqref{eq:simplifiedsystembinary} also implies that $\ltimefromto{b}{a}=\ltimefromto{a}{a}$. Hence, also using the symmetry, we obtain the following expressions for the lower transition and return times:

\begin{equation*}
\ltimefromto{a}{a}
=\ltimefromto{b}{a}
=\frac{1}{\utheta_a}
=\frac{2}{1+\epsilon}
\text{ and }
\ltimefromto{b}{b}
=\ltimefromto{a}{b}
=\frac{1}{\utheta_b}
=\frac{2}{1+\epsilon}.
\end{equation*}\\[-4mm]

An analogous argument---simplifying Equation~\eqref{eq:transition:times:upper} for $y=a$, solving the resulting system to find $\utimefromto{b}{a}$ and $\utimefromto{a}{a}$, and then invoking symmetry---leads to the following similar expressions for the upper transition and return times:
\begin{equation*}
\utimefromto{a}{a}=\utimefromto{b}{a}=\frac{1}{\ltheta_a}=\frac{2}{1-\epsilon}
\text{ and }
\utimefromto{b}{b}=\utimefromto{a}{b}=\frac{1}{\ltheta_b}=\frac{2}{1-\epsilon}.
\end{equation*}

\section{An interesting equality in imprecise Markov chains}\label{sec:interesting:equality}
We now prove an interesting equality for imprecise Markov chains, which will be instrumental in proving our point-wise ergodic theorem in the next section.

Consider, for any $f\in\gambles$, the corresponding \emph{gain} process $\gain{f}$, defined by:
\begin{equation}\label{eq:gain}
\gain{f}(\stateto{n})
\coloneqq
\sqgroup{f(\state{1})-\ljoint[1](f)}
+\sum_{k=2}^{n}\sqgroup{f(\state{k})-\ltrans f(\state{k-1})}
\text{ for any $n\in\nats$},
\end{equation}
the corresponding \emph{average gain} process $\avgain{f}$, defined by:
\begin{equation}\label{eq:average:gain}
\avgain{f}(\stateto{n})
\coloneqq
\frac{1}{n}
\sqgroup[\bigg]{\sqgroup{f(\state{1})-\ljoint[1](f)}
+\sum_{k=2}^{n}\sqgroup{f(\state{k})-\ltrans f(\state{k-1})}}
\text{ for any $n\in\nats$},
\end{equation}
and the \emph{ergodic average} process $\ergodic{f}$, defined by:
\begin{equation}\label{eq:ergodic:average}
\ergodic{f}(\stateto{n})
\coloneqq\frac{1}{n}\sum_{k=1}^{n}\sqgroup{f(\state{k})-\ljoint[k](f)}
\text{ for any $n\in\nats$}.
\end{equation}
We define these processes to be $0$ in the initial situation $\init$.
Now observe that, for any $n\in\nats$ and any $f\in\gambles$:
\begin{equation}\label{eq:start}
\sum_{\ell=0}^{n-1}\avgain{\ltrans^{\ell}f}(\stateto{n})
=\frac{1}{n}
\sum_{\ell=0}^{n-1}
\sqgroup[\big]{\ltrans^{\ell}f(\state{1})-\ljoint[1](\ltrans^{\ell}f)}
+\frac{1}{n}
\sum_{\ell=0}^{n-1}\sum_{k=2}^{n}
\sqgroup[\big]{\ltrans^{\ell}f(\state{k})-\ltrans^{\ell+1}f(\state{k-1})},
\end{equation}
and moreover
\begin{multline*}
\sum_{\ell=0}^{n-1}\sum_{k=2}^{n}
\sqgroup[\big]{\ltrans^{\ell}f(\state{k})-\ltrans^{\ell+1}f(\state{k-1})}\\
\begin{aligned}
&=\sum_{\ell=0}^{n-1}\sum_{k=2}^{n}\ltrans^{\ell}f(\state{k})
-\sum_{\ell=0}^{n-1}\sum_{k=2}^{n}\ltrans^{\ell+1}f(\state{k-1})
=\sum_{\ell=0}^{n-1}\sum_{k=2}^{n}\ltrans^{\ell}f(\state{k})
-\sum_{\ell=1}^{n}\sum_{k=1}^{n-1}\ltrans^{\ell}f(\state{k})\\
&=\sum_{k=2}^{n}f(\state{k})
+\sum_{\ell=1}^{n-1}\group[\bigg]{\ltrans^{\ell}f(\state{n})+\sum_{k=2}^{n-1}\ltrans^{\ell}f(\state{k})}\\
&\qquad\qquad-\sum_{k=1}^{n-1}\ltrans^{n}f(\state{k})
-\sum_{\ell=1}^{n-1}\group[\bigg]{\ltrans^{\ell}f(\state{1})+\sum_{k=2}^{n-1}\ltrans^{\ell}f(\state{k})}\\
&=\sum_{k=2}^{n}f(\state{k})
+\sum_{\ell=1}^{n-1}\ltrans^{\ell}f(\state{n})
-\sum_{k=1}^{n-1}\ltrans^{n}f(\state{k})
-\sum_{\ell=1}^{n-1}\ltrans^{\ell}f(\state{1})\\
&=\sum_{k=1}^{n}f(\state{k})
+\sum_{\ell=1}^{n}\ltrans^{\ell}f(\state{n})
-\sum_{k=1}^{n}\ltrans^{n}f(\state{k})
-\sum_{\ell=0}^{n-1}\ltrans^{\ell}f(\state{1}),
\end{aligned}
\end{multline*}
and if we substitute this back into Equation~\eqref{eq:start}, we find that, after getting rid of the cancelling terms, recalling that $\ljoint[1](\ltrans^{\ell}f)=\ljoint[\ell+1](f)$, and reorganising a bit:
\begin{align*}
\sum_{\ell=0}^{n-1}\avgain{\ltrans^{\ell}f}(\stateto{n})
&=\frac{1}{n}
\sqgroup[\bigg]{
-\sum_{\ell=0}^{n-1}\ljoint[1](\ltrans^{\ell}f)
+\sum_{k=1}^{n}f(\state{k})
+\sum_{\ell=1}^{n}\ltrans^{\ell}f(\state{n})
-\sum_{k=1}^{n}\ltrans^{n}f(\state{k})
}\\
&=\frac{1}{n}\sum_{k=1}^{n}\sqgroup{f(\state{k})-\ljoint[k](f)}
+\frac{1}{n}\sum_{\ell=1}^{n}\ltrans^{\ell}f(\state{n})
-\frac{1}{n}\sum_{k=1}^{n}\ltrans^{n}f(\state{k})
\end{align*}
or in other words:
\begin{equation}\label{eq:finish}
\ergodic{f}(\stateto{n})
=\sum_{\ell=0}^{n-1}\avgain{\ltrans^{\ell}f}(\stateto{n})
+\frac{1}{n}\sum_{k=1}^{n}\ltrans^{n}f(\state{k})
-\frac{1}{n}\sum_{\ell=1}^{n}\ltrans^{\ell}f(\state{n}).
\end{equation}
This is an important relationship between the ergodic average and the average gain. We now intend to show that under certain conditions the remaining terms on the right-hand side essentially cancel out for large enough $n$.

\section{Consequences of the {\pflike} character}\label{sec:perron:frobenius}
Let us associate with a lower transition operator $\ltrans$ the following (weak) \emph{coefficient of ergodicity} \citep{skulj2013,hermans2012}:
\begin{equation*}
\coe(\ltrans)
\coloneqq\max_{x,y\in\stateset}\max_{h\in\normedgambles}\abs{\ltrans h(x)-\ltrans h(y)}
=\max_{h\in\normedgambles}\varnorm{\ltrans h},
\end{equation*}
where $\normedgambles\coloneqq\cset{h\in\gambles}{0\leq h\leq 1}$, and where for any $h\in\gambles$, its \emph{variation (semi)norm} is given by $\varnorm{h}\coloneqq\max h-\min h$.
If we define the following distance between two lower expectation operators $\lex$ and $\alex$ \citep{skulj2013}:
\begin{equation*}
d(\lex,\alex)=\max_{h\in\normedgambles}\abs{\lex(h)-\alex(h)},
\end{equation*}
then it is not difficult to see [using~\ref{it:lex:homo}, \ref{it:lex:more:bounds} and~\ref{it:lex:constant:additivity}] that $0\leq d(\lex,\alex)\leq1$, and that for any $f\in\gambles$:
\begin{equation}\label{eq:value:inequality}
\abs{\lex(f)-\alex(f)}\leq d(\lex,\alex)\varnorm{f}.
\end{equation}
\noindent
\citet{skulj2013} prove the following results, which will turn out to be crucial to our argument.

\begin{theorem}[\citep{skulj2013}]\label{thm:coe:properties}
Consider lower transition operators\/ $\latrans$ and\/ $\ltrans$, and two lower expectations $\ljoint[a]$ and $\ljoint[b]$ on $\gambles$.
Then the following statements hold:
\begin{enumerate}[label=\upshape(\roman*),ref=\upshape(\roman*),leftmargin=*]
\item\label{thm:coe:bounds} $0\leq\coe(\ltrans)\leq1$.
\item\label{thm:coe:coepowers} $\coe(\latrans\ltrans)\leq\coe(\latrans)\coe(\ltrans)$ and therefore $\coe(\ltrans^{n})\leq\coe(\ltrans)^{n}$ for all $n\in\nats$.
\item\label{thm:coe:dist} $d(\ljoint[a]\ltrans,\ljoint[b]\ltrans)\leq d(\ljoint[a],\ljoint[b])\coe(\ltrans)$.
\item\label{thm:coe:ergodicity} The lower transition operator\/ $\ltrans$ is {\pflike} if and only if there is some $r\in\nats$ such that $\coe(\ltrans^r)<1$.
\end{enumerate}
\end{theorem}
\noindent
Indeed, they allow us to derive useful bounds for the various terms on the right-hand side of Equation~\eqref{eq:finish}. 
For any non-negative real number $a$ we denote by $\floor{a}=\max\cset{n\in\natswithzero}{n\leq a}$ the largest natural number that it still dominates---its integer part.

\begin{lemma}\label{lem:basic:bounding}
Let\/ $\ltrans$ be a {\pflike} lower transition operator, with invariant lower expectation $\ljoint[\infty]$, and let $r$ be the smallest natural number such that $\rho\coloneqq\coe(\ltrans^r)<1$.
Let\/ $\ljoint[a]$ and $\ljoint[b]$ be any two lower expectations on $\gambles$.
Then for all $f\in\gambles$, $\ell_1,\ell_2\in\natswithzero$: 
\begin{equation}\label{eq:basic:bounding}
\abs[\big]{\ljoint[a](\ltrans^{\ell_1}f)-\ljoint[b](\ltrans^{\ell_2}f)}
\leq\varnorm{f}\rho^{\floor{\frac{\min\{\ell_1,\ell_2\}}{r}}}.
\end{equation}
As a consequence, for all $f\in\gambles$, $\ell,\ell_1,\ell_2\in\natswithzero$ and $k,k_1,k_2\in\nats$: 
\begin{align}
\abs[\big]{\ltrans^{\ell}f(\state{k})-\ljoint[\infty](f)}
&\leq\varnorm{f}\rho^{\floor{\frac{\ell}{r}}},
\label{eq:basic:bounding:einfty}\\
\abs[\big]{\ljoint[a](\ltrans^{\ell}f)-\ljoint[\infty](f)}
&\leq\varnorm{f}\rho^{\floor{\frac{\ell}{r}}},
\label{eq:basic:bounding:eleft}\\
\abs[\big]{\ltrans^{\ell}f(\state{k})-\ljoint[b](\ltrans^{\ell}f)}
&\leq\varnorm{f}\rho^{\floor{\frac{\ell}{r}}},
\label{eq:basic:bounding:eright}\\
\abs[\big]{\ltrans^{\ell_1}f(\state{k_1})-\ltrans^{\ell_2}f(\state{k_2})}
&\leq\varnorm{f}\rho^{\floor{\frac{\min\{\ell_1,\ell_2\}}{r}}}.
\label{eq:basic:bounding:diff}
\end{align}
\end{lemma}

\begin{proof}
We may assume without loss of generality that $\ell_1\leq\ell_2$.
Using Equation~\eqref{eq:value:inequality}, Theorem~\ref{thm:coe:properties}\ref{thm:coe:dist} and the fact that we can consider $\ltrans^{\ell_1}$ as a lower transition operator in its own right:
\begin{equation*}
\abs[\big]{\ljoint[a](\ltrans^{\ell_1}f)-\ljoint[b](\ltrans^{\ell_2}f)} \leq d(\ljoint[a]\ltrans^{\ell_1},\ljoint[b]\ltrans^{\ell_2})\varnorm{f}
\leq d(\ljoint[a],\ljoint[b]\ltrans^{\ell_2-\ell_1})\coe(\ltrans^{\ell_1})\varnorm{f}.
\end{equation*}
Our proof of the first inequality~\eqref{eq:basic:bounding} is complete if we realise that $0\leq d(\ljoint[a],\ljoint[b]\ltrans^{\ell_2-\ell_1})\leq1$, and that $\coe(\ltrans^{\ell_1})\leq\coe(\ltrans^{r\floor{\frac{\ell_1}{r}}})\leq\coe(\ltrans^r)^{\floor{\frac{\ell_1}{r}}}$ by Theorem~\ref{thm:coe:properties}\ref{thm:coe:bounds}\&\ref{thm:coe:coepowers}.

Denote, for any $x\in\stateset$, by $\joint[x]$ the expectation operator that assigns all probability mass to $x$, meaning that $\joint[x](f)\coloneqq f(x)$ for all $f\in\gambles$.
To prove the second inequality~\eqref{eq:basic:bounding:einfty}, consider any $x\in\stateset$ and let $\ljoint[a]=\joint[x]$, $\ljoint[b]=\ljoint[\infty]$ and $\ell_1=\ell_2=\ell$, then we infer from~\eqref{eq:basic:bounding} that indeed:
\begin{equation*}
\abs[\big]{\ltrans^{\ell}f(x)-\ljoint[\infty](f)}
=\abs[\big]{\joint[x](\ltrans^{\ell}f)-\ljoint[\infty](\ltrans^{\ell}f)}
\leq\varnorm{f}\rho^{\floor{\frac{\ell}{r}}}.
\end{equation*}
To prove the third inequality~\eqref{eq:basic:bounding:eleft}, let $\ljoint[b]=\ljoint[\infty]$ and $\ell_1=\ell_2=\ell$, then we infer from~\eqref{eq:basic:bounding} that indeed:
\begin{equation*}
\abs[\big]{\ljoint[a](\ltrans^{\ell}f)-\ljoint[\infty](f)}
=\abs[\big]{\ljoint[a](\ltrans^{\ell}f)-\ljoint[\infty](\ltrans^{\ell}f)}
\leq\varnorm{f}\rho^{\floor{\frac{\ell}{r}}},
\end{equation*}
where we used that $\ljoint[\infty](f)=\ljoint[\infty](\ltrans^{\ell}f)$ for all $\ell\in\natswithzero$; see Proposition~\ref{prop:perron:frobenius}.

To prove the fourth inequality~\eqref{eq:basic:bounding:eright}, consider any $x\in\stateset$ and let $\ljoint[a]=\joint[x]$ and $\ell_1=\ell_2=\ell$, then we infer from~\eqref{eq:basic:bounding} that indeed:
\begin{equation*}
\abs[\big]{\ltrans^{\ell}f(x)-\ljoint[b](\ltrans^{\ell}f)}
=\abs[\big]{\joint[x](\ltrans^{\ell}f)-\ljoint[b](\ltrans^{\ell}f)}
\leq\varnorm{f}\rho^{\floor{\frac{\ell}{r}}}.
\end{equation*}
To prove the fifth inequality~\eqref{eq:basic:bounding:diff}, consider any $x,y\in\stateset$ and let $\ljoint[a]=\joint[x]$ and $\ljoint[b]=\joint[y]$.
Then we infer from~\eqref{eq:basic:bounding} that indeed:
\begin{equation*}
\abs[\big]{\ltrans^{\ell_1}f(x)-\ltrans^{\ell_2}f(y)}
=\abs[\big]{\joint[x](\ltrans^{\ell_1}f)-\joint[y](\ltrans^{\ell_2}f)}
\leq\varnorm{f}\rho^{\floor{\frac{\min\{\ell_1,\ell_2\}}{r}}}.\qedhere
\end{equation*}
\end{proof}

\begin{lemma}\label{lem:advanced:bounding}
Consider an imprecise Markov chain with initial---or marginal---model $\ljoint[1]$ and lower transition operator\/ $\ltrans$.
Assume that\/ $\ltrans$ is {\pflike}, with invariant lower expectation $\ljoint[\infty]$, and let $r$ be the smallest natural number such that $\rho\coloneqq\coe(\ltrans^r)<1$.
Then the following statements hold for all $f\in\gambles$:
\begin{enumerate}[label=\upshape(\roman*),leftmargin=*]
\item\label{it:advanced:bounding:avgain} $\abs{\avgain{\ltrans^{\ell}f}(\stateto{n})}\leq\varnorm{f}\rho^{\floor{\frac{\ell}{r}}}$ for all $\ell\in\natswithzero$ and $n\in\nats$.
\item\label{it:advanced:bounding:fixed:trans} $\lim_{n\to\infty}\frac{1}{n}\sum_{k=1}^{n}\ltrans^{n}f(\state{k})=\ljoint[\infty](f)$.
\item\label{it:advanced:bounding:avgain:fixed:state} $\lim_{n\to\infty}\frac{1}{n}\sum_{\ell=1}^{n}\ltrans^{\ell}f(\state{n})=\ljoint[\infty](f)$.
\item\label{it:advanced:bounding:expectations} $\lim_{n\to\infty}\frac{1}{n}\sum_{k=1}^{n}\ljoint[k](f)=\ljoint[\infty](f)$.
\end{enumerate} 
\end{lemma}

\begin{proof}
Recall from Equation~\eqref{eq:average:gain} that:
\begin{equation*}
n\avgain{\ltrans^{\ell}f}(\stateto{n})
=\sqgroup[\big]{\ltrans^{\ell}f(\state{1})-\ljoint[1](\ltrans^{\ell}f)}
+\sum_{k=2}^{n}\sqgroup[\big]{\ltrans^{\ell}f(\state{k})-\ltrans^{\ell+1}f(\state{k-1})}.
\end{equation*}
If we also invoke Lemma~\ref{lem:basic:bounding}, we find that:
\begin{align*}
n\abs[\big]{\avgain{\ltrans^{\ell}f}(\stateto{n})}
&\leq\abs[\big]{\ltrans^{\ell}f(\state{1})-\ljoint[1](\ltrans^{\ell}f)} 
+\sum_{k=2}^{n}\abs[\big]{\ltrans^{\ell}f(\state{k})-\ltrans^{\ell+1}f(\state{k-1})}\\
&\leq\varnorm{f}\rho^{\floor{\frac{\ell}{r}}}
+\sum_{k=2}^{n}\varnorm{f}\rho^{\floor{\frac{\ell}{r}}}
=n\varnorm{f}\rho^{\floor{\frac{\ell}{r}}},
\end{align*}
which proves statement~\ref{it:advanced:bounding:avgain}.
Similarly, by Lemma~\ref{lem:basic:bounding}:
\begin{equation*}
\abs[\bigg]{\frac{1}{n}\sum_{k=1}^{n}\sqgroup[\big]{\ltrans^{n}f(\state{k})-\ljoint[\infty](f)}}
\leq\frac{1}{n}\sum_{k=1}^{n}\abs[\big]{\ltrans^{n}f(\state{k})-\ljoint[\infty](f)}
\leq\frac{1}{n}\sum_{k=1}^{n}\varnorm{f}\rho^{\floor{\frac{n}{r}}}
=\varnorm{f}\rho^{\floor{\frac{n}{r}}},
\end{equation*}
which proves statement~\ref{it:advanced:bounding:fixed:trans}.
Similarly, again by Lemma~\ref{lem:basic:bounding}:
\begin{align*}
\abs[\bigg]{\frac{1}{n}\sum_{\ell=1}^{n}\sqgroup[\big]{\ltrans^{\ell}f(\state{n})-\ljoint[\infty](f)}}
&\leq\frac{1}{n}\sum_{\ell=1}^{n}\abs[\big]{\ltrans^{\ell}f(\state{n})-\ljoint[\infty](f)}\leq\frac{1}{n}\sum_{\ell=1}^{n}\varnorm{f}\rho^{\floor{\frac{\ell}{r}}}\\
&\leq\frac{\varnorm{f}}{n}\sum_{\ell=0}^{\infty}\rho^{\floor{\frac{\ell}{r}}}=\frac{\varnorm{f}}{n}r\sum_{s=0}^{\infty}\rho^{s}=\frac{\varnorm{f}}{n}\frac{r}{1-\rho},
\end{align*}
which proves statement~\ref{it:advanced:bounding:avgain:fixed:state}.
Finally, by Lemma~\ref{lem:basic:bounding} and an argumentation similar to our proof for statement~\ref{it:advanced:bounding:avgain:fixed:state}:
\begin{align*}
\abs[\bigg]{\frac{1}{n}\sum_{k=1}^{n}\ljoint[k](f)-\ljoint[\infty](f)}
\leq\frac{1}{n}\sum_{k=1}^{n}\abs[\big]{\ljoint[k](f)-\ljoint[\infty](f)}
&=\frac{1}{n}\sum_{k=1}^{n}\abs[\big]{\ljoint[1](\ltrans^{k-1}f)-\ljoint[\infty](f)}\\
&\leq\frac{1}{n}\sum_{k=1}^{n}\varnorm{f}\rho^{\floor{\frac{k-1}{r}}}
\leq\frac{\varnorm{f}}{n}\frac{r}{1-\rho},
\end{align*}
which proves statement~\ref{it:advanced:bounding:expectations}.
\end{proof}

We can now prove our main result.

\begin{theorem}[Point-wise ergodic theorem]\label{thm:point-wise:ergodic:theorem}
Consider an imprecise Markov chain with initial---or marginal---model $\ljoint[1]$ and lower transition operator\/ $\ltrans$.
Assume that\/ $\ltrans$ is {\pflike}, with invariant lower expectation $\ljoint[\infty]$.
Then for all $f\in\gambles$:
\begin{equation*}
\liminf\ergodic{f}\geq0
\text{ strictly almost surely},
\end{equation*}
and consequently,
\begin{equation*}
\liminf_{n\to\infty}\frac{1}{n}\sum_{k=1}^{n}f(\state{k})\geq\ljoint[\infty](f)
\text{ strictly almost surely}.
\end{equation*}
\end{theorem}

\begin{proof}
We begin with the first inequality.
Let $r$ be the smallest natural number such that $\rho\coloneqq\coe(\ltrans^r)<1$.
Consider any $q\in\nats$, and let $g_q\coloneqq\sum_{\ell=0}^{rq-1}\ltrans^\ell f$, then it follows from Equation~\eqref{eq:gain} and~\ref{it:lex:super} that for all $n\in\nats$:
\begin{equation*}
\gain{g_q}(\stateto{n})\leq\sum_{\ell=0}^{rq-1}\gain{\ltrans^\ell f}(\stateto{n})
\text{ and therefore }
\avgain{g_q}(\stateto{n})\leq\sum_{\ell=0}^{rq-1}\avgain{\ltrans^\ell f}(\stateto{n}).
\end{equation*}
Hence, if we also take into account Equation~\eqref{eq:finish} and Lemma~\ref{lem:advanced:bounding}, we find that: 
\begin{align}
\liminf\ergodic{f}
&=\liminf_{n\to\infty}\sum_{\ell=0}^{n-1}\avgain{\ltrans^{\ell}f}(\stateto{n})\notag\\
&\geq\liminf_{n\to\infty}\sum_{\ell=0}^{rq-1}\avgain{\ltrans^{\ell}f}(\stateto{n})
+\liminf_{n\to\infty}\sum_{\ell=rq}^{n-1}\avgain{\ltrans^{\ell}f}(\stateto{n})\notag\\
&\geq\liminf\avgain{g_q}
-\varnorm{f}\limsup_{n\to\infty}\sum_{\ell=rq}^{n-1}\rho^{\floor{\frac{\ell}{r}}}\notag\\
&=\liminf\avgain{g_q}
-\varnorm{f}\sum_{\ell=rq}^{\infty}\rho^{\floor{\frac{\ell}{r}}}
\geq\liminf\avgain{g_q}
-\varnorm{f}r\frac{\rho^q}{1-\rho}.\label{eq:handigeformule}
\end{align}
By combining Equation~\eqref{eq:gain} with the coherence [\ref{it:lex:more:bounds} and \ref{it:lex:constant:additivity}] of the local models of the Markov chain, we see that $\gain{g_q}$ is a submartingale for which $\Delta\gain{g_q}$ is uniformly bounded. 
It therefore follows from our strong law of large numbers for submartingale differences [Corollary~\ref{cor:slln:submartingale:differences}] that $\liminf\avgain{g_q}\geq0$ strictly almost surely, meaning that there is some test supermartingale $\test[(q)]$ that converges to $+\infty$ on any path $\pth$ for which $\liminf\avgain{g_q}<0$. 
Furthermore, by the argumentation in the proof of Corollary~\ref{cor:slln:submartingale:differences}, we also know that $0\leq\test[(q)](\xvalto{n})\leq(\frac{3}{2})^{n}$ for all $n\in\nats$ and $\xvalto{n}\in\statesto{n}$. 
If we now invoke Equation~\eqref{eq:handigeformule}, we see that $\test[(q)]$  converges to $+\infty$ on any path $\pth$ where $\liminf_{n\to\infty}\ergodic{f}(\pth)<-\varnorm{f}\frac{r\rho^q}{1-\rho}$.

Now consider any sequence of positive real numbers $w^{(q)}$ such that $\sum_{q\in\nats}w^{(q)}=1$, then it follows from the considerations above that the sequence of non-negative real numbers $a_i(\xvalto{n})\coloneqq\sum_{q=1}^{i}w^{(q)}\test[(q)](\xvalto{n})$, $i\in\nats$ is non-decreasing and bounded above by $(\frac{3}{2})^{n}$, and therefore converges to a non-negative real number, for all $n\in\nats$ and $\xvalto{n}\in\statesto{n}$.
Hence, we can define the real process $\test\coloneqq\sum_{q\in\nats}w^{(q)}\test[(q)]$, which clearly converges to $+\infty$ on any path $\pth$ where $\liminf_{n\to\infty}\ergodic{f}(\pth)<0$. 
Moreover, $\test(\init)=1$ and $\test$ is non-negative.
So we are done with the first inequality if we can prove that $\test$ is a supermartingale.
Consider, therefore, any situation $s$ and any $\sitinlocal{\cdot}{s}\in\solp(\sitinllocal{\cdot}{s})$, then, if we denote its (probability) mass function by $p(\cdot\vert s)$:
\begin{align*}
\sitinlocal{\Delta\test}{s}
&=\sum_{x\in\stateset}p(x\vert s)\Delta\test(s)(x)
=\sum_{x\in\stateset}p(x\vert s)\sum_{q\in\nats}w^{(q)}\Delta\test[(q)](s)(x)\\
&=\sum_{q\in\nats}w^{(q)}\sum_{x\in\stateset}p(x\vert s)\Delta\test[(q)](s)(x)
=\sum_{q\in\nats}w^{(q)}\sitinlocal{\Delta\test[(q)](s)}{s}
\leq0,
\end{align*}
where the inequality follows from $\sitinlocal{\Delta\test[(q)](s)}{s}\leq\sitinulocal{\Delta\test[(q)](s)}{s}\leq0$; see Equation~\eqref{eq:linear:dominates}.
If we now recall Equation~\eqref{eq:lower:envelope}, we see that indeed $\sitinulocal{\Delta\test(s)}{s}\leq0$.

The second inequality is equivalent with the first by Lemma~\ref{lem:advanced:bounding}\ref{it:advanced:bounding:expectations}.
\end{proof}

We can fairly easily extend this result to gambles that depend on a finite number of states.

\begin{corollary}\label{cor:point-wise:ergodic:theorem}
Consider an imprecise Markov chain with initial---or marginal---model $\ljoint[1]$ and lower transition operator\/ $\ltrans$.
Assume that\/ $\ltrans$ is {\pflike}, with invariant lower expectation $\ljoint[\infty]$.
Then for all $f\in\gamblessymbol(\stateset^r)$, with $r\in\nats$:
\begin{equation}\label{eq:point-wise:ergodic:theorem:finite}
\liminf_{n\to\infty}\frac{1}{n}\sum_{k=1}^{n}f(\statefromto{k}{k+r-1})\geq\statljoint(f(\statefromto{1}{r}))
\text{ strictly almost surely}.
\end{equation}
\end{corollary}

\begin{proof}
We give a proof by induction. 
We know from Theorem~\ref{thm:point-wise:ergodic:theorem} that Equation~\eqref{eq:point-wise:ergodic:theorem:finite} holds for $r=1$.
Now consider any $q\in\nats$, and assume as our induction hypothesis that Equation~\eqref{eq:point-wise:ergodic:theorem:finite} holds for $r=q$, then we prove that it also holds for $r=q+1$.

Consider any $f\in\gamblessymbol(\stateset^{q+1})$, and define the real process $\submartin$ by letting $\submartin(\stateto{\ell})\coloneqq0$ for $\ell=0,1,\dots,q$ and
\begin{equation*}
\submartin(\stateto{q+n})
\coloneqq\sum_{k=0}^{n-1}
\sqgroup[\big]{f(\statefromto{k+1}{k+q+1})-\statetoinlglobal{f(\statefromto{k+1}{k+q+1})}{k+q}}
\text{ for all $n\in\nats$}.
\end{equation*}
Then for any $n\in\natswithzero$ and any situation $\xvalto{q+n}\in\statesto{q+n}$, we find that
\begin{align*}
\Delta\submartin(\xvalto{q+n})(\state{q+n+1})
&=\submartin(\xvalto{q+n}\state{q+n+1})-\submartin(\xvalto{q+n})\\
&=f(\xvalfromto{n+1}{n+q}\state{q+n+1})
-\xtoinlglobal{f(\statefromto{n+1}{q+n+1})}{n+q}\\
&=f(\xvalfromto{n+1}{n+q}\state{q+n+1})
-\xtoinllocal{f(\xvalfromto{n+1}{n+q}\state{q+n+1})}{n+q},
\end{align*}
where the last equality follows from Corollary~\ref{cor:global:and:local:gambles}, and therefore, coherence [\ref{it:lex:constant:additivity}] implies that $\xtoinllocal{\Delta\submartin(\xvalto{q+n})}{{q+n}}=0$. Similary, for any $\ell\in\{0,1,\dots,q-1\}$ and any $\xvalto{\ell}\in\statesto{\ell}$, coherence [\ref{it:lex:more:bounds}] implies that $\xtoinllocal{\Delta\submartin(\xvalto{\ell})}{{\ell}}=\xtoinllocal{0}{{\ell}}=0$. Hence, we conclude that $\submartin$ is a submartingale, whose differences are uniformly bounded [because $f$ is, trivially so].
Corollary~\ref{cor:slln:submartingale:differences} then tells us that $\liminf\avg{\submartin}\geq0$ strictly almost surely, or in other words that there is a test supermartingale $\test$ that converges to $+\infty$ on the event $A\coloneqq\cset{\pth\in\pths}{\liminf\avg{\submartin}(\pth)<0}$.

Now observe that---keeping in mind that the second terms always lie between $\min f$ and $\max f$, due to \ref{it:lex:more:bounds}:
\begin{align}
\liminf_{n\to\infty}\frac{1}{n}\sum_{k=1}^{n}f(\statefromto{k}{k+q})
&=\liminf_{n\to\infty}
\sqgroup[\bigg]{\avg{\submartin}(\stateto{n+q})
+\frac{1}{n}\sum_{k=0}^{n-1}\statetoinlglobal{f(\statefromto{k+1}{k+q+1})}{k+q}}
\notag\\
&\geq\liminf_{n\to\infty}\avg{\submartin}(\stateto{n+q})
+\liminf_{n\to\infty}\frac{1}{n}\sum_{k=0}^{n-1}\statetoinlglobal{f(\statefromto{k+1}{k+q+1})}{k+q}.
\label{eq:point-wise:ergodic:theorem:intermediate:first}
\end{align}
If we consider the gamble 
$g(\stateto{q})\coloneqq\statljoint(f(\statefromto{1}{q+1})\vert\stateto{q})$ 
that only depends on the first $q$ states, then it follows from Corollary~\ref{cor:global:and:local:gambles} and the Markov condition~\eqref{eq:markov:condition} that for all $\xvalto{q}\in\statesto{q}$
\begin{equation}
g(\xvalto{q})
=\statljoint(f(\statefromto{1}{q+1})\vert\xvalto{q})
=\xtoinllocal{f(\xvalfromto{1}{q}\state{q+1})}{q}
=\xinllocal{f(\xvalfromto{1}{q}\state{q+1})}{q}.
\label{eq:point-wise:ergodic:theorem:intermediate:second}
\end{equation}
Similarly, it follows from Corollary~\ref{cor:global:and:local:gambles}, the Markov condition~\eqref{eq:markov:condition} and Equation~\eqref{eq:point-wise:ergodic:theorem:intermediate:second} that for all $k\in\natswithzero$ and all $\xvalto{k+q}\in\statesto{k+q}$
\begin{align*}
\xtoinlglobal{f(\statefromto{k+1}{k+q+1})}{k+q}
&=\xtoinllocal{f(\xvalfromto{k+1}{k+q}\state{k+q+1})}{k+q}\\
&=\xinllocal{f(\xvalfromto{k+1}{k+q}\state{k+q+1})}{k+q}
=g(\xvalfromto{k+1}{k+q}),
\end{align*}
and therefore the inequality~\eqref{eq:point-wise:ergodic:theorem:intermediate:first} can be rewritten as
\begin{equation}\label{eq:point-wise:ergodic:theorem:intermediate:third}
\liminf_{n\to\infty}\frac{1}{n}\sum_{k=1}^{n}f(\statefromto{k}{k+q})
\geq\liminf_{n\to\infty}\avg{\submartin}(\stateto{n+q})
+\liminf_{n\to\infty}\frac{1}{n}\sum_{k=1}^{n}g(\statefromto{k}{k+q-1}).
\end{equation}
We infer from the induction hypothesis that for the second term on the right-hand side
\begin{equation*}
\liminf_{n\to\infty}\frac{1}{n}\sum_{k=1}^{n}g(\statefromto{k}{k+q-1})
\geq\statljoint(g(\statefromto{1}{q}))
\text{ strictly almost surely},
\end{equation*}
meaning that there is some test supermartingale $\test[*]$ that converges to $+\infty$ on the set $B$ of all paths where this inequality does not hold.
This in turn implies that 
\begin{equation*}
\liminf_{n\to\infty}\frac{1}{n}\sum_{k=1}^{n}f(\statefromto{k}{k+q})
\geq\statljoint(g(\statefromto{1}{q}))
\text{ strictly almost surely},
\end{equation*}
because it follows from~\eqref{eq:point-wise:ergodic:theorem:intermediate:third} that the paths where this inequality does not hold must belong to $A\cup B$, where the test supermartingale $\frac{1}{2}(\test+\test[*])$ converges to $+\infty$.
Now observe that
\begin{equation*}
\statljoint(g(\statefromto{1}{q}))
=\statljoint(\statljoint(f(\statefromto{1}{q+1})\vert\stateto{q}))
=\statljoint(f(\statefromto{1}{q+1})),
\end{equation*}
by Theorem~\ref{thm:law:of:iterated:expectations:general} [with $m\coloneqq0$ and $n\coloneqq q$].
\end{proof}

\section{Conclusions and discussion}\label{sec:conclusion}

We have motivated expressions for joint lower and upper expectations on extended real-valued variables for imprecise Markov chains (with finite state spaces), and proved various interesting properties for them.
This has allowed us to deal quite elegantly with transition and return times, but we expect our approach to be equally useful in other problems involving unbounded and/or extended real-valued variables.

We have also proved versions of the point-wise ergodic theorem for our imprecise Markov chains, involving (bounded) functions of a finite number of states. 
It is a subject of current research whether this result can be extended to gambles that depend on the entire state trajectory, and not just on a finite number of states.

Our version in Theorem~\ref{thm:point-wise:ergodic:theorem} subsumes the one for (precise) Markov chains discussed in the Introduction, because there $\ljoint[\infty](f)=\ujoint[\infty](f)=\joint[\infty](f)$ and therefore
\begin{multline*}
\joint[\infty](f)
=\ujoint[\infty](f)
\geq\limsup_{n\to\infty}\frac{1}{n}\sum_{k=1}^{n}f(\state{k})
\geq\liminf_{n\to\infty}\frac{1}{n}\sum_{k=1}^{n}f(\state{k})
\geq\ljoint[\infty](f)
=\joint[\infty](f)\\
\text{ strictly almost surely,}
\end{multline*} 
implying that $\frac{1}{n}\sum_{k=1}^{n}f(\state{k})$ converges to $\joint[\infty](f)$ (strictly) almost surely.
In our more general case, however, we cannot generally prove that there is almost sure \emph{convergence}, and we retain only almost sure inequalities involving limits inferior and superior, as is also the case for our strong law of large numbers for submartingale differences.
Indeed, that such convergence should not really be expected for imprecise probability models was already argued by \citet{walley1982a}.

Ergodicity results for Markov chains are quite relevant for applications in queuing theory, where they are for instance used to prove Little's law~\cite{whitt1991}, or ASTA (Arrivals See Time Averages) properties~\cite{makowski1989}.
We believe the discussion in this paper could be instrumental in deriving similar properties for queues where the probability models for arrivals and departures are imprecise.

\section*{Acknowledgements}
Research by Gert de Cooman and Stavros Lopatatzidis was funded through project number G012512N of the Research Foundation -- Flanders (FWO). 
Jasper De Bock is a Postdoctoral Fellow of the FWO and wishes to acknowledge its financial support. 
The authors would like to express their gratitude to three anonymous referees for their comments on a conference version of this paper, and to Tom Ward and Volodya Vovk for taking the time to discuss some of the ideas behind this paper.

\end{document}